\documentclass{amsart}

\usepackage{amsmath,amsthm,amssymb,amsfonts}
\usepackage{bbm}
\usepackage[mathscr]{eucal}
\usepackage[english]{babel}
\usepackage{graphicx}

\allowdisplaybreaks[1]

\newtheorem{thm}{Theorem}[section]
\newtheorem{cor}[thm]{Corollary}
\newtheorem{lem}[thm]{Lemma}
\newtheorem{prop}[thm]{Proposition}
\newtheorem{rem}[thm]{Remark}

\newtheorem{defn}[thm]{Definition}

\theoremstyle{remark}

\numberwithin{equation}{section}

\newcommand{\thmref}[1]{Theorem~\ref{#1}}
\newcommand{\lemref}[1]{Lemma \ref{#1}}
\newcommand{\propref}[1]{Proposition \ref{#1}}

\newcommand{\remref}[1]{Remark \ref{#1}}

\def\BB{{\mathbb B}}
\def\CC{{\mathbb C}}
\def\NN{{\mathbb N}}
\def\RR{{\mathbb R}}
\def\SS{{{\mathbb S}^{d-1}}}
\def\SSS{{\mathbb S}}

\def\bA{{\bf A}}
\def\bB{{\bf B}}

\def\bD{{\bf D}}

\def\bF{{\bf F}}
\def\bG{{\bf G}}
\def\bH{{\bf H}}

\def\cB{\mathcal{B}}

\def\cF{\mathcal{F}}

\def\cH{\mathcal{H}}

\def\cJ{\mathcal{J}}
\def\cK{\mathcal{K}}

\def\cP{\mathcal{P}}
\def\cQ{\mathcal{Q}}
\def\cR{\mathcal{R}}
\def\cS{\mathcal{S}}

\def\cX{\mathcal{X}}
\def\cY{\mathcal{Y}}
\def\cZ{\mathcal{Z}}

\def\fb{{\mathfrak b}}
\def\fB{{\mathfrak B}}
\def\fD{{\mathfrak D}}
\def\fL{{\mathfrak L}}

\def\nA{\mathscr{D}}
\def\sF{\mathscr{F}}

\def\HHH{\mathscr{H}}

\newcommand{\PP}{Z}
\def\eps{{\varepsilon}}
\def\u*{*}
\def\scal{N}

\def\bbb{b}
\def\kk{k}
\def\rrr{r}
\def\supp{\operatorname{supp}}
\def\II{{I}}

\def\QQ{{\cQ}}
\def\NNN{\mathscr{N}}
\def\ONE{{\mathbbm 1}}
\def\aa{\varphi}
\def\LL{\Psi}
\def\YY{\mathscr Y}
\def\MM{\mathscr{M}}

\def\ff{{\mathfrak f}}
\def\bb{{\mathfrak b}}
\def\tQ{{\tilde{\mathcal{Q}}}}
\def\gam{\beta}
\def\tc{{\tilde c}}

\def\newPsi{\cK}
\def\newPhi{\Lambda}

\def\nA{A}
\def\nB{B}
\def\ww{\widetilde{w}}

\newcommand{\til}[1]{\breve#1}

\begin{document}

\title[Nonlinear Approximation of Harmonic Functions]
{Nonlinear \boldmath $n$-term Approximation of Harmonic Functions from Shifts of
the Newtonian Kernel}

\author[K. G. Ivanov]{Kamen G. Ivanov}
\address{Institute of Mathematics and Informatics\\
Bulgarian Academy of Sciences\\
Sofia, Bulgaria}
\email{kamen@math.bas.bg}

\author[P. Petrushev]{Pencho Petrushev}
\address{Department of Mathematics\\
University of South Carolina\\
Columbia, SC}
\email{pencho@math.sc.edu}

\thanks{The first author has been supported by Grant DN 02/14
of the Fund for Scientific Research of the Bulgarian Ministry of Education and Science.
The second author has been supported by NSF Grant DMS-1714369.}

\begin{abstract}
A basic building block in Classical Potential Theory is the fundamental solution
of the Laplace equation in $\RR^d$ (Newtonian kernel).
The main goal of this article is to study the rates of nonlinear $n$-term approximation
of harmonic functions on the unit ball $B^d$ from shifts of the Newtonian kernel
with poles outside $\overline{B^d}$ in the harmonic Hardy spaces.
Optimal rates of approximation are obtained in terms of harmonic Besov spaces.
The main vehicle in establishing these results is the construction of highly localized frames for
Besov and Triebel-Lizorkin spaces on the sphere whose elements are linear combinations of a fixed number of shifts
of the Newtonian kernel.
\end{abstract}


\maketitle

\tableofcontents

\section{Introduction}\label{s1}

The fundamental solution of the Laplace equation
$\frac{1}{|x|^{d-2}}$ in dimension $d>2$ or
$\ln \frac{1}{|x|}$ if $d=2$
with $|x|$ being the Euclidean norm of $x\in \RR^d$
is a basic building block in Potential theory.
As is customary, we shall term the harmonic function
$\frac{1}{|x|^{d-2}}$ or $\ln \frac{1}{|x|}$
``Newtonian kernel''.

The main purpose of this article is to study the nonlinear $n$-term approximation
of harmonic functions on the unit ball $B^d$ in $\RR^d$ from linear combinations
of shifts of the Newtonian kernel.
More explicitly, the problem is
for a given harmonic function $U$ on $B^d$ and $n\ge 1$ to find
$n$ locations $\{y_j\}$ in $\RR^d\setminus \overline{B^d}$
and coefficients $\{a_j\}$ in $\CC$ so that
\begin{equation}\label{lin-comb}
a_0+\sum_{j=1}^n \frac{a_j}{|x-y_j|^{d-2}}\quad\hbox{if} \;\; d>2
\quad\hbox{or}\quad
a_0+\sum_{j=1}^n a_j\ln \frac{1}{|x-y_j|}\quad\hbox{if} \;\; d=2
\end{equation}
approximates $U$ with an optimal rate (near best) in the harmonic Hardy space $\HHH^p(B^d)$, $0<p<\infty$.
Denote by $\NNN_n$ the set of all harmonic functions on $B^d$ that can be represented in the form \eqref{lin-comb}.
Here the points $\{y_j\}$ are allowed to vary with the function and hence $\NNN_n$ is nonlinear.
Given $U\in \HHH^p(B^d)$ we denote
\begin{equation}\label{def-En}
E_n(U)_{\HHH^p}:=\inf_{G\in\NNN_n}\|U-G\|_{\HHH^p(B^d)}.
\end{equation}
We shall term $E_n(U)_{\HHH^p}$ the best $n$-term approximation of $U$ in the harmonic Hardy space $\HHH^p(B^d)$
from shifts of the Newtonian kernel as in \eqref{lin-comb}.
Our goal is to study the rate of convergence of $\{E_n(U)_{\HHH^p}\}$
and the smoothness spaces that govern this approximation process.
The same approximation problem is also important in the case when the function $U$ to be approximated is harmonic on
$\RR^d\setminus \overline{B^d}$ and the poles $\{y_j\}$ are in $B^d$ or
$U$ is harmonic on
$\RR^d_+$ and the poles $\{y_j\}$ are in $\RR^d_-$.

The results of A. Pekarski \cite{Pek1, Pek2} on rational approximation of holomorphic functions on the unit disc in $\CC$
and also the results in \cite{KP2} served as an inspiration and motivation for the development in this article.
An important motivation to us also comes from some applications of Potential theory.
In Geodesy people consider approximation of the gravitational (disturbing) potential using the potential of $n$ point masses.
A given potential $U$ is approximated
by the potential of $n$ point charges in Electrostatics
or by the potential of $n$ magnetic poles in Magnetism.
There is also a great deal of work done on the Method of Fundamental Solutions
for the Dirichlet problem of the Laplace equation in Numerical Analysis.
This research is directly related to the problems we consider here.
The multipole method of V.~Rokhlin and his collaborators (e.g. \cite{CGR, GR}) is also relevant to our undertaking.

We refer the reader to \cite{AG, Helms, K} for the basics of Potential theory.

The focus of this article is on the establishment of a direct (Jackson type) estimate
for nonlinear $n$-term approximation of functions in the harmonic Hardy space $\HHH^p(B^d)$, $0<p<\infty$,
from shifts of the Newtonian kernel.
As one can expect the harmonic Besov spaces on the ball
\begin{equation*}
B^{s\tau}_\tau(\HHH) \quad\hbox{with}\quad 1/\tau=s/(d-1)+1/p,~s>0,
\end{equation*}
will be naturally involved in the approximation process.

The poor localization of the Newtonian kernel is the first obstacle to overcome
in approximating from linear combinations of its shifts.
An important step forward in solving this approximation problem is the construction in \cite{IP2}
of highly localized summability kernels on the unit sphere $\SS$ in $\RR^d$ that are restrictions
to the sphere of linear combinations of a fixed number of
shifts of the Newtonian kernel just as in \eqref{lin-comb}.
Note that the harmonic functions by their nature cannot be well localized in an open subset of $\RR^d$,
but they can be well localized on the boundary of such a~set; typical examples are  $\SS$ and $\RR^{d-1}$.

To obtain our approximation result we proceed as follows:
We first use the result from \cite{IP2} to construct a pair of dual frames $\{\theta_\xi\}$, $\{\tilde\theta_\xi\}$
for all spaces of interest on $\SS$ whose elements $\{\theta_\xi\}$  are linear combinations of
a fixed number of shifts of the Newtonian kernel and are well localized.
Armed with these frames we apply an intermediate nonlinear $n$-term approximation from $\{\theta_\xi\}$
to the boundary value function/distribution $f_U$ of the harmonic function $U$ to be approximated.
This leads us to the desired estimate by harmonic extension to $B^d$ of the approxiamant and using the fact that each $\theta_\xi$ is
a finite linear combination of shifts of the Newtonian kernel.

Thus a major step in our development is to construct such a pair of dual frames.
More precisely, one of our main goals is to construct (see \thmref{thm:prop-frame}) a pair of frames
$\{\theta_\xi\}$, $\{\tilde\theta_\xi\}$
for the Besov and Triebel-Lizorkin spaces $\cB^{sq}_p(\SS)$ and $\cF^{sq}_p(\SS)$
with parameters $(s, p, q)$ in the range
\begin{equation}\label{indices-1}
\QQ=\QQ(A):=\big\{(s, p, q): |s| \le A,\; A^{-1}\le p \le A, \;\hbox{and}\; A^{-1}\le q<\infty\big\},
\end{equation}
where $A >1$ is a fixed constant.
This construction employs the small perturbation method for construction of frames developed in \cite{DKKP}
and relies on the kernels from \cite{IP2}.
While the basic ideas behind the construction of the frames $\{\theta_\xi\}$, $\{\tilde\theta_\xi\}$ is relatively simple,
some of the details become technical when applied to the specific case of this article.
For example, the requirement that $\{\theta_\xi\}$, $\{\tilde\theta_\xi\}$ are frames for the class
of Besov and Triebel-Lizorkin spaces $\cB^{sq}_p(\SS)$ and $\cF^{sq}_p(\SS)$
with parameters $(s, p, q)\in \QQ(A)$ compels us to carefully trace the constants
appearing in all relevant estimates.

The next several remarks will perhaps clarify some of the issues arising in our construction of
the pair of frames $\{\theta_\xi\}$, $\{\tilde\theta_\xi\}$ described above:

(1)
In applying the small perturbation method from \cite{DKKP} we use as a~backbone a~frame $\{\psi_\xi\}$ on $\SS$ from \cite{NPW2},
which can characterize the Besov and Triebel-Lizorkin spaces $\cB^{sq}_p(\SS)$ and $\cF^{sq}_p(\SS)$
with complete range of parameters $(s, p, q)$, i.e. $s\in\RR$, $0<p,q<\infty$.
With the restriction that each frame element $\{\theta_\xi\}$ is a linear combination of a fixed number of shifts
of the Newtonian kernel comes the natural limitation that the new frames $\{\theta_\xi\}$, $\{\tilde\theta_\xi\}$
can characterize the Besov and Triebel-Lizorkin spaces $\cB^{sq}_p(\SS)$ and $\cF^{sq}_p(\SS)$
with parameters from $\QQ(A)$ (see\eqref{indices-1}).
Here $A>1$ can be arbitrarily large but is fixed and the number of shifts depends on $A$.

(2)
If the old frame $\{\psi_\xi\}$ is a basis, then the new frame $\{\theta_\xi\}$ is also a basis.
This is the case in dimension $d=2$, where we use Meyer's periodic wavelet basis on $\SSS^1$.
As for now there are no convenient bases on $\SS$ when $d>2$.
For this reason we work with frames, which are completely satisfactory for our purposes.

(3)
The rotation group on $\SS$ is not commutative in dimensions $d\ge 3$,
which is a major difference from the translation group in $\RR^{d-1}$.
This is an essential obstacle in constructing highly localized linear combinations of a fixed number
of shifts of the Newtonian kernel with \emph{vanishing} moments on $\SS$.
In order to overcome this difficulty we replace the \emph{vanishing moment} conditions
on the $\varphi$-transform of Frazier and Jawerth
with \emph{small moment} conditions, see e.g. Propositions~\ref{prop:5_1}--\ref{prop:5_3}
and \eqref{eq:small-1} in \thmref{thm:new_frame_diamond}.
In general, the vanishing moment conditions are not valid for $\theta_\xi$.

(4)
We restrict the parameters to $p,q<\infty$ for several reasons.
First, whenever $p,q<\infty$ the Besov and Triebel-Lizorkin spaces $\cB^{sq}_p(\SS)$ and $\cF^{sq}_p(\SS)$
are separable and the finite sequences are dense in the respective Besov and Triebel-Lizorkin sequence spaces.
Also, the respective frame representations converge unconditionally.
These facts are important in the construction and utilization of the frames $\{\theta_\xi\}$, $\{\tilde\theta_\xi\}$.
Furthermore, as is well known, in general, nonlinear $n$-term approximation from frames or bases
in $L^\infty$ or as in our case $\HHH^\infty$ is not quite natural.
Just as in Harmonic analysis one should work in BMO instead.

The intimate relation between the harmonic Hardy and Besov spaces on $B^d$ on the one hand
and the Triebel-Lizorkin and Besov spaces of functions/distributions on $\SS$ on the other
will play a critical role in our development.
Harmonic Besov and Triebel-Lizorkin spaces on $B^d$
and Besov and Triebel-Lizorkin spaces of distributions on $\SS$ with full range of parameters
are treated in \cite{IP}.
In particular, the equivalence of these spaces on $B^d$ and on its boundary $\SS$ is established in \cite{IP}.
These equivalences enable us to mediate between spaces and frames on $B^d$ and on $\SS$.
For example, it allows to transfer the constructed frame $\{\theta_\xi\}$ on $\SS$ to a~frame on $B^d$
and approximation results from $\SS$ to $B^d$.

Our main result in \thmref{thm:Jackson} asserts that if $U\in B^{s\tau}_\tau(\HHH)$ with $s>0$ and $1/\tau=s/(d-1)+1/p$,
then $U\in\HHH^p(B^d)$ and
\begin{equation}\label{Jackson-0}
E_n(U)_{\HHH^p} \le c n^{-s/(d-1)}\|U\|_{B^{s\tau}_\tau(\HHH)}, \quad n\ge 1.
\end{equation}
We derive this estimate from a respective estimate for nonlinear $n$-term approximation
of functions/distributions from $\{\theta_\xi\}$ on the unit sphere $\SS$.

Denote by $\sigma_n(f)_{\cF_p^{02}}$ the best $n$-term approximation of $f$ in the Triebel-Lizorkin space $\cF_p^{02}(\SS)$
from the frame $\{\theta_\xi\}$ mentioned above.
We show that whenever $f~\in~B^{s\tau}_\tau(\SS)$ with $1/\tau=s/(d-1)+1/p$, $0<p<\infty$,
then $f\in\cF_p^{02}(\SS)$ and
\begin{equation}\label{Jackson-1}
\sigma_n(f)_{\cF_p^{02}(\SS)} \le c n^{-s/(d-1)}\|f\|_{B^{s\tau}_\tau(\SS)}, \quad n\ge 1.
\end{equation}
As is well known the harmonic Hardy space $\HHH^p(B^d)$, $0<p<\infty$, can be identified
with the Triebel-Lizorkin space $\cF_p^{02}=\cF_p^{02}(\SS)$ of functions/distributions on $\SS$,
and hence \eqref{Jackson-1} implies that for any harmonic function $U\in B^{s\tau}_\tau(\HHH)$
\begin{equation}\label{Jackson-2}
\sigma_n(U)_{\HHH^p} \le c n^{-s/(d-1)}\|U\|_{B^{s\tau}_\tau(\HHH)}, \quad n\ge 1,
\end{equation}
where $\sigma_n(U)_{\HHH^p}$ stands for the best $n$-term approximation of $U$ in $\HHH^p(B^d)$
from the harmonic extension to $B^d$ of $\{\theta_\xi\}$.
Finally, estimate \eqref{Jackson-2} yields \eqref{Jackson-0} taking into account that
each frame element $\theta_\xi$ is a linear combination of a fixed number of shifts
of the Newtonian kernel.

It is insightful to study the nonlinear approximation from functions as in \eqref{lin-comb}
in the norms of the closely related to $\HHH^p(B^d)$
harmonic Triebel-Lizorkin and Besov spaces $F_p^{0q}(B^d)$ and $B_p^{0q}(B^d)$.
As shown in Theorem~\ref{thm:Jackson-B-F} the nonlinear $n$-term approximations in $F_p^{0q}(B^d)$
have the optimal rate $O(n^{-s/(d-1)})$ for any $0<q<\infty$,
while the nonlinear $n$-term approximation in $B_p^{sq}(B^d)$
achieves this optimal order only for $p\le q<\infty$, see Theorems~\ref{thm:Jackson-B-B}.

\medskip

\noindent
{\em Bernstein inequality: Conjecture.}
We conjecture that the following Bernstein type inequality is valid:
Let $1<p<\infty$, $s>0$, and $1/\tau=s/(d-1)+1/p$. Then
\begin{equation}\label{bernstein}
\|G\|_{B_\tau^{s\tau}(\HHH)} \le cn^{s/(d-1)}\|G\|_{\HHH^p(B^d)}, \quad \forall G\in \NNN_n.
\end{equation}
If valid this estimate along with the Jackson estimate (\ref{Jackson-0}) would lead to a complete characterization
of the rates of approximation (approximation spaces) of nonlinear $n$-term approximation in $\HHH^p(B^d)$, $1<p<\infty$,
from shifts of the Newtonian kernel.

\smallskip

It is natural to pose the question whether the approximation results of this paper hold when $p=\infty$.
We think that just as in the case of rational approximation
analogues of these results are valid if $\HHH^\infty(B^d)$ is replaced by the harmonic BMO space on $B^d$.
We shall not pursue this line of research in the present article.

\smallskip

\noindent
{\em Organization.}
The outline of the paper is as follows.
In Section~\ref{sec:background} we introduce some basic notation and assemble background material
about the maximal operator, spherical harmonics, and maximal $\delta$-nets;
we also give some technical estimates on inner products of localized functions on the sphere.
Section~\ref{s2} presents some basic facts about harmonic Besov and Triebel-Lizorkin spaces
on $B^d$ and $\SS$ developed in \cite{IP};
it also recalls the construction of frames for Besov and Triebel-Lizorkin spaces on $\SS$
and their frame decomposition developed in \cite{NPW2}.
Section~\ref{s3} presents and somewhat refines the small perturbation method for construction of frames developed in \cite{DKKP}.
Section~\ref{s4} deals with localization properties of the frame elements from \S \ref{s2}
and highly localized kernels induced by shifts of the Newtonian kernel, developed in \cite{IP2}.
Section~\ref{s6} contains the construction of a pair of frames for Besov and Triebel-Lizorkin spaces
whose elements are finite linear combinations of shifts of the Newtonian kernel.
Section~\ref{s7} is devoted to nonlinear $n$-term approximation of functions in the harmonic Hardy spaces
from shifts of the Newtonian kernel.
Section~\ref{s8} deals with nonlinear $n$-term approximation in the exterior of the unit ball in $\RR^d$
and in the upper half space in $\RR^d$ from shifts of the Newtonian kernel.
Proofs of key estimates supporting our main results are given in Section~\ref{appendix}.

\smallskip

\noindent
{\em Notation.}
Throughout this article the constants $d$, $M$, and $K$ will appear frequently.
Here $d\in\NN$ is the dimension of the space $\RR^d$,
$M>0$ determines decay rates, and
$K\in \NN$ is a parameter determining the upper bound of the order of derivatives required from some functions.
Positive constants will be denoted by $c$ and they may vary at every occurrence.
Most of these constants will depend only on $d, K, M$.
By $C$'s we denote numbers (constants) that also depend on parameters different from $d, K, M$.
When we would like to trace the dependence of a constant $c$ on these parameters we use indexing, e.g. $c_1,c_2$, etc.
or indicate the dependence on parameters in parenthesis.
These indexed constants preserve their values throughout the article.
The relation $a\sim b$ means that there exists a constant $c\ge 1$ such that $c^{-1}a\le b \le ca$.

\section{Background and technical ground work}\label{sec:background}

\subsection{Basic notation and simple inequalities}\label{s1_2}

In this article we use standard notation.
Thus $\RR^d$ stands for the $d$-dimensional Euclidean space.
The inner product of $x,y\in\RR^d$ is denoted by $x\cdot y =\sum_{k=1}^d x_ky_k$ and
the Euclidean norm of $x$ by $|x|=\sqrt{x\cdot x}$.
We write $\BB(x_0,r):=\{x : |x-x_0|< r\}$
and set $B^d:=\BB(0, 1)$, the open unit ball in $\RR^d$.

As usual $\NN_0$ stands for the set of non-negative integers.
For $\beta=(\beta_1,\dots,\beta_d)\in\NN_0^d$ the monomial $x^\beta$
is defined by $x^\beta:=x_1^{\beta_1}\dots x_d^{\beta_d}$ and its degree is $|\beta|:=\beta_1+\dots+\beta_d$.
The set of all polynomials in $x\in\RR^d$ of total degree $n$ is denoted by $\cP_n^d$.
We denote $\partial_k:=\partial/\partial x_k$
and then $\partial^\beta:=\partial_1^{\beta_1}\dots \partial_d^{\beta_d}$ is a differential operator of order $|\beta|$,
the gradient operator is $\nabla:=(\partial_1,\dots, \partial_d)$,
and $\Delta:=\partial_1^2+\dots +\partial_d^2$ stands for the Laplacian.
When necessary we indicate the variable of differentiation by a subscript, e.g. $\partial^\beta_x$.

The unit sphere in $\RR^d$ is denoted by $\SS:=\{x : |x|=1\}$.
We denote by $\rho(x, y)$ the geodesic distance between $x, y\in\SS$, that is,
$\rho(x,y):=\arccos (x\cdot y)$.
The open spherical cap (ball on the sphere) centred at $\eta\in\SS$ of radius $r$ is denoted by $B(\eta,r)=\{x\in\SS : \rho(\eta,x)<r\}$.
We denote by $\Delta_0$ the Laplace-Beltrami operator on $\SS$.
As is well known (e.g. \cite[Theorem 1.8.2]{DX}) $\Delta_0$ has the decomposition
\begin{equation}\label{eq:laplace_beltrami_decomposition}
\Delta_0 =\sum_{1\le i<\ell\le d}D_{i,\ell}^2,\quad D_{i,\ell}=x_i\partial_\ell-x_\ell\partial_i,\quad x\in\SS.
\end{equation}
For any function $g$ on $\SS$ we denote by $\til{g}$ its standard extension, defined by
\begin{equation}\label{extension}
\til{g}(x):=g(x/|x|) \quad\hbox{for}\quad x\in\RR^d\backslash\{0\}.
\end{equation}
As is well known (e.g. \cite[Corollary 1.4.3]{DX} or \cite{Seeley}) for any $g\in C^2(\SS)$
\begin{equation}\label{eq:laplace_extension}
\Delta \til{g}(x)=\Delta_0 g(x),\quad x\in\SS.
\end{equation}
By definition $g\in W_\infty^K(\SS)$, $K\in \NN_0$, if
$\|\partial^\beta\til{g}\|_{L^\infty(\BB(0,2)\setminus\BB(0,1/2))} \le c$, $\forall |\beta|\le K$.

The Lebesgue measure on $\SS$ is denoted by $\sigma$ and
we set $|E|:=\sigma(E)$ for a~measurable set $E\subset \SS$.
Thus, $\omega_d:=|\SS|=2\pi^{d/2}/\Gamma(d/2)$.

The inner product of $f,g\in L^2(\SS)$ is given by
\begin{equation*}
	\left\langle f,g\right\rangle := \int_{\SS} f(y)\overline{g(y)}\,d\sigma(y).
\end{equation*}
The nonstandard convolution of functions
$F\in L^\infty[-1,1]$ and $g\in L(\SS)$ is defined by
\begin{equation}\label{def-conv}
	F*g(x):=\left\langle F(x\cdot\bullet),\overline{g}\right\rangle=\int_{\SS}F(x\cdot y)g(y) d\sigma(y).
\end{equation}

We say that a function $f$ defined on $\SS$
is \emph{localized around} $\eta\in\SS$ with dilation factor $N$ and decay rate $M>0$ if the estimate
\begin{equation}\label{eq:local_1}
	|f(x)|\le \kappa N^{d-1}(1+N\rho(\eta,x))^{-M},\quad x\in\SS,
\end{equation}
holds for some constant $\kappa >0$ independent of $N, x, \eta$.
The multiplier $N^{d-1}$ is used as part of the decay function in \eqref{eq:local_1} in order to have $\|f\|_{L^1(\SS)}\le c$.
Namely, for $M> d-1$ we have
\begin{equation}\label{eq:conv_1}
\int_{\SS} \frac{N^{d-1}}{(1+N\rho(\eta,y))^M}d\sigma(y)\le c_0,\quad \forall \eta\in\SS,~\forall N\ge 1,
\end{equation}
where $c_0=c(d)/(M-d+1)$ depends only on $d$ and $M$.
The weight function in the right-hand side of \eqref{eq:local_1} also has the property:
For any $\eta_1,\eta_2\in\SS$  with
$\rho(\eta_1,\eta_2)\le N^{-1}$
\begin{equation}\label{comp_local}
(1+N\rho(\eta_2, x))^{-1}\le 2(1+N\rho(\eta_1, x))^{-1},\quad \forall x\in\SS.
\end{equation}
Indeed, $1+N\rho(\eta_1, x)\le 1+N(\rho(\eta_1,\eta_2)+\rho(\eta_2, x))\le 2+N\rho(\eta_2, x)$, which implies \eqref{comp_local}.

Another simple inequality that will be useful is:
\begin{equation}\label{eq:small-1_powers}
|\Delta_0^{K/2}x^\beta|\le c_6,\quad \forall x\in\SS,~0\le|\beta|\le K,
\end{equation}
where $c_6$ depends only on $d$ and $K$.

\subsection{Spherical harmonics}\label{s1_3}

The spherical harmonics will be our main vehicle in dealing with harmonic functions
on the unit ball $B^d$ in $\RR^d$.

Denote by $\cH_\kk$ the space of all spherical harmonics of degree $\kk$ on $\SS$.
As is well known the dimension of $\cH_\kk$ is
$N(\kk, d)= \frac{2\kk+d-2}{\kk}\binom{\kk+d-3}{\kk-1} \sim \kk^{d-1}$.
Furthermore, the spaces $\cH_\kk$, $\kk=0, 1, \dots$, are orthogonal
and $L^2(\SS) =\bigoplus_{k\ge 0} \cH_k$.

Let $\{Y_{\kk \nu}: \nu=1, \dots, N(\kk, d)\}$ be a real-valued orthonormal basis for $\cH_\kk$.
Then the kernel of the orthogonal projector onto $\cH_\kk$ is given by
\begin{equation}\label{Pk}
Z_\kk(x\cdot y) = \sum_{\nu=1}^{N(\kk, d)} Y_{\kk \nu}(x)Y_{\kk \nu}(y),
\quad x, y\in \SS.
\end{equation}
As is well known (see e.g. \cite[Theorem 1.2.6]{DX})
\begin{equation}\label{def-Pk}
Z_\kk(x\cdot y)= \frac{\kk+\mu}{\mu\omega_d}\,
C^{\mu}_\kk(x\cdot y),
\quad x, y\in \SS,
\quad \mu:=\frac{d-2}{2},~d>2.
\end{equation}
Here $C_\kk^{\mu}$ is the Gegenbauer (ultraspherical) polynomial of degree
$\kk$ normalized by
$C_\kk^{\mu}(1)= \binom{\kk + 2\mu-1}{\kk}$.
The Gegenbauer polynomials are usually defined by the following generating function
\begin{equation*}
(1-2uz+z^2)^{-\mu}=\sum_{\kk=0}^\infty C^{\mu}_\kk(u)z^\kk,\quad |z|<1,~|u|<1.
\end{equation*}
The polynomials $C_\kk^{\mu}$, $k=0, 1, \dots$, are orthogonal in the space $L^2([-1, 1], w)$
with weight $w(u):= (1-u^2)^{\mu-1/2}$,
see \cite[p.~80, (4.7.1)]{Sz} or \cite[Table 18.3.1]{OLBC}.
In the case $d=2$ the kernel of the orthogonal projector onto $\cH_\kk$ takes the form
\begin{equation*}
Z_0(x\cdot y)=\frac{1}{2\pi},\quad 	Z_k(x\cdot y)=\frac{1}{\pi}T_k(x\cdot y),\quad k\ge 1,
\end{equation*}
where $T_k(u):=\cos(n\arccos u)$ is the $k$-th degree Chebyshev polynomial of the first kind.
We refer the reader to \cite{Muller, SW} for the basics of spherical harmonics.

As is well known (see e.g. \cite[Theorem 1.4.5]{DX}) the spherical harmonics are
eigenfunctions of the Laplace-Beltrami operator $\Delta_0$ on $\SS$, namely,
\begin{equation}\label{eq:LB1}
-\Delta_0Y(x)=\kk(\kk+d-2)Y(x),\quad~x\in\SS,~\forall Y\in \cH_\kk.
\end{equation}

The set of all band-limited functions (i.e. spherical polynomials) on $\SS$ of degree $\le N$ will be denoted by $\Pi_N$,
i.e. $\Pi_N :=\bigoplus_{\kk=0}^N \cH_\kk$.

The Poisson kernel on the unit ball $B^d$ is given by
\begin{equation}\label{Poisson}
P(y, x):=\sum_{k= 0}^\infty |x|^k\PP_\kk\Big(\frac{x}{|x|}\cdot y\Big)
=\frac{1}{\omega_d}\frac{1-|x|^2}{|x-y|^d},
\quad |x| <1,\; y\in \SS.
\end{equation}

Kernels on the sphere $\SS$ of the form
\begin{equation}\label{def-Lam}
\Lambda_N(x\cdot y) := \sum_{\kk=0}^\infty
\lambda(\kk/N)\PP_\kk(x\cdot y),
\quad x, y\in \SS,
\quad N\ge 1,
\end{equation}
where $\lambda\in C^\infty[0,\infty)$ is compactly supported,
will play a key role in this article.
Observe that in this case
\begin{equation}\label{def-Lam-2}
\Lambda_N(u) := \sum_{\kk=0}^\infty
\lambda(\kk/N)\PP_\kk(u), \quad u\in[-1, 1],
\end{equation}
is simply a polynomial kernel.
The localization of this kernel is given in the following

\begin{thm}\label{thm:localization}
Let $\nu\ge 0$ and $M\in\NN$.
Assume $\lambda\in C^\infty[0, \infty)$,
$\|\lambda^{(m)}\|_\infty \le \kappa$ for $0\le m\le M$ and
either $\supp \lambda \subset [1/4, 4]$ or $\supp \lambda \subset [0, 2]$ and $\lambda(t)=1$ for $t\in [0, 1]$.
Then there exists a constant $c>0$ depending only on $M$, $\nu$, and $d$ such that
for any $N\ge 1$ the kernel $\Lambda_N$ from $(\ref{def-Lam})-(\ref{def-Lam-2})$
obeys
\begin{equation}\label{local-Lam}
|\Lambda_N^{(\nu)}(\cos \theta)| \le c \kappa \frac{N^{d-1+2\nu}}{(1+N|\theta|)^{M}},
\quad |\theta|\le \pi,
\end{equation}
and hence
\begin{equation}\label{local-Lam-2}
|\Lambda_N^{(\nu)}(x\cdot y)| \le c \kappa \frac{N^{d-1+2\nu}}{(1+N\rho(x, y))^{M}},
\quad x, y\in\SS.
\end{equation}
Furthermore, for $x, y, z\in \SS$
\begin{equation}\label{local-Lip-1}
|\Lambda_N(x\cdot z)-\Lambda_N(y\cdot z)| \le c\kappa \frac{\rho(x, y) N^d}{(1+N\rho(x, z))^{M}},
\quad\hbox{if}\quad  \rho(x, y) \le N^{-1}.
\end{equation}
\end{thm}

For a proof, see  \cite[Theorem 3.5]{NPW1} and \cite[Lemmas 2.4, 2.6]{NPW2}, also \cite{IPX}.

\subsection{Maximal $\delta$-nets and cubature formulas on the sphere}\label{s5_3}

For discretization of integrals and construction of frames on $\SS$ we shall need cubature formulas,
which are naturally constructed using maximal $\delta$-nets on $\SS$.

\smallskip

\noindent
{\bf Definition.}
Given $\delta>0$ we say that a finite set $\cZ\subset\SS$ is a maximal $\delta$-net on $\SS$
if
(i) $\rho(\zeta_1,\zeta_2)\ge\delta$ for all $\zeta_1,\zeta_2\in\cZ$, $\zeta_1\ne\zeta_2$, and
(ii) $\cup_{\zeta\in\cZ}B(\zeta,\delta)=\SS$.

Clearly, a maximal $\delta$-net on $\SS$ exists for any $\delta>0$.
For every maximal $\delta$-net $\cZ$
it is easy to construct (see \cite[Proposition 2.5]{CKP}) a disjoint partition $\{\nA_\zeta\}_{\zeta\in\cZ}$ of $\SS$
consisting of measurable sets such that
\begin{equation}\label{disjoint}
B(\zeta,\delta/2)\subset \nA_\zeta\subset B(\zeta,\delta),\quad \zeta\in\cZ.
\end{equation}

Two kinds of cubature formulas on $\SS$ will be utilized.

\smallskip

\noindent
{\bf Simple cubature formulas on $\SS$.}
Let $0<\gamma\le 1$ be a parameter to be selected.
Let $\cZ_j\subset\SS$ ($j\in\NN$) be a maximal $\delta_j$-net on $\SS$ with $\delta_j:=\gamma 2^{-j+1}$.
We shall use the cubature formula
\begin{equation}\label{simpl-cubature}
\int_{\SS} f(x) d\sigma(x)\approx \sum_{\zeta\in \cZ_j} w_\zeta f(\zeta), \quad w_\zeta:=|\nA_\zeta|,
\end{equation}
where $\nA_\zeta$ is from \eqref{disjoint} with $\cZ=\cZ_j$, $\delta=\delta_j$.
The cubature \eqref{simpl-cubature} is apparently exact for all constants.
Evidently,
\begin{equation}\label{eq:maximal_net}
	w_\zeta=|\nA_\zeta|\sim (\gamma 2^{-j+1})^{d-1}
\end{equation}
with constants of equivalence depending only on $d$.
Note that \eqref{disjoint} implies that the number of elements in $\cZ_j$ is
$\le c(d)(\gamma^{-1}2^j)^{d-1}$.

Further,
given $j\in\NN$ we define a map $\zeta$ from $\SS$ to $\cZ_j$ as follows:
For every $y\in\SS$ we set $\zeta(y):=\eta\in\cZ_j$ if $y\in\nA_\eta$.
We shall use this map in Lemmas~\ref{lem:1_3} and \ref{lem:1_4} below.

\smallskip

\noindent
{\bf Nontrivial cubature formulas on $\SS$.}
Let $\cX_j\subset\SS$ be a maximal $\delta_j$-net on $\SS$ with $\delta_j:=\gamma 2^{-j+1}$, $0<\gamma<1$, $j\ge 1$.
In \cite[Theorem 4.3]{NPW1} it is shown that there exist $\gamma$ ($0<\gamma <1$), depending only on $d$,
and weights $\{\ww_\xi\}_{\xi\in\cX_j}$, satisfying
\begin{equation}\label{cubature_w}
c_7^{-1} 2^{-j(d-1)}\le \ww_\xi\le c_7 2^{-j(d-1)}, \quad \xi\in\cX_j,
\end{equation}
with constant $c_7$ depending only on $d$, such that the cubature formula
\begin{equation}\label{cubature}
\int_\SS f(x)d\sigma(x)\approx \sum_{\xi\in\cX_j} \ww_\xi f(\xi)
\end{equation}
is exact for all spherical harmonics $f$ of degree $\le 2^{j+1}$, $j\ge 1$.

As before,
the number of nodes in $\cX_j$ is $\le c(d)(\gamma^{-1}2^j)^{d-1}$,
since $\cX_j$ is a maximal $\delta_j$-net on $\SS$.
Also the disjoint partition $\{\nA_\xi\}_{\xi\in\cX_j}$ of $\SS$ exists with
$B(\xi,\delta_j/2)\subset \nA_\xi\subset B(\xi,\delta_j)$, but the equality $\ww_\xi=|A_\xi|$ \emph{does not hold} in general.

\subsection{Maximal operator}

The maximal operator is an important technical tool when dealing with Besov and Triebel-Lizorkin spaces.
We shall use the following version of the Hardy-Littlewood maximal operator:
\begin{equation}\label{def:max-op}
\MM_tf(x):=\sup_{B\ni x}\Big(\frac1{|B|}\int_B |f|^t\, d\sigma \Big)^{1/t},
\quad x\in \SS, \; t>0,
\end{equation}
where the sup is over all spherical caps $B\subset \SS$ such that $x\in B$.

The Fefferman-Stein vector-valued maximal inequality (see \cite[Ch. II (13), p. 56]{Stein})
can be written in the form:
{\em If $0<p<\infty, 0<q\le\infty$, and
$0<t<\min\{p,q\}$, then for any sequence of measurable functions
$\{f_\nu\}$ on $\SS$}
\begin{equation}\label{max-ineq}
\Big\|\Big(\sum_{\nu}|\MM_tf_\nu(\cdot)|^q\Big)^{1/q} \Big\|_{L^p}
\le \tc_1\Big\|\Big(\sum_{\nu}| f_\nu(\cdot)|^q\Big)^{1/q}\Big\|_{L^p}.
\end{equation}
From Theorem 2.1 in \cite{GLY} it follows that the constant $\tc_1$ above can be written in the form
\begin{equation}\label{max-const}
\tc_1=\Big(c^* \max\big\{p/t, (p/t-1)^{-1}\big\}\max\big\{1, (q/t-1)^{-1}\big\}\Big)^{1/t},
\end{equation}
where $c^*>0$ is a constant depending only on $d$.

Note that the area/volume of a spherical cap $B(x,r)$ on $\SS$, $d\ge 2$, is given by
\begin{equation*}
|B(x,r)|=\omega_{d-1}\int_0^r \sin^{d-2}v\,dv.
\end{equation*}
Hence
\begin{equation}\label{sph_cap2}
|B(x_1,r_1)|/|B(x_2,r_2)|\le (r_1/r_2)^{d-1},\quad 0<r_2\le r_1\le\pi,\quad x_1,x_2\in\SS,
\end{equation}
\begin{equation}\label{sph_cap}
1/\tc_2\le |B(x,r)|/r^{d-1}\le \tc_2,\quad 0<r\le\pi,\quad x\in\SS,
\end{equation}
where $\tc_2$ is a constant depending only on $d$.

\subsection{Inner products of zonal functions}\label{s5_1}

A function $f$ on $\SS$ is zonal if it is invariant under rotation about a fixed axis.
If this axis is in the direction of $\eta\in\SS$, then $f$ can be represented as $f(x)=F(\eta\cdot x)$, $x\in\SS$,
for an appropriate function $F:[-1,1]\to\RR$.

\begin{lem}\label{lem:0}
Let $F, G\in L^\infty[-1,1]$. Then there exists $H\in C[-1,1]$ such that
\begin{equation}\label{convol}
H(x\cdot z)=\int_{\SS}F(x\cdot y)G(y\cdot z)\,d\sigma(y),\quad \forall x,z\in\SS.
\end{equation}
\end{lem}
\begin{proof}
Assume first that $F$ and $G$ are algebraic polynomials of degree $m$.
Then we can expend them in Gegenbauer polynomials to obtain
$F=\sum_{k=0}^m \hat F_k\PP_k$ and
$G=\sum_{k=0}^m \hat G_k\PP_k$.
Using that $\PP_k(x\cdot y)$ is the kernel of the orthogonal projector onto $\cH_k$
we have
$\int_{\SS}\PP_k(x\cdot y)\PP_k(y\cdot z)\,d\sigma(y)=\PP_k(x\cdot z)$ (see \eqref{Pk}).
Therefore,
\begin{align*}
\int_{\SS}F(x\cdot y)G(y\cdot z)\,d\sigma(y)
=\sum_{k=0}^m \hat F\hat G\PP_k(x\cdot z)
=H(x\cdot z),
\end{align*}
where $H$ is an algebraic polynomial of degree $m$.
Thus (\ref{convol}) holds for polynomials.
Finally, a limiting argument implies that (\ref{convol}) is valid in general.
\end{proof}

From \lemref{lem:0} it follows that for any $F, G\in L^\infty[-1,1]$
\begin{equation}\label{eq:comut}
\int_{\SS}F(x\cdot y)G(y\cdot z)\,d\sigma(y)
=\int_{\SS}F(z\cdot y)G(y\cdot x)\,d\sigma(y),\quad \forall x,z\in\SS.
\end{equation}

\subsection{Inner products of localized functions}\label{s5_2}

The estimation of the inner products of well localized functions and functions with small moments on the sphere
will play a key role in our further development.
The following proposition is an analogue of \cite[Lemma B1]{FJ}.
We replace the vanishing moment condition used in \cite{FJ}
by the weaker ``small moments'' condition \eqref{eq:3a}.

\begin{prop}\label{prop:5_1}
Let $K\in\NN$, $M>K+d-1$, $\scal_2\ge\scal_1\ge 1$ $(\scal_1,\scal_2\in\RR)$, and $\kappa_1, \kappa_2>0$.
Assume $f\in L^\infty(\SS)$ and $g\in W_\infty^K(\SS)$, see \S\ref{s1_2}.
Furthermore, assume that for some $x_1,x_2\in\SS$
\begin{align}
\left|\partial^\beta \til{g}(y)\right|
&\le \frac{\kappa_1 \scal_1^{|\beta|+d-1}}{(1+\scal_1\rho(x_1,y))^M},
\quad \forall y\in\SS,~0\le |\beta|\le K, \label{eq:1a}\\
|f(y)|
&\le \frac{\kappa_2 \scal_2^{d-1}}{(1+\scal_2\rho(x_2,y))^M},
\quad \forall y\in\SS, \quad\hbox{and}\label{eq:2a}
\end{align}
\begin{equation}\label{eq:3a}
\Big|\int_{\SS}y^\beta f(y)\, d\sigma(y)\Big|
\le \kappa_2 \scal_2^{-K},\quad 0\le |\beta|\le K-1.
\end{equation}
Then
\begin{equation}\label{eq:4a}
|\left\langle g,f\right\rangle|=\Big|\int_{\SS}g(y)\overline{f(y)}\, d\sigma(y)\Big|
\le c_1\frac{\kappa_1 \kappa_2 (\scal_1/\scal_2)^{K}\scal_1^{d-1}}{(1+\scal_1\rho(x_1,x_2))^M},
\end{equation}
where $c_1$ depends only on $d$, $K$, and $M$.
Above $\til{g}(y):=g(y/|y|)$ for $y\in\RR^d\backslash\{0\}$ just as in \eqref{extension}.
\end{prop}

For cases where condition \eqref{eq:3a} may not be satisfied we modify Proposition~\ref{prop:5_1}
as follows

\begin{prop}\label{prop:5_2}
Let $M>d$, $\scal_2\ge\scal_1\ge 1$,
and $\kappa_1, \kappa_2>0$.
Let $f\in L^\infty(\SS)$ and $g\in W_\infty^1(\SS)$.
Also, assume that for some $x_1,x_2\in\SS$
\begin{align}
\left|\partial^\alpha \til{g}(y)\right|
&\le \frac{\kappa_1 \scal_1^{d}}{(1+\scal_1\rho(x_1,y))^M},
\quad \forall y\in\SS,~|\alpha|=1, \label{eq:1b}\\
|f(y)|&\le \frac{\kappa_2 \scal_2^{d-1}}{(1+\scal_2\rho(x_2,y))^M},\quad \forall y\in\SS.\label{eq:2b}
\end{align}
Then
\begin{equation}\label{eq:4b}
\Big|\int_{\SS}g(y)\overline{f(y)}\, d\sigma(y)-g(x_2)\int_{\SS}\overline{f(y)}\, d\sigma(y)\Big|
\le c_2\frac{\kappa_1 \kappa_2 (\scal_1/\scal_2)\scal_1^{d-1}}{(1+\scal_1\rho(x_1,x_2))^M},
\end{equation}
where $c_2$ depends only on $d$ and $M$.
\end{prop}

In cases where only the function localizations are known/matter we use

\begin{prop}\label{prop:5_3}
Let $M> d-1$, $\scal_2\ge\scal_1\ge 1$,
$\kappa_1, \kappa_2>0$.
Let $f,g\in L^\infty(\SS)$,
and for some $x_1,x_2\in\SS$
\begin{equation}\label{eq:1c}
	\left|g(y)\right|
\le \frac{\kappa_1 \scal_1^{d-1}}{(1+\scal_1\rho(x_1,y))^M},\quad \forall y\in\SS,
\end{equation}
\begin{equation}\label{eq:2c}
	|f(y)|\le \frac{\kappa_2 \scal_2^{d-1}}{(1+\scal_2\rho(x_2,y))^M},\quad \forall y\in\SS.
\end{equation}
Then
\begin{equation}\label{eq:4c}
|\left\langle g,f\right\rangle|=\Big|\int_{\SS}g(y)\overline{f(y)}\, d\sigma(y)\Big|
\le c_3\frac{\kappa_1 \kappa_2 \scal_1^{d-1}}{(1+\scal_1\rho(x_1,x_2))^M},
\end{equation}
where $c_3$ depends only on $d$ and $M$.
\end{prop}

To streamline our presentation we defer the proofs of
Propositions~\ref{prop:5_1}, \ref{prop:5_2}, and \ref{prop:5_3} to Section~\ref{appendix}.

\section{Spaces of functions and distributions on the ball and sphere}\label{s2}

The theory of harmonic Besov and Triebel-Lizorkin spaces on $B^d$,
and their relation with the respective Besov and Triebel-Lizorkin spaces of distributions on $\SS$ is developed in \cite{IP}.
In this section we review all definitions and results that will be needed from \cite{IP}.

Denote by $\HHH(B^d)$ the set of all harmonic functions on the unit ball $B^d$ in $\RR^d$.

\subsection{Harmonic Besov and Triebel-Lizorkin spaces on $B^d$}\label{s2_1}

It is convenient to define the harmonic Besov and Triebel-Lizorkin spaces on $B^d$
by using their expansion in solid spherical harmonics.
As in \S~\ref{s1_3} let $\{Y_{\kk j}: j=1, \dots, N(\kk, d)\}$ be a real-valued orthonormal basis for $\cH_\kk$.
The harmonic coefficients of $U\in\HHH(B^d)$ are defined by
\begin{equation}\label{coeff-cU}
\bbb_{\kk \nu}(U):= \frac{1}{a^\kk}\int_\SS U(a\eta)Y_{\kk \nu}(\eta) d\sigma(\eta)
\end{equation}
for some $0<a<1$. It is an important observation that the coefficients
are independent of $a$ for all $0<a<1$.
This implies the representation
\begin{equation}\label{rep-U-3a}
U(r\xi) = \sum_{\kk=0}^\infty \sum_{\nu=1}^{N(\kk, d)} \bbb_{\kk \nu}(U)r^\kk Y_{\kk \nu}(\xi),
\quad 0\le r<1, ~\xi\in \SS,
\end{equation}
where the convergence is absolute and uniform on every compact subset of $B^d$.

For $U\in \HHH(B^d)$ and $\beta\in \RR$ we define
\begin{equation}\label{def-JU}
J^\beta U(r\xi)
:= \sum_{\kk=0}^\infty r^\kk (\kk+1)^{-\beta}\sum_{\nu=1}^{N(\kk, d)} \bbb_{\kk\nu}(U)Y_{\kk\nu}(\xi),
\quad 0\le r<1, ~\xi\in \SS.
\end{equation}
The above series converges absolutely and uniformly
on every compact subset of $B^d$
and hence $J^\beta U$ is a well defined harmonic function on $B^d$.

\begin{defn}\label{def:H-B-F}
Let $s\in\RR$, $0<q\le \infty$, and $\beta:=s+1$.

$(a)$
The harmonic Besov space $B^{sq}_p(\HHH)$, $0<p\le \infty$, is defined as the set of all $U\in \HHH(B^d)$ such that
$$
\|U\|_{B^{sq}_p(\HHH)}
:= \Big(\int_0^1 (1-r)^{(\beta-s)q}\|J^{-\beta} U(r\cdot)\|_{L^p(\SS)}^q \frac{dr}{1-r}\Big)^{1/q} <\infty
\quad\hbox{if}\; q\ne \infty
$$
and
$$
\|U\|_{B^{s\infty}_p(\HHH)} := \sup_{0<r<1} (1-r)^{\beta-s}\|J^{-\beta} U(r\cdot)\|_{L^p(\SS)} <\infty.
$$

$(b)$ The harmonic Triebel-Lizorkin space $F^{sq}_p(\HHH)$, $0<p<\infty$, is defined as the set of all $U\in \HHH(B^d)$ such that
$$
\|U\|_{F^{sq}_p(\HHH)}
:= \Big\|\Big(\int_0^1 (1-r)^{(\beta-s)q}|J^{-\beta} U(r\cdot)|^q \frac{dr}{1-r}\Big)^{1/q}\Big\|_{L^p(\SS)} <\infty
\quad\hbox{if}\; q\ne \infty
$$
and
$$
\|U\|_{F^{s\infty}_p(\HHH)} := \Big\|\sup_{0<r<1} (1-r)^{\beta-s}|J^{-\beta} U(r\cdot)|\Big\|_{L^p(\SS)} <\infty.
$$
\end{defn}
Choosing an arbitrary $\beta>s$ above results in equivalent quasi-norms
for the spaces $B^{sq}_p(\HHH)$ and $F^{sq}_p(\HHH)$.

\subsection{Besov and Triebel-Lizorkin spaces on $\SS$}\label{s2_2}

The Besov and Triebel-Lizorkin spaces on $\SS$ in general are spaces of distributions.
As test functions we use the class
$\cS:= C^\infty(\SS)$ of all functions $\phi$ on $\SS$ such that
$$
\|\PP_\kk*\phi\|_2 \le c(\phi, m)(1+\kk)^{-m}, \quad \forall \kk, m\ge 0.
$$
Recall that the convolution $\PP_\kk*\phi$ is defined in \eqref{def-conv}.
The topology on $\cS$ is defined by the sequence of norms
\begin{equation}\label{test_norms}
P_m(\phi):= \sum_{\kk=0}^\infty (\kk+1)^m \|\PP_\kk*\phi\|_2
= \sum_{\kk=0}^\infty (\kk+1)^m \Big(\sum_{\nu=1}^{N(\kk, d)} |\langle \phi, Y_{\kk \nu}\rangle|^2\Big)^{1/2}.
\end{equation}
$\cS$ is complete in this topology.

Observe that all $Y_{\kk \nu}\in\cS$ and hence by (\ref{Pk}) $\PP_\kk(x\cdot y)\in \cS$
as a function of $x$ for every fixed $y$ and as function of $y$ for every fixed $x$.

The space $\cS':=\cS'(\SS)$ of distributions on $\SS$ is defined as the space of
all continuous linear functionals on $\cS$.
The pairing of $f\in \cS'$ and $\phi\in\cS$ will be denoted by
$\langle f, \phi\rangle := f(\overline{\phi})$, which is consistent with the inner product
on $L^2(\SS)$.
More precisely, $\cS'$ consists of all linear functionals $f$ on $\cS$ for which
there exist constants $c>0$ and $m\in\NN_0$ such that
\begin{equation}\label{def-distr}
|\langle f, \phi\rangle| \le c P_m(\phi),\quad \forall \phi\in\cS.
\end{equation}
For any $f\in\cS'$ we define $\PP_\kk*f$ by
\begin{equation}\label{def-P-f}
\PP_\kk*f(x) :=\langle f, \overline{\PP_\kk(x\cdot\bullet)}\rangle
=\langle f, \PP_\kk(x\cdot\bullet)\rangle,
\end{equation}
where on the right $f$ is acting on $\overline{\PP_\kk(x\cdot y)}=\PP_\kk(x\cdot y)$ as a function of $y$
($\PP_k$ is real-valued).

Observe that the representation
\begin{equation}\label{represent-f}
f=\sum_{\kk=0}^\infty \PP_\kk*f, \quad \forall f\in\cS'
\end{equation}
holds with convergence in distributional sense.

\begin{defn}\label{def:B-F-spaces}
Let $s\in\RR$, $0<q\le \infty$, and
$\varphi$ satisfy the conditions:
$\varphi \in C^\infty(\RR_+)$,
$\supp \varphi \subset [1/2, 2]$, and
$|\varphi(u)|\ge c>0$ for $u\in [3/5, 5/3]$.
For a distribution $f\in \cS'$ set
\begin{equation}\label{rep-Phi-j-f}
\Phi_0*f = \PP_0*f,\quad
\Phi_j*f = \sum_{\kk=0}^\infty \varphi\Big(\frac{\kk}{2^{j-1}}\Big)\PP_\kk*f, ~ j\ge 1,
\end{equation}
where $\PP_k*f$ is defined in \eqref{def-P-f}.

$(a)$
The Besov space $\cB^{sq}_p:=\cB^{sq}_p(\SS)$, $0<p\le \infty$, is defined as the set of all distributions
$f\in \cS'$ such that
\begin{equation}\label{B-norm}
\|f\|_{\cB^{s q}_p} :=
\Big(\sum_{j=0}^\infty \Big(2^{sj}\|\Phi_j*f\|_{L^p(\SS)}\Big)^q\Big)^{1/q}
< \infty,
\end{equation}
where the $\ell^q$-norm is replaced by the sup-norm if $q=\infty$.

$(b)$
The Triebel-Lizorkin space $\cF^{sq}_p:=\cF^{sq}_p(\SS)$, $0<p<\infty$, is defined as the set of all distributions
$f\in \cS'$ such that
\begin{equation}\label{F-norm}
\|f\|_{\cF^{s q}_p} :=
\Big\|\Big(\sum_{j=0}^\infty \big(2^{sj}|\Phi_j*f(\cdot)|\big)^q\Big)^{1/q}\Big\|_{L^p(\SS)}
< \infty,
\end{equation}
where the $\ell^q$-norm is replaced by the sup-norm if $q=\infty$.
\end{defn}

Note that the definitions of the Besov and Triebel-Lizorkin spaces above are independent of
the particular selection of the function $\varphi$ with the required properties, that is,
different $\varphi$'s produce equivalent quasi-norms.

\subsection{Identification of harmonic Besov and Triebel-Lizorkin spaces}\label{s2_3}

We are interested in harmonic functions $U\in \HHH(B^d)$
with coefficients of \emph{at most polynomial growth}:
\begin{equation}\label{coef-growth}
|\bbb_{\kk \nu}(U)| \le c(\kk+1)^\gamma,
\quad \nu=1, \dots, N(\kk, d), \;\; \kk=0, 1, \dots,
\end{equation}
for some constants $\gamma, c>0$.
The functions in the harmonic Besov and Triebel-Lizorkin spaces
have this property.

The relationship between harmonic functions on $B^d$ and distributions on $\SS$
is clarified by the following

\begin{prop}\label{prop:identify}
$(a)$ To any $U\in \HHH(B^d)$ represented by $(\ref{rep-U-3a})$
with coefficients satisfying $(\ref{coef-growth})$
there corresponds a distribution $f\in \cS'$, $f=f_U$,
$($the boundary value function/distribution of $U$$)$
defined by
\begin{equation}\label{rep-f-33}
f:= \sum_{\kk=0}^\infty \sum_{\nu=1}^{N(\kk, d)} \bbb_{\kk \nu}(U) Y_{\kk \nu}
\quad \hbox{$($convergence in $\cS'$$)$}
\end{equation}
with coefficients
$\bbb_{\kk \nu}(U)=\langle f, Y_{\kk \nu}\rangle$.

$(b)$ To any distribution $f\in\cS'$ with coefficients $\bbb_{\kk \nu}(f):=\langle f, Y_{\kk \nu}\rangle$
there corresponds a harmonic function $U\in \HHH(B^d)$, $U=U_f$,
$($the harmonic extension of $f$ to $B^d$$)$, defined by
\begin{equation}\label{rep-U-33}
U(x) = \sum_{\kk=0}^\infty \sum_{\nu=1}^{N(\kk, d)} \bbb_{\kk \nu}(f)|x|^\kk Y_{\kk \nu}\Big(\frac{x}{|x|}\Big),
\quad |x|<1,
\end{equation}
with coefficients $\bbb_{\kk \nu}(U)=\bbb_{\kk \nu}(f)$ obeying $(\ref{coef-growth})$,
where the series converges uniformly on every compact subset of $B^d$.

$(c)$ For every $U\in \HHH(B^d)$ we have $U_{f_U} = U$ and for every $f\in \cS'$ we have $f_{U_f} = f$.
\end{prop}

The principle results of this subsection are:

\begin{thm}\label{thm:equiv-norms-F}
Let $s\in \RR$, $0<p < \infty$, $0<q\le \infty$.
A harmonic function $U\in F^{s q}_p(\HHH)$
if and only if its boundary value distribution $f=f_U$
defined by $(\ref{rep-f-33})$ belongs to $\cF^{sq}_p(\SS)$,
moreover
$\|U\|_{F^{s q}_p} \sim \|f\|_{\cF^{s q}_p}$.
\end{thm}

\begin{thm}\label{thm:equiv-norms-B}
Let $s\in \RR$, $0<p, q\le \infty$.
A harmonic function $U\in B^{s q}_p(\HHH)$
if and only if its boundary value distribution $f=f_U$
defined by $(\ref{rep-f-33})$ belongs to $\cB^{sq}_p(\SS)$,
moreover
$\|U\|_{B^{s q}_p} \sim \|f\|_{\cB^{s q}_p}$.
\end{thm}

\subsection{Harmonic Hardy spaces}\label{s2_5}

Here we consider the harmonic Hardy spaces $\HHH^p(B^d)$ on the ball
(usually denoted by $h^p(B^d)$).

\begin{defn}\label{def:Hp}
The space $\HHH^p :=\HHH^p(B^d)$, $0< p\le \infty$,
is defined as the set of all harmonic functions $U\in \HHH(B^d)$
such that
\begin{equation}\label{def:Hp-norm}
\|U\|_{\HHH^p} := \|\sup_{0\le r<1} |U(r\cdot)|\|_{L^p(\SS)} <\infty.
\end{equation}
\end{defn}

The following identification
of harmonic Hardy spaces holds.

\begin{thm}\label{thm:identify-Hp}
A harmonic function $U\in\HHH^p(B^d)$, $0<p<\infty$, if and only if its boundary distribution
$f_U \in \cF^{02}_p(\SS)$ and
\begin{equation}\label{norm-Hp-Fp}
\|U\|_{\HHH^p} \sim \|U\|_{F^{02}_p(\HHH)} \sim \|f_U\|_{\cF^{02}_p(\SS)}.
\end{equation}
Furthermore, $U\in\HHH^p(B^d)$, $1<p<\infty$, if and only if $f_U \in L^p(\SS)$ and
\begin{equation}\label{norm-Hp-Lp}
\|U\|_{\HHH^p} \sim \|f_U\|_{L^p(\SS)}.
\end{equation}
In addition, for any $U\in \HHH^p(B^d)$, $1<p<\infty$,
\begin{equation}\label{equiv-Hp-norm}
\|U\|_{\HHH^p} \sim \sup_{0\le r<1} \|U(r\cdot)\|_{L^p(\SS)}
\end{equation}
and the right-hand side quantity is sometimes used to define $\|U\|_{\HHH^p}$ for $p>1$.
\end{thm}

To prove this theorem we shall need the following

\begin{lem}\label{lem:coeff}
If $U\in\HHH^p(B^d)$, $0<p\le\infty$, then
\begin{equation}\label{coeff-growth2}
|\bbb_{\kk \nu}(U)| \le c(\kk+1)^\gamma\|U\|_{\HHH^p},
\quad \nu=1, \dots, N(\kk, d), \;\; \kk=0, 1, \dots,
\end{equation}
for some constants $\gamma, c>0$, depending only on $d$ and $p$,
i.e. inequalities \eqref{coef-growth} are valid.
Consequently, there exists a distribution $f_U\in\cS'$ with spherical harmonic coefficients
the same as the coefficients of $U$,
which in turn leads to
$$
U=P*f_U
$$
with $P(y, x)$ being the Poisson kernel, see \eqref{Poisson}.
Here $P*f_U$ is defined by
$$
P*f_U(x):= \langle f_U, \overline{P(\cdot, x)}\rangle = \langle f_U, P(\cdot, x)\rangle,
$$
where $f_U$ acts on $\overline{P(y, x)} = P(y, x)$ as a function of $y$ $($$P(y, x)$ is real-valued$)$.
\end{lem}

\begin{proof}
To prove \eqref{coeff-growth2} we invoke Proposition 4.2 from \cite{IP}
which, in particular, asserts that for any $U\in B_p^{sq}(\HHH)$, $s\in\RR$, $0<p,q\le \infty$,
\begin{equation}\label{coeff-grow3}
|\bbb_{\kk \nu}(U)| \le c(\kk+1)^\gamma\|U\|_{B_p^{sq}(\HHH)},
\quad \nu=1, \dots, N(\kk, d), \;\; \kk=0, 1, \dots,
\end{equation}
where the constants $\gamma, c>0$ depend only on $d,s,p,q$.

If $U\in \HHH^p(B^d)$ then by Definition~\ref{def:H-B-F} with $\beta=0$
$$
\|U\|_{B_p^{-1,1}(\HHH)}=\int_0^1\|U(r\cdot)\|_p dr
\le \|\sup_{0\le r<1}|U(r\cdot)|\|_p
= \|U\|_{\HHH^p},
$$
which implies that $\HHH^p(B^d)$ is continuously embedded in the harmonic Besov space $B_p^{-1,1}(\HHH)$.
Now, the above and \eqref{coeff-grow3} imply \eqref{coeff-growth2}.

We set
\begin{equation}\label{rep-f-main}
f_U:= \sum_{\kk\ge 0}\sum_{\nu=1}^{N(\kk, d)} \bbb_{\kk \nu}(U) Y_{\kk \nu}.
\end{equation}
Inequalities \eqref{coeff-growth2} and Proposition~\ref{prop:identify} lead to the conclusion that
the series in \eqref{rep-f-main} converge in $\cS'$ and defines a distribution
$f_U\in \cS'$ with coefficients $\bbb_{\kk \nu}(f_U):=\bbb_{\kk \nu}(U)$.
In turn, this implies that
$$
\sum_{\nu=1}^{N(\kk, d)} \bbb_{\kk \nu}(U) Y_{\kk \nu}\Big(\frac{x}{|x|}\Big) = Z_k*f_U\Big(\frac{x}{|x|}\Big),
\quad |x|<1,
$$
and hence
\begin{align*}
U(x)&= \sum_{k=0}^\infty |x|^k Z_k*f_U\Big(\frac{x}{|x|}\Big)
= \sum_{k=0}^\infty |x|^k \Big\langle f_U, Z_k\Big(\cdot, \frac{x}{|x|}\Big)\Big\rangle
\\
&=\lim_{m\to\infty} \Big\langle f_U, \sum_{k=0}^m |x|^k Z_k\Big(\cdot, \frac{x}{|x|}\Big)\Big\rangle
= \langle f_U, P(\cdot, x)\rangle
= P*f_U(x).
\end{align*}
Here we used the obvious fact that for any $|x|<1$
$$
P(y, x)=\lim_{m\to\infty}\sum_{k=0}^m |x|^k Z_k\Big(y, \frac{x}{|x|}\Big)
\quad\hbox{(convergence in $\cS$)}.
$$
The proof is complete.
\end{proof}

\begin{proof}[Proof of Theorem~\ref{thm:identify-Hp}]
The Hardy space $\HHH^p(\SS)$, $0<p<\infty$, on the sphere is defined as the set of all distributions $f\in\cS'$
such that
\begin{equation}\label{Hardy-sphere}
\|f\|_{\HHH^p(\SS)} :=\|\sup_{0\le r<1}|P*f(r \cdot)|\|_{L^p(\SS)} <\infty.
\end{equation}
A frame characterization of the Triebel-Lizorkin spaces on $\SS$ has been established in \cite[Theorem~4.5]{NPW2}
(see Theorem~\ref{thm:F-Bnorm-equivalence} (b) below),
which along with the same frame characterization of the Hardy spaces $\HHH^p(\SS)$ from \cite[Theorem~1.1]{Dai}
implies that $\HHH^p(\SS)=\cF_p^{0 2}(\SS)$, $0<p<\infty$, with equivalent quasi-norms.

By Lemma~\ref{lem:coeff} and \eqref{Hardy-sphere} it follows that
$U\in \HHH^p(B^d)$, $0<p<\infty$, if and only if $f_U \in \HHH^p(\SS)$ and
$\|U\|_{\HHH^p(B^d)} = \|f_U\|_{\HHH^p(\SS)}$.
This along with the above observation and Theorem~\ref{thm:equiv-norms-F} implies
$$
\|U\|_{\HHH^p(B^d)} = \|f_U\|_{\HHH^p(\SS)} \sim \|f_U\|_{\cF_p^{0 2}(\SS)} \sim \|U\|_{\cF_p^{0 2}(\HHH)},
$$
which confirms \eqref{norm-Hp-Fp}.

The equivalence $\|U\|_{\HHH^p} \sim \|f_U\|_{L^p(\SS)}$, when $1<p<\infty$,
follows by Lemma~\ref{lem:coeff} and the maximal inequality
just as in the case of Hardy spaces on $\RR^d$,
see \cite[\S~1.2.1, p.~91]{Stein}.
For the equivalence
$\sup_{0\le r<1} \|U(r\cdot)\|_{L^p(\SS)} \sim \|f_U\|_{L^p(\SS)}$, $1<p<\infty$,
see \cite[Chapter 6]{ABR}.
The proof is complete.
\end{proof}

\subsection{Frame decomposition of distribution spaces on $\SS$}\label{subsec:frame-SS}

We next recall the construction of the frame (needlets) on $\SS$
from \cite{NPW2}.
Note that in dimension $d=2$ the Meyer's periodic wavelets (see \cite{Meyer})
form a basis with the desired properties.

The first step in the construction of needlets on $\SS$, $d>2$, is the selection of
a real-valued function $\aa\in C^\infty(\RR_+)$ with the properties:
$\supp \aa\subset [1/2, 2]$, $0\le \aa \le 1$, $\aa(u)\ge c>0$ for $u\in[3/5, 5/3]$,
$\aa^2(u)+\aa^2(u/2)=1$ for $u\in [1, 2]$,
and hence
$\sum_{\nu=0}^\infty \aa^2(2^{-\nu}u) =1$ for $u\in [1, \infty)$.
Set
\begin{equation}\label{def-Psi-j}
\LL_0:=Z_0,
\quad\hbox{and}\quad
\LL_j:=\sum_{\kk=0}^\infty \aa\Big(\frac{\kk}{2^{j-1}}\Big)Z_\kk, \quad j\ge 1.
\end{equation}
It is easy to see that
$f=\sum_{j=0}^\infty \LL_j*\LL_j*f$ for every $f\in\cS'$ (convergence in $\cS'$).

The next step is to discretize $\LL_j*\LL_j$ for $j\ge 1$ by using the cubature formula on $\SS$
from \eqref{cubature}, where $\cX_j$ is a maximal $\delta_j$-net with $\delta_j=\gamma 2^{-j+1}$, $0<\gamma <1$.
In addition, for $j=0$ we set $\cX_0:=\{e_1\}$ with $e_1:= (1, 0, \dots, 0)$, and $\ww_{e_1}:=\omega_d$.

Since the cubature formula \eqref{cubature} is exact for spherical harmonics of degree $\le 2^{j+1}$
we have
$$
\LL_j*\LL_j(x\cdot y) =\int_\SS \LL_j(x\cdot\eta)\LL_j(\eta\cdot y)d\sigma(\eta)
= \sum_{\xi\in\cX_j} \ww_\xi \LL_j(x\cdot \xi)\LL_j(\xi\cdot y),
$$
which allows to discretize $f=\sum_{j=0}^\infty \LL_j*\LL_j*f$ and obtain
\begin{equation}\label{discretize}
f=\sum_{j=0}^\infty \sum_{\xi\in\cX_j}\langle f,\psi_\xi \rangle \psi_\xi,
\quad \forall f\in\cS' \quad\hbox{(convergence in $\cS'$)},
\end{equation}
\begin{equation}\label{def-psi-xi}
\psi_\xi(x):=\ww_\xi^{1/2} \LL_j(\xi\cdot x), \quad \xi\in\cX_j, \;j\ge 0.
\end{equation}
We set $\cX:=\cup_{j\ge 0}\cX_j$ assuming that equal points from different sets $\cX_j$
are distinct points in $\cX$ so that $\cX$ can be used as an index set.
This completes the construction of the system
$\Psi=\{\psi_\xi\}_{\xi\in\cX}$.

Observe that the frame elements $\{\psi_\xi\}$ are not only band limited,
but also have excellent localization on $\SS$.
From the properties of $\aa$ and Theorem~\ref{thm:localization} and \eqref{cubature_w}
it follows that (see also \cite{NPW1, NPW2}) for any $M>0$
\begin{equation}\label{local-needlet-0}
|\psi_\xi(x)| \le \tc_4 2^{(j-1)(d-1)/2}(1+2^{j-1}\rho(x, \xi))^{-M},
\quad x\in\SS,\; \xi\in\cX_j, \; \;j\ge 0,
\end{equation}
where $\tc_4>0$ is a constant depending only on $d$, $M$, $c_7$ and $\aa$.
Moreover, the localization of $\psi_\xi$ can be improved to sub-exponential
as shown in \cite[Theorem~5.1]{IPX}.

The normalization factor $\ww_\xi^{1/2}$ in \eqref{def-psi-xi} makes all $\psi_\xi$ essentially normalized in $L^2(\SS)$,
i.e. $\|\psi_\xi\|_{L^2(\SS)}\sim 1$.
In what follows we only need the lower bound estimate
\begin{equation}\label{psi-xi-norm2}
\|\psi_\xi\|_{L^2(\SS)} \ge \tc_5, \quad \forall\xi\in\cX_j, \;j\ge 0,
\end{equation}
with a constant $\tc_5$ depending only on $d$, $M$, $c_7$ and $\aa$.
Inequality \eqref{psi-xi-norm2} follows from \eqref{def-psi-xi}, \eqref{cubature_w}, \eqref{def-Psi-j},
the properties of $\aa$, and $\int_\SS Z_k^2(\xi\cdot x)d\sigma(x)=Z_k(1)\sim k^{d-2}$.

\smallskip

We next define the Besov and Triebel-Lizorkin sequence spaces $\bb_p^{sq}$ and $\ff_p^{sq}$ associated to $\cX$.

\begin{defn}\label{def:B}
Let $s\in \RR$, $0<p,q\le\infty$. Then $\bb_p^{sq}:=\bb_p^{sq}(\cX)$
is defined as the space of all complex-valued sequences
$h:=\{h_{\xi}\}_{\xi\in \cX}$ such that
\begin{equation}\label{def-b-space}
\|h\|_{\bb_p^{sq}} :=
\Big(\sum_{j=0}^\infty
\Big[2^{j[s+(d-1)(1/2-1/p)]}
\Big(\sum_{\xi \in \cX_j}
|h_\xi|^p\Big)^{1/p}\Big]^q\Big)^{1/q}<\infty
\end{equation}
with the usual modification when $p=\infty$ or  $q=\infty$.
\end{defn}

\begin{defn}\label{def:TL}
Let $s\in \RR$, $0<p<\infty$, and $0<q\le\infty$. Then $\ff_p^{sq}:=\ff_p^{sq}(\cX)$
is defined as the space of all complex-valued sequences
$h:=\{h_{\xi}\}_{\xi\in \cX}$ such that
\begin{equation}\label{def-f-space}
\|h\|_{\ff_p^{sq}} :=\Big\|\Big(\sum_{\xi\in\cX}
\big[|B_\xi|^{-s/(d-1)-1/2}
|h_{\xi}|\ONE_{B_\xi}(\cdot)\big]^q\Big)^{1/q}\Big\|_{L^p} <\infty
\end{equation}
with the usual modification for $q=\infty$.
Here $B_\xi:=B(\xi,\gamma 2^{-j+1})$, $\xi\in \cX_j$, where $\gamma$ is used in the selection of $\cX_j$,
$|B_\xi|$ is the measure of $B_\xi$
and
$\ONE_{B_\xi}$ is the characteristic function of $B_\xi$.
\end{defn}

\begin{rem}\label{rem:TL}
The replacement of $B_\xi=B(\xi,\gamma 2^{-j+1})$ in Definition~\ref{def:TL}
with $B(\xi,\gamma 2^{-j})$ or with the disjoint partition sets $\nA_\xi$ produces equivalent quasi-norms.
This immediately follows from the vector-valued maximal inequality as observed in \cite[Proposition 2.7]{FJ}.
\end{rem}

The main result here asserts that $\{\psi_\xi\}_{\xi\in\cX}$ is a self-dual real-valued frame for
Besov and Triebel-Lizorkin spaces on the sphere.
To state this result we introduce the following
{\em analysis} and {\em synthesis} operators:
\begin{equation}\label{def-oper-S-T}
S_{\psi}: f\mapsto \{\langle f, \psi_\xi\rangle\}_{\xi\in\cX},
\quad
T_\psi: \{h_\xi\}_{\xi\in\cX} \mapsto \sum_{\xi\in\cX}h_\xi\psi_\xi.
\end{equation}

\begin{thm}\label{thm:F-Bnorm-equivalence}
Let $s\in \RR$ and $0< p, q< \infty$.

\smallskip

\noindent
$(a)$
The operators
$S_{\psi}: \cB_p^{s q} \to \bb_p^{s q}$ and
$T_\psi: \bb_p^{s q} \to \cB_p^{s q}$ are bounded, and
$T_\psi\circ S_{\psi}= I$ on $\cB_p^{s q}$.
Hence,
if $f\in \cS'$, then $f\in \cB_p^{sq}$ if and only if
$\{\langle f,\psi_\xi\rangle\}_{\xi \in \cX}\in \bb_p^{sq}$,
and
\begin{equation}\label{B-disc-calderon}
f =\sum_{\xi\in \cX}\langle f, \psi_\xi\rangle \psi_\xi
\quad\mbox{and}\quad
\|f\|_{\cB_p^{sq}}
\sim  \|\{\langle f,\psi_\xi\rangle\}\|_{\bb_p^{sq}}.
\end{equation}

\smallskip

\noindent
$(b)$
The operators
$S_{\psi}: \cF_p^{s q} \to \ff_p^{s q}$ and
$T_\psi: \ff_p^{s q} \to \cF_p^{s q}$ are bounded, and
$T_\psi\circ S_{\psi}= I$ on $\cF_p^{s q}$.
Hence,
if $f\in \cS'$, then $f\in \cF_p^{sq}$ if and only if
$\{\langle f,\psi_\xi\rangle\}_{\xi \in \cX}\in \ff_p^{sq}$, and
\begin{equation}\label{F-disc-calderon}
f =\sum_{\xi\in \cX}\langle f, \psi_\xi\rangle \psi_\xi
\quad \mbox{and}\quad
\|f\|_{\cF_p^{sq}}
\sim  \|\{\langle f,\psi_\xi\rangle\}\|_{\ff_p^{sq}}.
\end{equation}
The convergence in $(\ref{B-disc-calderon})$ and $(\ref{F-disc-calderon})$ is unconditional
in $\cB_p^{sq}$ and $\cF_p^{sq}$, respectively.
\end{thm}
For details and proofs, see \cite[Theorems 4.5 and 5.5]{NPW2}.

\begin{rem}\label{rem:unif-bounds}
A careful examination of the proofs in \cite{NPW2} shows that the operators $S_{\psi}$ and $T_{\psi}$
are uniformly bounded on the respective spaces with parameters
$$
(s, p, q)\in \QQ(A), \quad \mbox{for fixed}~A >1,
$$
where $\QQ(A)$ is the index set defined in $(\ref{indices-1})$,
that is, all constants that appear in the equivalences in Theorem~\ref{thm:F-Bnorm-equivalence}
depend only on $A$, $d$, and $\varphi$, if $(s, p, q)\in \QQ(A)$.
In fact, the only nontrivial source of constants is the maximal inequality $(\ref{max-ineq})$,
however, as seen in $(\ref{max-const})$ these constant are compatible with the definition
of $\QQ(A)$ in $(\ref{indices-1})$.
\end{rem}
The above observation will be needed for the construction of new frames below.

\begin{rem}\label{rem:frame}
In general, one normally constructs and works with a pair of dual frames
$\{\psi_\xi\}_{\xi\in\cX}$, $\{\tilde\psi_\xi\}_{\xi\in\cX}$ on $\SS$,
see \cite{NPW2}.
In the construction presented above  we consider the case when $\tilde\psi_\xi=\psi_\xi$ for simplicity.
\end{rem}

Some embeddings between Besov or Triebel-Lizorkin spaces will be needed.

\begin{prop}\label{embed}
Assume $s,s_0,s_1\in\RR$ and let $0<p,p_0,p_1,q,q_0,q_1\le\infty$ in the case of Besov spaces
and $0<p,p_0,p_1<\infty$, $0<q,q_1,q_2\le\infty$ in the case of Triebel-Lizorkin spaces.
The following continuous embeddings are valid:
\begin{equation}\label{eq:2}
\cB^{s_0q_0}_p\subset \cB^{s_1q_1}_p,~~\cF^{s_0q_0}_p\subset \cF^{s_1q_1}_p,
\quad \hbox{if}\;\;
s_0=s_1,~q_0\le q_1~~\mbox{or}~~s_0>s_1,~\forall q_0, q_1;
\end{equation}
\begin{equation}\label{eq:4}
\cB^{sq}_{p_0}\subset \cB^{sq}_{p_1},
\quad \cF^{sq}_{p_0}\subset \cF^{sq}_{p_1},
\quad  \hbox{if}\;\;
p_0\ge p_1;
\end{equation}
\begin{equation}\label{eq:11}
\cB^{s_0q}_{p_0}\subset \cB^{s_1q}_{p_1},
\quad  \hbox{if}\;\;
s_0\ge s_1,~s_0-\frac{d-1}{p_0}=s_1-\frac{d-1}{p_1};
\end{equation}
\begin{equation}\label{eq:12}
\cF^{s_0q_0}_{p_0}\subset \cF^{s_1q_1}_{p_1},
\quad  \hbox{if}\;\;
s_0> s_1,~s_0-\frac{d-1}{p_0}=s_1-\frac{d-1}{p_1},~\forall q_0, q_1;
\end{equation}
\begin{equation}\label{eq:8b}
\cB^{sq}_{p}\subset \cF^{sq}_{p}\subset \cF^{sp}_{p}=\cB^{sp}_{p},
\quad  \hbox{if}\;\;
q<p;
\end{equation}
\begin{equation}\label{eq:8c}
\cB^{sp}_{p}=\cF^{sp}_{p}\subset \cF^{sq}_{p}\subset \cB^{sq}_{p},
\quad  \hbox{if}\;\;
p<q.
\end{equation}
\end{prop}
\begin{proof}
The proofs of embeddings \eqref{eq:2}, \eqref{eq:4}, \eqref{eq:8b}, and \eqref{eq:8c} are easy and will be omitted.

Embedding \eqref{eq:11} is an immediate consequence of the Nikolski inequality for spherical polynomials.
Indeed, by Definition~\ref{def:B-F-spaces} it follows that $\Phi_j*f$ is a spherical polynomial of degree $\le 2^j$,
i.e. $\Phi_j*f \in \Pi_{2^j}$. Then by the Nikolski inequality, see e.g. \cite[Theorem~5.5.1]{DX},
\begin{equation}\label{Nikolski}
\|\Phi_j*f\|_{L^{p_1}}\le c2^{j(1/p_0-1/p_1)(d-1)}\|\Phi_j*f\|_{L^{p_0}},
\quad p_0\le p_1,
\end{equation}
and \eqref{eq:11} follows readily.

The proof of embedding \eqref{eq:12} relies on the Nikolski inequality \eqref{Nikolski}
and can be carried out along the lines of the proof of the same embedding result in the classical
case on $\RR^n$ from \cite[Theorem~2.1]{Jawerth}, see also \cite[Theorem~2.7.1]{Triebel-1}.
We omit the details.
\end{proof}

\section{Construction of frames by small perturbation}\label{s3}

Here we present the small perturbation method for construction of frames, developed in \cite{DKKP}.
Special attention is paid to the dependence of the numerous constants on the parameters of the distribution spaces involved.

\subsection{Setting and conditions on the old frame}\label{set-up}

As in Section~\ref{s2_2} we denote by $\cS:=C^\infty(\SS)$ the set of all test functions
on $\SS$ and let $\cS'$ be its dual.
We assume that $\YY$ is a collection of quasi-Banach spaces $\fB=\fB(\SS)\subset\cS'$
of distributions on $\SS$ with quasi-norms $\|\cdot\|_\fB$,
which are continuously embedded in $\cS'$,
i.e. there exist $m=m(\fB)\in\NN$ and $C=C(\fB)>0$
such that $|\langle f, \phi\rangle|\le C\|f\|_\fB P_m(\phi)$ for all $f\in\fB$, $\phi\in\cS$.
Also we assume that $\cS$ is a dense subset of each $\fB\in\YY$.

\begin{sloppypar}
Furthermore, we assume that there exists a collection $\YY_d$ of quasi-Banach complex-valued sequence spaces $\fb=\fb(\cX)$
with quasi-norms $\|\cdot\|_\fb$,
such that every $\fB\in\YY$ is associated with a space $\fb\in \YY_d$.
We assume that the constants in the quasi-triangle inequalities
for the quasi-Banach spaces in $\YY$ and $\YY_d$ are uniformly bounded,
i.e. there exists a constant $C_1=C_1(\YY,\YY_d)$ such that
\begin{equation}\label{quasi-triangle}
\begin{split}
\|f_1+f_2\|_\fB&\le C_1(\|f_1\|_\fB+\|f_2\|_\fB), \quad\forall f_1,f_2\in\fB,~\forall \fB\in\YY;\\
\|h_1+h_2\|_\fb&\le C_1(\|h_1\|_\fb+\|h_2\|_\fb), \quad\forall h_1,h_2\in\fb,~\forall \fb\in\YY_d.
\end{split}
\end{equation}
A popular version of the Aoki-Rolewicz theorem states
(see e.g. \cite[Lemma 3.10.1]{BL}) that for any quasi-Banach space $\fB$ with a quasi-norm $\|\cdot\|_\fB$
satisfying the quasi-triangle inequalities with constant $C_1$ there exists a \emph{norm} $\|\cdot\|^*$ on $\fB$, such that
\begin{equation}\label{quasi-triangle2}
\|f\|^*\le \|f\|_\fB^\tau\le 2\|f\|^*, \quad\forall f\in\fB,\quad \mbox{where}~\tau=\ln 2/(\ln 2+\ln C_1)\le 1.
\end{equation}
\end{sloppypar}

Targeted application of this construction is to
the Besov and Triebel-Lizorkin function spaces
introduced in Section~\ref{s2_2}
and the corresponding sequence spaces
introduced in Section~\ref{subsec:frame-SS}.
For the Besov spaces the sets $\YY$ and $\YY_d$ are given by
$$
\YY=\big\{\cB^{sq}_p(\SS): (s, p, q)\in\QQ(A)\big\}
\quad{ and }\quad
\YY_d=\big\{\bb^{sq}_p(\cX): (s, p, q)\in\QQ(A)\big\},
$$
where $A>1$ is fixed and
$\QQ(A)$ is introduced in (\ref{indices-1}).
A similar observation is valid for the Triebel-Lizorkin spaces $\cF^{sq}_p(\SS)$ and $\ff^{sq}_p(\cX)$.
\smallskip

\noindent
{\bf The old frame.}
We stipulate the existence of a pair of dual frames $\{\psi_\xi\}_{\xi\in\cX}$, $\{\tilde\psi_\xi\}_{\xi\in\cX}$ for all $\fB\in\YY$
such that $\psi_\xi, \tilde\psi_\xi \in \cS$, where $\cX$ is a countable index set,
with the following properties:

\smallskip

\noindent
{\bf A1.}
The {\em analysis} and {\em synthesis} operators $S_{\psi}$, $S_{\tilde\psi}$
and $T_{\psi}$, $T_{\tilde\psi}$ from \eqref{def-oper-S-T} have the properties:

(a)
The operators $S_{\psi}, S_{\tilde\psi}: \fB\mapsto \fb$ are bounded.

(b)
For any sequence $h=\{h_\xi\}_{\xi\in\cX}\in \fb$
the series $\sum_{\xi\in\cX}h_\xi\psi_\xi$ and $\sum_{\xi\in\cX}h_\xi\tilde\psi_\xi$ converge unconditionally in $\fB$ and
 $T_\psi, T_{\tilde\psi}: \fb\to \fB$ are bounded.

It is assumed that the norms of the operators $S_{\psi}$, $S_{\tilde\psi}$
and $T_{\psi}$, $T_{\tilde\psi}$
are uniformly bounded relative to $\fB\in\YY$ and $\fb\in\YY_d$ by a constant $C_2=C_2(\YY,\YY_d,\{\psi_\xi\})>1$.
Thus, for any $\fB\in \YY$ and $f\in \fB$ we have
\begin{equation}\label{equiv_frame}
\begin{split}
C_2^{-1}\|f\|_{\fB} \le \|S_{\psi}f\|_\fb=\|\{\langle f,\psi_{\xi}\rangle\}\|_\fb\le C_2 \|f\|_{\fB},\\
C_2^{-1}\|f\|_{\fB} \le \|S_{\tilde\psi}f\|_\fb=\|\{\langle f,\tilde\psi_{\xi}\rangle\}\|_\fb\le C_2 \|f\|_{\fB}.
\end{split}
\end{equation}

\smallskip

\noindent
{\bf A2.} We have $T_{\tilde\psi}S_{\psi}=T_{\psi}S_{\tilde\psi}=I$ in $\cB$, i.e. for any $f\in \fB$
\begin{equation}\label{A1}
f=\sum_{\xi\in \cX} \langle f, \tilde\psi_{\xi}\rangle\psi_{\xi}
=\sum_{\xi\in \cX} \langle f, \psi_{\xi}\rangle\tilde\psi_{\xi},
\end{equation}
where the two series converge unconditionally in $\fB$ and hence in $\cS'$.

Note that the compositions $S_{\psi}T_{\tilde\psi}$, $S_{\tilde\psi}T_{\psi}$ are projectors due to
\begin{equation*}
(S_{\psi}T_{\tilde\psi})^2=S_{\psi}(T_{\tilde\psi}S_{\psi})T_{\tilde\psi}=S_{\psi}IT_{\tilde\psi}=S_{\psi}T_{\tilde\psi}.
\end{equation*}

\smallskip

\noindent
{\bf A3.} In addition, we assume that each $\fb\in\YY_d$ obeys the conditions:
\begin{enumerate}
\item[(a)]
For any sequence $\{h_\xi\}_{\xi\in\cX}\in\fb$ one has $\|\{h_\xi\}\|_\fb = \|\{|h_\xi|\}\|_\fb$.
\item[(b)]
If the sequences $\{h_\xi\}_{\xi\in\cX}, \{g_\xi\}_{\xi\in\cX}\in\fb$
and $|h_\xi|\le |g_\xi|$ for $\xi\in\cX$, then $\|\{h_\xi\}\|_\fb \le \|\{g_\xi\}\|_\fb$.
\item[(c)]
Compactly supported sequences belong to $\fb$ and are dense in $\fb$.
\end{enumerate}

\smallskip

Note that
conditions {\bf A3} (b)-(c) imply condition {\bf A3} (ii) in \cite{DKKP}.

As a consequence of {\bf A1} we obtain that the operator $A:=S_{\psi}T_{\psi}$ with matrix
\begin{equation*}
\bA:=\{a_{\xi,\eta}\}_{\xi, \eta \in\cX}, \quad
a_{\xi, \eta}:=\langle\psi_\eta,\psi_\xi\rangle
\end{equation*}
is uniformly bounded on the sequence spaces $\fb\in \YY_d$, i.e.
\begin{equation}\label{oper-A}
\|A\|_{\fb\mapsto \fb}\le C_3:=C_2^2,\quad \forall \fb\in \YY_d.
\end{equation}

\subsection{Construction of new frames}\label{subsec:new-frame}

We next construct a pair of dual frames
$\{\theta_\xi\}_{\xi\in\cX}$,
$\{\tilde\theta_\xi\}_{\xi\in\cX}$ for all spaces $\fB\in\YY$,
where $\cX$ is the index set from above.

For a system $\{\theta_\xi\}_{\xi\in\cX}\subset\cap\{\fB : \fB\in\YY\}$ of real-valued functions $\theta_\xi\in\fB$, $\xi\in\cX$,
we define the matrices
\begin{equation*}
\begin{aligned}
&\bB:=\{b_{\xi,\eta}\}_{\xi, \eta \in\cX}, \quad
b_{\xi, \eta}:=\langle\theta_\eta,\psi_\xi\rangle,\\
&\bD:=\{d_{\xi,\eta}\}_{\xi,\eta\in\cX}, \quad
d_{\xi,\eta}:=\langle\psi_\eta-\theta_\eta,\psi_\xi\rangle.
\end{aligned}
\end{equation*}
The only condition that we require when constructing $\{\theta_\xi\}_{\xi\in\cX}$ is that the operator
$$D=S_{\psi}T_{\psi}-S_{\psi}T_{\theta}, \;\; D:\fb\to\fb,$$
with matrix $\bD$, defined by $(D h)_\xi=\sum_{\eta\in\cX} d_{\xi,\eta} h_\eta$,
has a sufficiently small norm uniformly for all $\fb\in\YY_d$.
More precisely we assume that
\begin{equation}\label{oper-eps}
\|D\|_{\fb\mapsto \fb}\le \epsilon:=\frac{(1-2^{-\tau})^{1/\tau}}{2C_1C_2^4 2^{1/\tau}},\quad\forall \fb\in\YY_d,
\end{equation}
with $\tau$ given in \eqref{quasi-triangle2}, where $C_1$ is the constant from \eqref{quasi-triangle} and
$C_2$ is the constant in \eqref{equiv_frame}.
For the operator $B$ with matrix $\bB$ we have
$B=A-D=S_{\psi}T_{\theta}$ and hence by \eqref{quasi-triangle}, \eqref{oper-A}, and \eqref{oper-eps}
it follows that $B$ is uniformly bounded on $\YY_d$,
more precisely,
\begin{equation}\label{est-B}
\|B\|_{\fb\mapsto \fb}\le C_4, \quad\forall \fb\in\YY_d,
\end{equation}
with constant $C_4= C_1(C_3+\epsilon)$.

Condition \eqref{oper-eps} will be sufficient to show that $\{\theta_\xi\}_{\xi\in\cX}$ is a frame for all spaces $\fB\in\YY$
and to construct its dual frame $\{\tilde\theta_\xi\}_{\xi\in\cX}$.
To this end we introduce the operator:
\begin{equation}\label{oper-T}
Tf :=\sum_{\xi\in\cX} \langle f, \tilde\psi_\xi\rangle \theta_\xi,
\quad f\in \fB.
\end{equation}

The next three lemmas will be instrumental in the construction of $\{\tilde\theta_\xi\}_{\xi\in\cX}$.
They are direct adaptations of Lemmas 3.1 - 3.3 in \cite{DKKP}; we omit their proofs.

\begin{lem}\label{lem:T-bounded}
The operators $T_\theta$, defined in \eqref{def-oper-S-T} with $\theta_\xi$ in the place of $\psi_\xi$,
and $T$ are well defined and uniformly bounded,
that is,
\begin{equation}\label{oper-T-bound}
\begin{split}
\|T_\theta h\|_\fB \le C_2C_4\|h\|_\fb,
\quad \forall h\in \fb,~\forall\fb\in\YY_d;\\
\|Tf\|_\fB \le C_2^2 C_4\|f\|_\fB,
\quad \forall f\in \fB,~\forall\fB\in\YY;
\end{split}
\end{equation}
where $C_2$ is from \eqref{equiv_frame} and $C_4$ is from \eqref{est-B}.
Furthermore, the series in \eqref{def-oper-S-T} and \eqref{oper-T} converge unconditionally in $\fB$
and hence in~$\cS'$.
\end{lem}

\smallskip

The fact the operator $T$ is invertible plays a key role in this construction.

\begin{lem}\label{lem:T-invert}
If $(\ref{oper-eps})$ is satisfied, then
\begin{equation}\label{T-identity}
\|I-T\|_{\fB\mapsto\fB}
\le C_2^2\epsilon=\frac{(1-2^{-\tau})^{1/\tau}}{2C_1C_2^2 2^{1/\tau}} \le \frac{1}{2},
\quad \forall \fB\in\YY,
\end{equation}
and hence
$T^{-1}$ exists and
\begin{equation}\label{T-inverse}
\|T^{-1}\|_{\fB\mapsto\fB} \le C_5:=\frac{2^{1/\tau}}{(1-2^{-\tau})^{1/\tau}},
\quad \forall \fB\in\YY,
\end{equation}
where $\tau$ is from \eqref{quasi-triangle2}.
\end{lem}

\begin{lem}\label{lem:T-invert-psi}
\begin{sloppypar}
Assume $(\ref{oper-eps})$ holds. Then the operator
$H$ with matrix $\bH:=\{\langle T^{-1}\psi_\eta, \tilde\psi_\xi\rangle\}_{\xi, \eta\in\cX}$
is uniformly bounded on $\fb\in\YY_d$, i.e.
\end{sloppypar}
\begin{equation}\label{norm_H}
\|H\|_{\fb\mapsto\fb} \le C_6:=C_2^2C_5,\quad \forall\fb\in\YY_d.
\end{equation}
\end{lem}

The operators from the previous three lemmas can be written as $T_\theta=T_{\tilde\psi} B$,
$T=T_{\theta}S_{\tilde\psi}=T_{\tilde\psi}BS_{\tilde\psi}$,
$I-T=T_{\tilde\psi}DS_{\tilde\psi}$, $H=S_{\tilde\psi}T^{-1}T_{\psi}$.

\smallskip

\noindent
{\bf Construction of the dual frame \boldmath $\{\tilde\theta_\xi\}$.}
For any $\xi\in\cX$ we define the linear functional $\tilde\theta_\xi$ by
\begin{equation}\label{def-f-dual}
\tilde\theta_\xi(f)=\langle f, \tilde\theta_\xi\rangle
:= \sum_{\eta\in\cX}
\langle T^{-1}\psi_\eta, \tilde\psi_\xi\rangle \langle f, \tilde\psi_\eta\rangle
\quad \hbox{for}\;\; f\in \fB,\;\; \fB\in \YY.
\end{equation}
Lemma~\ref{lem:T-invert-psi} and {\bf A1} imply that for any $f\in\fB$, $\fB\in\YY$,
\begin{equation}\label{analysis-bounded}
\|\{\langle f, \tilde\theta_\xi\rangle\}\|_{\fb}
\le \|H\|_{\fb\mapsto \fb}\|\{\langle f, \tilde\psi_\eta\rangle\}\|_{\fb}\le C_6C_2\|f\|_\fB.
\end{equation}
Denote
$\ONE_\xi:=\{\delta_{\xi\eta}\}_{\eta\in\cX}$.
Then $\ONE_\xi\in\fb$ by {\bf A3} (c) and $\|\ONE_\xi\|_\fb>0$ because $\fb$ is a~quasi-normed space.
Now, condition {\bf A3} (b) and inequality \eqref{analysis-bounded} imply
\begin{equation*}
|\tilde\theta_\xi(f)|=|\langle f, \tilde\theta_\xi\rangle|
= \frac{1}{\|\ONE_\xi\|_\fb}\|\langle f, \tilde\theta_\xi\rangle\ONE_\xi\|_{\fb}
\le\frac{1}{\|\ONE_\xi\|_\fb}\|\{\langle f, \tilde\theta_\xi\rangle\}\|_{\fb}
\le \frac{C_6C_2}{\|\ONE_\xi\|_\fb}\|f\|_\fB,
\end{equation*}
i.e. $\tilde\theta_\xi$ ($\xi\in\cX$) is a bounded linear functional on every $\fB\in\YY$.

Also, for any $f\in \fB$ by Lemma~\ref{lem:T-invert} $T^{-1}f\in \fB$ and
using Lemma~\ref{lem:T-bounded}
\begin{equation}\label{identity}
f=T(T^{-1}f) = \sum_{\xi\in\cX} \langle T^{-1}f, \tilde\psi_\xi\rangle\theta_\xi.
\end{equation}
Furthermore,
from the fact that $T^{-1}$ is a bounded operator on $\fB$
and (\ref{A1}) it follows that for any $f\in\fB$
$$
T^{-1}f = \sum_{\eta\in\cX}\langle f, \tilde\psi_\eta\rangle T^{-1}\psi_\eta,
$$
where the series converges unconditionally in $\fB$ and hence in $\cS'$.
This and the fact that $\tilde\psi_\xi\in\cS$ imply
\begin{equation}\label{inner-prod}
\langle T^{-1}f, \tilde\psi_\xi \rangle
= \sum_{\eta\in\cX}\langle T^{-1}\psi_\eta, \tilde\psi_\xi \rangle\langle f, \tilde\psi_\eta\rangle
= \langle f, \tilde\theta_\xi\rangle.
\end{equation}
Here the series converges unconditionally and hence absolutely
because of the unconditional convergence of the former series.
From \eqref{identity}--\eqref{inner-prod} it follows that
\begin{equation}\label{frame-rep}
f = \sum_{\xi\in\cX} \langle f, \tilde\theta_\xi\rangle\theta_\xi,
\quad f\in \fB,
\end{equation}
where $\langle f, \tilde\theta_\xi\rangle$ is defined in (\ref{def-f-dual}) and
the convergence is unconditional in $\fB$.

\smallskip

The following theorem shows that $\{\theta_\xi\}$, $\{\tilde\theta_\xi\}$
is a pair of dual frames for all spaces $\fB\in \YY$ if $\epsilon$ from \eqref{oper-eps} is sufficiently small.

\begin{thm}\label{thm:new-frames}
Let $\{\psi_\xi\}$, $\{\tilde\psi_\xi\}$ be a pair of dual old frames
for all $\fB\in\YY$ satisfying conditions {\bf A1}--{\bf A3} in Subsection~\ref{set-up}.
Assume $\{\theta_\xi\}_{\xi\in\cX}\subset\cap\{\fB : \fB\in\YY\}$ satisfies \eqref{oper-eps}
and $\{\tilde\theta_\xi\}$ is defined as in \eqref{def-f-dual}.
Then the analysis operator $S_{\tilde\theta}:\fB\to\fb$
and the synthesis operator $T_\theta: \fb\to\fB$ are uniformly bounded for $\fB\in\YY$, $\fb\in\YY_d$.
Furthermore, $T_\theta S_{\tilde\theta}=I$ on $\fB$, i.e. for any $f\in \fB$, $\fB\in\YY$, we have
\begin{equation}\label{frame-repr}
f= \sum_{\xi\in\cX} \langle f, \tilde\theta_\xi\rangle\theta_\xi,
\end{equation}
where the convergence is unconditional in $\fB$,
and
\begin{equation}\label{norm-equiv}
\|f\|_\fB \le C_7\|\{\langle f, \tilde\theta_\xi\rangle\}\|_\fb,\quad
\|\{\langle f, \tilde\theta_\xi\rangle\}\|_\fb\le C_8\|f\|_\fB
\end{equation}
with $C_7=2C_1C_2$ and $C_8=C_2C_5$.
\end{thm}

For a proof of Theorem~\ref{thm:new-frames} see the proof of Theorem~3.5 in \cite{DKKP}.

The main assumption in constructing the frames $\{\theta_\xi\}_{\xi\in\cX}$, $\{\tilde\theta_\xi\}_{\xi\in\cX}$
is the operator norm condition (\ref{oper-eps}).
A standard tool for evaluating the operator norms in sequence spaces is the following monotonicity lemma.

\begin{lem}\label{lem:0.3}
Assume the quasi-norm in $\fb(\cX)$ satisfies conditions {\bf A3} $(a)$--$(b)$ in Subsection~\ref{set-up}.
If the entries of two matrices $\bF=\{f_{\xi,\eta}\}_{\xi,\eta\in\cX}, \bG=\{g_{\xi,\eta}\}_{\xi,\eta\in\cX}$
are related by $|f_{\xi,\eta}|\le g_{\xi,\eta}$, $\xi,\eta\in\cX$, then the respective operators $F$, $G$ are related by
\begin{equation}\label{eq:13}
	\|F\|_{\fb\to\fb}\le \|G\|_{\fb\to\fb}.
\end{equation}
\end{lem}

\begin{proof}
The relation between $\bF$ and $\bG$ implies
\begin{equation*}
	|(\bF h)_\xi|=|\sum_{\eta\in\cX}h_\eta f_{\xi,\eta}|\le \sum_{\eta\in\cX}|h_\eta||f_{\xi,\eta}|
	\le \sum_{\eta\in\cX}|h_\eta|g_{\xi,\eta} = (\bG |h|)_\xi.
\end{equation*}
Now, (b) and (a) of {\bf A3} imply that for every $h\in \fb(\cX)$
\begin{equation*}
	\|\bF h\|_\fb \le \|\bG |h|\|_\fb
	\le \|\bG\|_{\fb\to\fb} \| |h|\|_\fb = \|\bG\|_{\fb\to\fb} \| h\|_\fb,
\end{equation*}
which implies \eqref{eq:13}.
\end{proof}

\subsection{Almost diagonal operators}\label{s3_4}

In the next two sections we shall apply the small perturbation method described above for construction of new frames
for the Besov spaces from
$$
\YY=\big\{\cB^{sq}_p(\SS): (s, p, q)\in\QQ(A)\big\}
$$
for an arbitrary fixed $A>1$
as well as for the respective collection of Triebel-Lizorkin spaces $\cF^{sq}_p(\SS)$.
All spaces from $\YY$ satisfy \eqref{quasi-triangle} with $C_1=C_1(A)$.
On~account of \thmref{thm:F-Bnorm-equivalence} and \remref{rem:unif-bounds}
the frame $\Psi=\{\psi_\xi\}_{\xi\in\cX}$ is real-valued, self-dual, i.e. $\tilde\psi_\xi=\psi_\xi$, and
conditions {\bf A1}--{\bf A2} in Subsection~\ref{set-up} are satisfied
with a constant $C_2=C_2(d,A,\{\psi_\xi\})$.
Conditions {\bf A3} are trivially satisfied for the sequence spaces
$\bb^{sq}_p, \ff^{sq}_p$ with $0<p,q<\infty$, $s\in\RR$, and for $\ell^p$, $0<p<\infty$, as well.
It remains to establish sufficient conditions for verifying the operator norm bound \eqref{oper-eps}.
To this end, using \lemref{lem:0.3} one can compare the operator matrix elements
with the elements of an appropriate \emph{almost diagonal matrix} (cf. \cite{FJ, KP1}).

The almost diagonal matrices we shall use are
$\Omega_{K,M}:=\{\omega_{\xi,\eta}^{(K,M)}\}_{\xi,\eta\in\cX}$
with entries
\begin{equation}\label{eq:omega_xi_eta}
\omega_{\xi,\eta}^{(K,M)}
:=\left(\frac{\min\{N_\xi,N_\eta\}}{\max\{N_\xi,N_\eta\}}\right)^{K+(d-1)/2}
\frac{1}{\left(1+\min\{N_\xi,N_\eta\}\rho(\xi,\eta)\right)^M},
\end{equation}
where $N_\xi:= 2^{j-1}$ for $\xi\in\cX_j$, $j\ge 0$.
Other (non-symmetric) examples of almost diagonal matrices with index set $\cX$ are given in \cite[Definition 3.9]{KP1}.

We next show that under
appropriate conditions on $K$ and $M$ the operator $\Omega$ with matrix $\Omega_{K,M}$
is bounded on $\bb^{sq}_p$ and $\ff^{sq}_p$.
In the following, we shall use the notation $\cJ:=(d-1)/\min\{1,p\}$ in the case of $\bb$-spaces
and $\cJ:=(d-1)/\min\{1,p, q\}$ in the case of $\ff$-spaces.

\begin{thm}\label{thm:almost-diag}
Let $s\in\RR$, $0<p, q<\infty$.
For a fixed $\delta\in(0,1]$ assume that $K,M\in\NN$ satisfy
\begin{equation}\label{cond-KM}
K\ge \max\{s,\cJ-s-d+1\}+\delta\quad\hbox{and}\quad M\ge \cJ+\delta.
\end{equation}
Then the operator $\Omega$ with matrix $\Omega_{K,M}$ is bounded on $\bb^{sq}_p$ and on $\ff^{sq}_p$.
More precisely, there exists a constant $C_9>0$ such that
\begin{equation*}
\|\Omega h\|_{\bb^{sq}_p} \le C_9\|h\|_{\bb^{sq}_p},\quad \forall h\in\bb^{sq}_p;\qquad
\|\Omega h\|_{\ff^{sq}_p} \le C_9\|h\|_{\ff^{sq}_p},\quad \forall h\in\ff^{sq}_p.
\end{equation*}
Here in the case of $\bb$-spaces the constant $C_9$ can be written in the form
$$
C_9=\Big(\frac{c_{\dag}}{\delta^2}+\Big(\frac{c_{\dag}}{\delta p}\Big)^{2/p-1}\Big)
\Big(\frac{c_{\star}}{\delta}+\Big(\frac{c_{\star}}{\delta q}\Big)^{1/q}\Big),
$$
and in the case of $\ff$-spaces in the form
$$
C_9={c_{\dag}}^{d/\cJ+2|s/(d-1)+1/2|}
\Big(\frac{c_{\star}}{\delta}+\Big(\frac{c_{\star}}{\delta q}\Big)^{1/q}\Big)\frac{\tc_1}{\delta},
$$
where $c_{\dag}$ is a constant depending only on $d$, $c_{\star}$ is an absolute constant,
and $\tc_1$ is the constant from the maximal inequality \eqref{max-const} with $1/t=1/\min\{1,p,q\}+\delta/(d-1)$.
\end{thm}

Observe that if $(s, p, q)\in\QQ(A)$ for some fixed $A>1$ and $\delta=1$
then \thmref{thm:almost-diag} holds with $C_9$ depending only on $d$ and $A$.

To streamline our presentation we defer the
long tedious proof of \thmref{thm:almost-diag} to Section~\ref{appendix}.

In light of \thmref{thm:almost-diag} we next use \propref{prop:5_3} and the localization property \eqref{local-needlet-0} to show that
the scalar products of the elements of the needlet system $\{\psi_\xi\}$ from Subsection~\ref{subsec:frame-SS}
are majorized by the entries of an almost diagonal matrix.

\begin{prop}\label{prop:almost_diag}
For any $K\in\NN_0$ and $M>d-1$ the needlet system $\{\psi_\xi\}$ from Subsection~\ref{subsec:frame-SS} satisfies
\begin{equation}\label{eq:almost_diag_1}
|\langle\psi_\eta,\psi_\xi\rangle|\le C_{10} \omega_{\xi,\eta}^{(K,M)},\quad\forall \xi,\eta\in\cX,
\end{equation}
where $\omega_{\xi,\eta}^{(K,M)}$ is defined in \eqref{eq:omega_xi_eta} and $C_{10}= 2^K c_3\tc_4^2$ with
$c_3$ from \propref{prop:5_3} and $\tc_4$ from \eqref{local-needlet-0}.
\end{prop}
\begin{proof}
Let $\xi\in\cX_j$ and $\eta\in\cX_k$.
From \eqref{def-Psi-j} and \eqref{def-psi-xi} it readily follows that $\langle\psi_\eta,\psi_\xi\rangle=0$ if $|j-k|\ge 2$.
The symmetry of $\omega_{\xi,\eta}^{(K,M)}$
implies that it suffices to consider only the cases $k=j$ and $k=j+1$.

On account of \eqref{local-needlet-0}
condition \eqref{eq:1c} is satisfied for $g=\psi_\xi$  with $x_1=\xi$, $N_1=N_\xi$, and $\kappa_1=\tc_4 N_\xi^{-(d-1)/2}$ and
condition \eqref{eq:2c} is satisfied for $f=\psi_\eta$ with $x_2=\eta$, $N_2=N_\eta$, and $\kappa_2=\tc_4 N_\eta^{-(d-1)/2}$.
Then \propref{prop:5_3} and $N_\eta=N_\xi$ for $k=j$ or $N_\eta=2N_\xi$ for $k=j+1$
show that \eqref{eq:almost_diag_1} is satisfied with $C_{10}= c_3\tc_4^2$ or $C_{10}= 2^K c_3\tc_4^2$, respectively.
\end{proof}

Note that a combination of \propref{prop:almost_diag}, \lemref{lem:0.3}, and \thmref{thm:almost-diag}
readily yields another proof of \eqref{oper-A}.

\begin{cor}\label{cor:almost_diag}
Let $\fb=\bb_p^{sq}$ or $\fb=\ff_p^{sq}$ with $(s, p, q)\in \QQ(A)$ for a fixed $A>1$.
Then $\{\psi_\xi\}$ satisfies \eqref{oper-A}
with a constant $C_3$ depending only on $d$, $A$ and $\varphi$,
namely $C_3=C_9 C_{10}$, where $C_9$ is from \thmref{thm:almost-diag} with
$K = M = \left\lceil Ad\right\rceil$
and $C_{10}$ is from \propref{prop:almost_diag} with the same $K$ and $M$.
\end{cor}
\begin{proof}
As $K$ and $M$ satisfy assumption \eqref{cond-KM} of \thmref{thm:almost-diag} with $\delta=1$,
Proposition~\ref{prop:almost_diag}, \lemref{lem:0.3}, and \thmref{thm:almost-diag} imply that
\eqref{oper-A} is valid with $C_3=C_9 C_{10}$.
\end{proof}

We shall also apply \thmref{thm:almost-diag} in Section~\ref{s6} to show
that condition \eqref{oper-eps} holds for the constructed new Newtonian kernel frame.

\section{Space localization of needlets and Newtonian kernels}\label{s4}

The basic localization property of the needlets is given in \eqref{local-needlet-0}.
In this section we establish some additional localization properties of the needlets introduced in \S\ref{subsec:frame-SS}.
We also introduce the localized kernels developed in \cite{IP2}.
These kernels are linear combinations of shifts of the Newtonian kernel
and will be the building blocks in the construction of Newtonian kernel frames in Section~\ref{s6}.

\subsection{Properties of the needlets}\label{s4_1}
Let $N:=2^{j-1}$, $j\in\NN$, and assume that the integer parameter $K$ is even, i.e. $K\in 2\NN$.
Let $\varphi$ be the $C^\infty[0,\infty)$ function introduced in Subsection~\ref{subsec:frame-SS}.
We define (cf. \eqref{def-Psi-j})
\begin{equation}\label{eq:constr_3}
	\newPsi_N(u):=\Psi_j(u)=\sum_{\kk=0}^{\infty} \varphi\Big(\frac{\kk}{N}\Big)\PP_\kk(u)
	=\sum_{N/2<\kk<2N} \varphi\Big(\frac{\kk}{N}\Big)\PP_\kk(u)
\end{equation}
and
\begin{equation}\label{eq:constr_4}
\newPhi_N(u):=(-1)^{K/2} \sum_{N/2<\kk<2N} \varphi\Big(\frac{\kk}{N}\Big)[\kk(\kk+d-2)]^{-K/2}\PP_\kk(u),
\end{equation}
where $\PP_\kk$ is from (\ref{def-Pk}).
By (\ref{eq:LB1}) it follows that
\begin{equation*}
-\Delta_0\PP_k(\eta\cdot x) =\kk(\kk+d-2)\PP_k(\eta\cdot x),
\end{equation*}
implying
\begin{equation}\label{eq:LB2}
\Delta_0^{K/2}\newPhi_N(\eta\cdot x)= \newPsi_N(\eta\cdot x),\quad \eta, x\in\SS.
\end{equation}
Here $\Delta_0$ is the Laplace-Beltrami operator on $\SS$ (see Subsection~\ref{s1_2}).

Given $\eta\in\SS$ we extend $\newPhi_N(\eta\cdot x)$ and $\newPsi_N(\eta\cdot x)$ by
\begin{equation}\label{eq:constr_5}
	\til{\newPhi}_N(\eta; x):=\newPhi_N\Big(\frac{\eta\cdot x}{|x|}\Big),\quad
	\til{\newPsi}_N(\eta; x):=\newPsi_N\Big(\frac{\eta\cdot x}{|x|}\Big),\quad x\in\RR^d\backslash\{0\}.
\end{equation}
In light of \eqref{eq:laplace_extension} and \eqref{eq:LB2} this implies
\begin{equation}\label{eq:constr_6}
\Delta^{K/2}\til{\newPhi}_N(\eta; x)=\Delta_0^{K/2}\newPhi_N(\eta\cdot x)= \newPsi_N(\eta\cdot x)=\til{\newPsi}_N(\eta; x),
\quad x\in\SS.
\end{equation}

We shall need the following simple claim.

\begin{lem}\label{lem:2.1}
Let $W(x,y):=\displaystyle{\frac{x\cdot y}{|x||y|}}$ for $x,y\in\RR^d\backslash\{0\}$.

{\em (a)} For any $j=1,2, \dots, d$ we have
$$\frac{\partial}{\partial x_j}W(x,y) = P_j(x,y)|x|^{-3}|y|^{-1},\quad x,y\in\RR^d\backslash\{0\},$$
where $P_j$ is a homogeneous polynomial of degree 2 in $x$
and a homogeneous polynomial of degree 1 in $y$ and
\begin{equation}\label{prop-W}
\Big|\frac{\partial}{\partial x_j}W(x,y)\Big| \le 2\rho(x,y),\quad
\Big|\frac{\partial}{\partial y_j}W(x,y)\Big| \le 2\rho(x,y),\quad
x,y\in\SS.
\end{equation}

{\em (b)} For any multi-index $\beta$ with $|\beta|\ge 1$ and any $G\in C^{|\beta|}[-1,1]$ we have the representation
\begin{equation}\label{rep-par-beta}
\partial^\beta_x G(W(x,y))
= \sum_{1\le \nu \le |\beta|}G^{(\nu)}(W(x,y)) R_{\beta,\nu}(x,y)|x|^{-\nu-2|\beta|}|y|^{-\nu}
\end{equation}
with
\begin{equation}\label{rep-par-beta2}
R_{\beta,\nu}(x,y)=\!\sum_{\substack{\mu\\ |\mu|=2\nu-|\beta|}}
\prod_{k=1}^d\Big(|x|^3|y|\frac{\partial}{\partial x_j} W(x,y)\Big)^{\mu_k}
Q_{\beta,\nu,\mu}(x,y),~ \frac{|\beta|}{2} < \nu\le |\beta|,
\end{equation}
where $x,y\in\RR^d\backslash\{0\}$, $R_{\beta,\nu}$, $1\le \nu \le |\beta|$,
is a homogeneous polynomial of degree $|\beta|+\nu$ in  $x$ and a homogeneous polynomial of degree $\nu$ in  $y$, and
 $Q_{\beta,\nu,\mu}$, $|\beta|/2 < \nu\le |\beta|$,
is a homogeneous polynomial of degree $3|\beta|-3\nu$ in  $x$ and a homogeneous polynomial of degree $|\beta|-\nu$ in  $y$.
The coefficients of $R_{\beta,\nu}$ and $Q_{\beta,\nu,\mu}$ are independent of $G$ and
some of the polynomials $Q_{\beta,\nu,\mu}$ are identically equal to zero.

{\em (c)} For any multi indices $\alpha,\beta$ with $|\alpha|=1$, $|\beta|\ge 0$ and
any $G\in C^{|\beta|+1}[-1,1]$ we have the representation
\begin{multline}\label{rep-par-beta_c}
\partial^\alpha_y\partial^\beta_x G(W(x,y))\\
= \sum_{0\le \nu \le |\beta|}G^{(\nu+1)}(W(x,y)) (\partial^\alpha_y W(x,y)) R_{\beta,\nu}(x,y)|x|^{-\nu-2|\beta|}|y|^{-\nu}\\
+ \sum_{1\le \nu \le |\beta|}G^{(\nu)}(W(x,y))
\big[|y|^2\partial^\alpha_yR_{\beta,\nu}(x,y)-\nu y^\alpha R_{\beta,\nu}(x,y)\big]|x|^{-\nu-2|\beta|}|y|^{-\nu-2},
\end{multline}
where $x,y\in\RR^d\backslash\{0\}$, $R_{0,0}\equiv 1$ and $R_{\beta,\nu}$, $1\le \nu \le |\beta|$, are from part (b).
\end{lem}

\begin{proof}
Clearly,
\begin{multline*}
\frac{\partial}{\partial x_j}W(x,y) = |x|^{-3}|y|^{-1} \sum_{\nu\ne j}x_\nu(y_j x_\nu-y_\nu x_j)\\
= |x|^{-3}|y|^{-1} \sum_{\nu\ne j}x_\nu[y_j(x_\nu-y_\nu)-y_\nu(x_j-y_j)]
\end{multline*}
and the first inequality of part (a) follows using the Cauchy-Schwarz inequality and that $|x-y|\le\rho(x,y)$ for $x,y\in\SS$.
The second inequality of part (a) follows from the first one by symmetry.

Part (b) follows by induction on $|\beta|$.
Note that $R_{\beta,\nu}$ is defined recursively by $R_{(0,\dots,0),0}(x,y)\equiv 1$,
$R_{\beta,0}(x,y)\equiv 0$ for $|\beta|\ge 1$, $R_{\beta,|\beta|+1}(x,y)\equiv 0$ for $|\beta|\ge 0$,
and for $|\alpha|=1$, $|\beta|\ge 0$
\begin{multline*}
R_{\beta+\alpha,\nu}(x,y)=(|x|^3|y|\partial^\alpha_x W(x,y))R_{\beta,\nu-1}(x,y)\\
+|x|^2\partial^\alpha_x R_{\beta,\nu}(x,y)-(\nu+2|\beta|)x^\alpha R_{\beta,\nu}(x,y),\quad 1\le\nu\le|\beta|+1.
\end{multline*}

Part (c) follows from part (b) by differentiating \eqref{rep-par-beta} with respect to $y$ for $|\beta|\ge 1$
or trivially for $|\beta|=0$.
\end{proof}

From \lemref{lem:2.1} one easily derives localization estimates for zonal functions.

\begin{lem}\label{lem:2.2}
Let $K,d\in\NN$, $G\in C^K [-1,1]$.
Assume that for some $N\ge 1$ and $M>0$
\begin{equation}\label{eq:2.2.1}
|G^{(\nu)}(u)|\le \frac{\kappa N^{2\nu}}{(1+N\arccos u)^{M+\nu}}, \quad \forall u\in[-1,1],~0\le\nu\le K,
\end{equation}
where $\kappa>0$ is a constant  depending on $K$, $d$, $M$, and $N$.
Then for all $x,y\in\SS$ we have
\begin{equation}\label{eq:2.2.2}
\Big|\partial^\beta_x G\Big(\frac{y\cdot x}{|y||x|}\Big)\Big|
\le c\frac{\kappa N^{|\beta|}}{(1+N\rho(y,x))^M}, \quad 0\le|\beta|\le K,
\end{equation}
\begin{equation}\label{eq:2.2.2c}
\Big|\partial^\alpha_y\partial^\beta_x G\Big(\frac{y\cdot x}{|y||x|}\Big)\Big|
\le c\frac{\kappa N^{|\beta|+1}}{(1+N\rho(y,x))^M}, \quad |\alpha|=1,~0\le|\beta|\le K-1,
\end{equation}
where $c>0$ is a constant depending only on $K$ and $d$.
\end{lem}

\begin{proof}
For $|\beta|=0$ \eqref{eq:2.2.1} with $\nu=0$ coincides with \eqref{eq:2.2.2} with $c=1$.

Let $1\le|\beta|\le K$.
From \eqref{prop-W} and \eqref{rep-par-beta2} we get $|R_{\beta,\nu}(x,y)|\le c \rho(x,y)^{(2\nu-|\beta|)_+}$, $x,y\in\SS$.
Using this estimate and \eqref{eq:2.2.1} with  $u=y\cdot x=\cos\rho(y, x)$, $x,y\in\SS$,
in \eqref{rep-par-beta} we get
\begin{align*}
\Big|\partial^\beta_x G\Big(\frac{y\cdot x}{|y||x|}\Big)\Big|
&\le c\sum_{1\le\nu\le|\beta|}|G^{(\nu)}(y\cdot x)| \rho(y, x)^{(2\nu-|\beta|)_+}\\
&\le c\sum_{1\le\nu\le|\beta|}
\kappa N^{2\nu}(1+N\rho(y, x))^{-M-\nu}\rho(y, x)^{(2\nu-|\beta|)_+}\\
&\le c\sum_{1\le\nu\le|\beta|}
\kappa N^{2\nu-(2\nu-|\beta|)_+}(1+N\rho(y, x))^{-M-\nu+(2\nu-|\beta|)_+}\\
&\le c \kappa N^{|\beta|}(1+N\rho(y, x))^{-M},
\end{align*}
which confirms \eqref{eq:2.2.2}.

For the proof of \eqref{eq:2.2.2c}
with the help of \eqref{prop-W} and \eqref{rep-par-beta2}
we estimate the quantities in \eqref{rep-par-beta_c} for $x,y\in\SS$ as follows
\begin{equation*}
|(\partial^\alpha_y W(x,y))R_{\beta,\nu}(x,y)|\le c \rho(x,y)^{1+(2\nu-|\beta|)_+},
\end{equation*}
\begin{equation*}
\big||y|^2\partial^\alpha_y R_{\beta,\nu}(x,y)-\nu y^\alpha R_{\beta,\nu}(x,y)\big|\le c \rho(x,y)^{(2\nu-|\beta|-1)_+}.
\end{equation*}
Now \eqref{rep-par-beta_c} implies for $x,y\in\SS$
\begin{align*}
\Big|\partial^\alpha_y\partial^\beta_x G\Big(\frac{y\cdot x}{|y||x|}\Big)\Big|
&\le c\sum_{0\le\nu\le|\beta|}|G^{(\nu+1)}(y\cdot x)| \rho(y, x)^{1+(2\nu-|\beta|)_+}\\
&+c\sum_{1\le\nu\le|\beta|}|G^{(\nu)}(y\cdot x)| \rho(y, x)^{(2\nu-|\beta|-1)_+}
\end{align*}
and one completes the proof of \eqref{eq:2.2.2c} along the lines of proof of \eqref{eq:2.2.2}.
\end{proof}

For $\til{\newPhi}_N$, $\til{\newPsi}_N$ and their partial derivatives we have the following estimates:

\begin{prop}
For any $N\ge 1$, $K\in\NN$, $M>0$ and $x,\eta\in\SS$ we have
\begin{equation}\label{eq:local_Phi}
\big|\partial^\beta \til{\newPhi}_N(\eta; x)\big|
\le c_4 \frac{N^{-K+|\beta|+d-1}}{(1+N\rho(\eta,x))^M},\quad~0\le |\beta|\le K+1,
\end{equation}
\begin{equation}\label{eq:local_Phi_c}
\big|\partial^\alpha_\eta\partial^\beta_x \newPhi_N\Big(\frac{\eta\cdot x}{|\eta||x|}\Big)\big|
\le c_4\frac{N^{-K+|\beta|+d}}{(1+N\rho(\eta,x))^M}, \quad |\alpha|=1,~0\le|\beta|\le K,
\end{equation}
and
\begin{equation}\label{eq:local_Psi}
\big|\partial^\beta \til{\newPsi}_N(\eta; x)\big|
\le c_5 \frac{N^{|\beta|+d-1}}{(1+N\rho(\eta,x))^M},\quad~0\le |\beta|\le K+1,
\end{equation}
\begin{equation}\label{eq:local_Psi_c}
\big|\partial^\alpha_\eta\partial^\beta_x \newPsi_N\Big(\frac{\eta\cdot x}{|\eta||x|}\Big)\big|
\le c_5\frac{N^{|\beta|+d}}{(1+N\rho(\eta,x))^M}, \quad |\alpha|=1,~0\le|\beta|\le K.
\end{equation}
where $c_4$, $c_5$ depend only on $d, K,M$, and $\varphi$.
\end{prop}

\begin{proof}
Let
$\lambda(t):=(-1)^{K/2} N^{-K}\big(t[t+(d-1)/N]\big)^{-K/2}\varphi(t)$, $t\in[0,\infty)$.
Clearly,
$$
\newPhi_N(u) = \sum_{N/2<\kk<2N} \lambda\Big(\frac{\kk}{N}\Big)\PP_\kk(u),
\quad\lambda\in C^\infty[0, \infty),\quad\hbox{and}\quad \supp \lambda \subset [1/2, 2].
$$
It is readily seen that
$\|\lambda^{(m)}\|_\infty \le cN^{-K}$ for each $m\ge 0$ with $c=c(d, m, K, \varphi)$.
Then for any $M>0$ by Theorem~\ref{thm:localization} with $M+K+1$ instead of $M$ we have
\begin{equation}\label{local-Phi}
|\newPhi_N^{(\nu)}(\cos \theta)| \le \frac{cN^{-K+ d-1}N^{2\nu}}{(1+N|\theta|)^{M+K+1}},
\quad |\theta|\le \pi, \; 0\le \nu\le K+1,
\end{equation}
where $c=c(d, M, K, \varphi)$.
Applying \lemref{lem:2.2} with $\kappa=cN^{-K+ d-1}$,
$K$ replaced by $K+1$, and \eqref{local-Phi} with $|\theta|=\rho(\eta,x)$ we get \eqref{eq:local_Phi} and \eqref{eq:local_Phi_c}.

For the localization of $\newPsi_N$ we use \eqref{eq:constr_3} and the fact that
$\|\varphi^{(m)}\|_\infty \le c$ for each $m\ge 0$ with $c=c(d, m, \varphi)$.
Thus for any $M>0$ by \thmref{thm:localization} with $M$ replaced by $M+K+1$ we obtain
\begin{equation*}
|\newPsi_N^{(\nu)}(\cos \theta)| \le \frac{cN^{d-1}N^{2\nu}}{(1+N|\theta|)^{M+K+1}},
\quad |\theta|\le \pi, \; 0\le \nu\le K+1.
\end{equation*}
This estimate along with \lemref{lem:2.2} with $\kappa=cN^{d-1}$ and
$K$ replaced by $K+1$ imply \eqref{eq:local_Psi} and \eqref{eq:local_Psi_c}.
\end{proof}

For $\xi\in\cX_j$, $j\in\NN_0$, we set $N_\xi:=2^{j-1}$.
The elements of the needlet frame $\LL=\{\psi_\xi(x): \xi\in\cX\}$, defined in \eqref{def-psi-xi},
can be represented in terms of the kernels $\newPsi_N$ as follows
\begin{equation}\label{eq:needlet_1}
	\psi_\xi(x):=C^\diamond_\xi \psi^\diamond_\xi(x),\quad
	\psi^\diamond_\xi(x):=\newPsi_{N_\xi}(\xi\cdot x)=\Psi_j(\xi\cdot x),\quad C^\diamond_\xi:=\ww_\xi^{1/2}
\end{equation}
for $x\in\SS$, $\xi\in\cX_j$, $j\in\NN$, and the coefficients $C^\diamond_\xi$ satisfy
\begin{equation}\label{eq:needlet_2}
	C^\diamond_\xi \le c_9 N_\xi^{-(d-1)/2},\quad \xi\in\cX,
\end{equation}
with $c_9=2^{(1-d)/2}c_7^{1/2}$ depending only on $d$ (cf. \eqref{cubature_w}).

In the following sections we assume that the needlet frame $\Psi$ is fixed;
the dependence of some of the constants on $\varphi$ will not be indicated explicitly.

\subsection{Highly localized kernels in terms of shifts of the Newtonian kernel}\label{s4_2}

As already explained in the introduction
our tool for approximation of harmonic functions on the ball will consist of linear combinations of shifts
of the Newtonian kernel:
$$
\frac{1}{|x|^{d-2}}\quad\hbox{in dimension}\quad d>2
\quad\hbox{or}\quad
\ln \frac{1}{|x|} \quad\hbox{if}\quad d=2,
$$
just as in \eqref{lin-comb}.
The poor localization of the Newtonian kernel, however, creates problems.
Its directional derivatives achieve much better localization
and are well approximated by finite differences.
However, as explained in \cite{IP2} they do not have either the right localization in the sense of \eqref{eq:local_1}
or $L^1(\SS)$ normalization.

We next invoke Theorem 3.1 from \cite{IP2} to show (see Corollary~\ref{cor:loc-kern} below) the existence of
highly localized summability kernels that are linear combinations of finitely many directional derivatives
of the Newtonian kernel.
Consequently, they will be arbitrarily well approximated by
linear combinations of a fixed number of shifts of the Newtonian kernel.
\begin{thm}[Theorem 3.1 in \cite{IP2}]\label{thm:2}
Let $d\ge 2$, $M>d-2$, and $0<\eps\le 1$.
Set $a:=1+\eps$, $\delta:=1-a^{-2}$ and
\begin{equation}\label{eq:potential0}
m:=\left\lceil (M-d+2)/2\right\rceil.
\end{equation}
Consider the function
\begin{equation}\label{eq:potential1}
\cF_{\eps}(u) := \eps^{2m-1} (a^2+1-2au)^{-d/2+1-m},\quad u\in[-1,1].
\end{equation}
The function $\cF_{\eps}$ has these properties:
\begin{equation}\label{eq:2.3}
\cF_{\eps}(x\cdot\eta) = \eps^{2m-1}	|x-a\eta|^{-d+2-2m},\quad \forall x, \eta\in \SS,
\end{equation}
\begin{equation}\label{main-1}
0< \cF_{\eps}(x\cdot\eta) \le \frac{c_1^\#\eps^{-d+1}}{(1+\eps^{-1}\rho(x, \eta))^{M}},
\quad\forall x, \eta\in \SS,
\end{equation}
and
\begin{equation}\label{main-2}
\int_{\SS}\cF_{\eps}(x\cdot\eta)d\sigma(x) \ge c_2^\#>0,
\quad\forall \eta\in \SS,
\end{equation}
where $c_1^\#, c_2^\#>0$ are constants depending only on $m$ and $d$.
Furthermore, there exist real numbers $b_0,b_1,\dots,b_m$ depending only on $\eps$, $m$, and $d$
such that for every $\eta\in \SS$ the function $\cF_{\eps}(x\cdot\eta)$
is the restriction on $\SS$ of the harmonic function $\sF_{\eps,m}$
defined on $\RR^d\setminus\{a\eta\}$ by
\begin{equation}\label{eq:2.1}
	\sF_{\eps,m}(a\eta,x):=\sum_{\ell=0}^m
	b_\ell (\eta\cdot\nabla)^\ell |x-a\eta|^{2-d}
\quad\mbox{if}~d\ge3,
\end{equation}
or
\begin{equation}\label{eq:2.2}
	\sF_{\eps,m}(a\eta,x):=b_0+\sum_{\ell=1}^m
	b_\ell (\eta\cdot\nabla)^\ell \ln \frac{1}{|x-a\eta|} \quad\mbox{if}~d=2.
\end{equation}
\end{thm}

We define the univariate function
\begin{equation}\label{eq:potential3}
F_\eps(u):=\kappa_{\eps,m,d}\cF_{\eps}(u),\quad u\in[-1,1],
\end{equation}
where $\cF_{\eps}$ is defined in \eqref{eq:potential1} and
\begin{equation}\label{eq:potential3b}
\kappa_{\eps,m,d}:=\Big(\int_{\SS}\cF_{\eps}(x\cdot\eta)\,d\sigma(x)\Big)^{-1},\quad \forall \eta\in\SS.
\end{equation}
Note that $\kappa_{\eps,m,d}$ is independent of $\eta$ and \eqref{main-2} implies
\begin{equation*}
	\kappa_{\eps,m,d}\le 1/c_2^\#,\quad \forall~ 0<\eps\le 1.
\end{equation*}

Given $\eta\in\SS$ we extend $F_\eps(\eta\cdot x)$ just as $\newPsi_N(\eta\cdot x)$ in \eqref{eq:constr_5} by
\begin{equation}\label{eq:potential8}
	\til{F}_\eps(\eta; x)=F_\eps\Big(\frac{\eta\cdot x}{|x|}\Big),\quad x\in\RR^d\backslash\{0\}.
\end{equation}
In this case \eqref{eq:laplace_extension} takes the form
\begin{equation}\label{eq:potential9}
	\Delta\til{F}_\eps(\eta; x)=\Delta_0 F_\eps(\eta\cdot x),\quad x\in\SS.
\end{equation}
We use $\til{F}_\eps$ to bound the derivatives of $F_\eps(\eta\cdot x)$ for $x\in\SS$.

\begin{cor}\label{cor:loc-kern}
Let $d\ge 2$, $M>d-2$, $K\in\NN$.
Let $0<\eps\le 1$
and let $F_\eps$ be defined by \eqref{eq:potential3}--\eqref{eq:potential3b}.
Then for all $x,\eta\in\SS$ we have
\begin{equation}\label{eq:potential4}
F_\eps(\eta\cdot x)=\kappa_{\eps,m,d}\sF_{\eps,m}(a\eta,x),
\end{equation}
\begin{equation}\label{eq:potential5}
\int_{\SS} F_{\eps}(\eta\cdot y)\,d\sigma(y)=1,
\end{equation}
\begin{equation}\label{eq:potential6}
\big|\partial^\beta\til{F}_\eps(\eta; x)\big|
\le c_8 \frac{(\eps^{-1})^{|\beta|+d-1}}{(1+\eps^{-1}\rho(\eta,x))^M}, \quad 0\le|\beta|\le 2K+1,
\end{equation}
\begin{equation}\label{eq:potential6_c}
\big|\partial^\alpha_\eta\partial^\beta_x F_\eps\Big(\frac{\eta\cdot x}{|\eta||x|}\Big)\big|
\le c_8 \frac{(\eps^{-1})^{|\beta|+d}}{(1+\eps^{-1}\rho(\eta,x))^M}, \quad |\alpha|=1,~0\le|\beta|\le 2K,
\end{equation}
where $\til{F}_\eps$ is defined by \eqref{eq:potential8} and $c_8$ depends only on $d,K,M$.
\end{cor}

\begin{proof}
Identity \eqref{eq:potential4} follows from \eqref{eq:potential3} and the second part of \thmref{thm:2},
while \eqref{eq:potential5} follows from \eqref{eq:potential3} and \eqref{eq:potential3b}.
From \eqref{eq:potential3} and \eqref{eq:potential1} it follows that for any $\nu\in\NN_0$
\begin{equation*}
F_\eps^{(\nu)}(u)
=\kappa_{\eps,m,d}\,a^{\nu}\eps^{2m-1} (a^2+1-2au)^{-d/2+1-m-\nu}
\prod_{k=0}^{\nu-1} (2m+d-2+2k).
\end{equation*}
On the other hand, using that $a=1+\eps$ it is easy to show that
\begin{equation*}
\frac{1}{5}\le\frac{(a^2+1-2au)^{1/2}}{\eps(1+\eps^{-1}\arccos u)}\le 2, \quad u\in[-1,1],
\end{equation*}
see the proof of inequalities (3.7) in \cite{IP2}.
The above, \eqref{eq:potential0}, and \eqref{main-2} yield
\begin{equation*}
|F_\eps^{(\nu)}(u)|\le \frac{c \eps^{-(2\nu+d-1)}}{(1+\eps^{-1}\arccos u)^{d-2+2m+2\nu}}
\le \frac{c \eps^{-(2\nu+d-1)}}{(1+\eps^{-1}\arccos u)^{M+\nu}}
\end{equation*}
with $c$ depending on $d$, $m$, and $\nu$.
In turn, this estimate and \lemref{lem:2.2} with $G=F_\eps$, $N=\eps^{-1}$, $\kappa=cN^{d-1}$, and
$K$ replaced by $2K+1$ imply \eqref{eq:potential6} and \eqref{eq:potential6_c}.
\end{proof}

Note that the extension $\til{F}_\eps(\eta; x)$ of $F_\eps(\eta\cdot x)$ from \eqref{eq:potential8}
is \emph{different} from its harmonic extension $\kappa_{\eps,m,d}\sF_{\eps,m}(a\eta,x)$
given in \eqref{eq:potential4} of the form  \eqref{eq:2.1}--\eqref{eq:2.2}.

\section{Frames in terms of shifts of the Newtonian kernel}\label{s6}

We now come to the most technical part of our development -- the construction of a frame
whose elements are finite lineal combinations of shifts of the Newtonian kernel.
We shall carry our this construction in several steps.

\subsection{The main technical step in the construction of the new frame on $\SS$}\label{subsec:scheme}

We now focus on the construction of highly localized frame elements
$\{\theta_\xi: \xi\in \cX\}$
of the form
\begin{equation}\label{form-frame-2}
\theta_\xi(x) =\sum_{\nu=1}^{\tilde{n}} \frac{a_\nu}{|x-y_\nu|^{d-2}},\quad\hbox{if}\quad d>2,
\end{equation}
or
\begin{equation}\label{form-frame-3}
\theta_\xi(x) =\sum_{\nu=1}^{\tilde{n}} a_\nu \ln \frac{1}{|x-y_\nu|},\quad\hbox{if}\quad d=2.
\end{equation}
Here $y_\nu\in\RR^d$ with $|y_\nu|>1$, $a_\nu\in\RR$,
and $\{y_\nu\}_{\nu=1}^{\tilde{n}}$ and $\{a_\nu\}_{\nu=1}^{\tilde{n}}$ may vary with $\xi\in\cX$,
but $\tilde{n}$ is fixed.

Assume that $\LL=\{\psi_\xi: \xi\in \cX\}$, $\cX=\cup_{j\ge 0} \cX_j$, is the existing frame,
described in \S\ref{subsec:frame-SS}.
For the construction of the new frame elements $\{\theta_\xi\}$
we utilize the small perturbation method, described in \S\ref{s3}.
In applying this scheme the main step is to construct frame elements $\theta_\xi$, $\xi\in\cX$,
of the form (\ref{form-frame-2})--(\ref{form-frame-3}) so that
\begin{equation}\label{inner-prod-1}
|\langle \psi_\eta-\theta_\eta, \psi_\xi\rangle| \le \gamma_0\omega_{\xi, \eta},
\quad\xi,\eta\in\cX,
\end{equation}
and
\begin{equation}\label{psi-theta-est}
|\psi_\xi(x)-\theta_\xi(x)| \le \frac{\gamma_0 N_\xi^{(d-1)/2}}{(1+N_\xi\rho(x, \xi))^M},
\quad x\in\SS,\;\xi\in\cX.
\end{equation}
Here $N_\xi=2^{j-1}$ for $\xi\in\cX_j$, $\gamma_0>0$ is a small parameter,
$\omega_{\xi, \eta}$ are the entries of an~almost diagonal matrix
like $\omega_{\xi, \eta}^{(K, M)}$ from (\ref{eq:omega_xi_eta}),
and $M>0$ is sufficiently large.
The result of this construction will be a frame $\{\theta_\xi\}_{\xi\in\cX}$
for the Besov and Triebel-Lizorkin spaces of interest.

It will be convenient to us to approximate the essentially $L^1$-normalized frame elements
$\psi_\xi^\diamond(x):=\newPsi_{N_\xi}(\xi\cdot x)$
defined in \eqref{eq:needlet_1}
by essentially $L^1$-normalized new frame elements $\theta_\xi^\diamond$.
Then multiplication by constants $C_\xi^\diamond$ (see \eqref{eq:needlet_1})
will complete the construction of $L^2$-normalized frame elements.

The construction of the new frame elements $\theta_\xi^\diamond$ will be carried out in four steps:

(a) Approximation of $\newPsi_{N_\xi}(\xi\cdot x)$, $\xi\in\cX$, by convolving $\newPsi_{N_\xi}$
with the kernel $F_{\eps}(y\cdot x)$ from \eqref{eq:potential4}.

(b) Discretization of the convolutions by using the cubature formula from \eqref{simpl-cubature}.

(c) Truncation of the resulting sums.

(d) Approximation of the truncated sums by discrete versions of the operators involved.

These approximation steps will be governed by four small parameters (constants):
$\gamma_1$, $\gamma_2$, $\gamma_3$, $\gamma_4 >0$.
The relations between these parameters and all involved constants will be carefully traced.

We next introduce some convenient notation and set up the approximation steps described above.
For the only index $\xi\in\cX_0$ we set $\theta_\xi(x):=\psi_\xi(x)\equiv 1$.
In the remaining part of this subsection we consider
$\xi\in\cX\setminus\cX_0=\cup_{j=1}^\infty \cX_j$.

Given $0<\gamma_1\le 1$ (to be selected), we set
\begin{equation}\label{eq:eps}
\eps:=\gamma_1/N_\xi
\end{equation}
and define
\begin{equation}\label{eq:g1}
g_1(\xi;x):=\int_{\SS}\newPhi_{N_\xi}(\xi\cdot y)F_{\eps}(y\cdot x)\, d\sigma(y),\quad x\in\SS,
\end{equation}
where $\newPhi_{N_\xi}$ is defined in \eqref{eq:constr_4}, $F_{\eps}(y\cdot x)$ is the kernel from \eqref{eq:potential4}
with $\varepsilon$ from \eqref{eq:eps}, and $m$ from \eqref{eq:potential0}.

Given $0<\gamma_2\le \gamma_1$ (to be selected), we let $\cZ_j\subset\SS$ be a fixed maximal $\delta$-net
with $\delta:=\gamma_2 2^{-j+1}$
and let $\{\nA_\zeta\}_{\zeta\in\cZ_j}$ be the associated partition of $\SS$ (see Subsection~\ref{s5_3}).
Applying cubature formula \eqref{simpl-cubature} with nodes $\zeta\in\cZ_j$ and weights $w_\zeta=|\nA_\zeta|$ to \eqref{eq:g1}
we arrive at
\begin{equation}\label{eq:g2}
g_2(\xi;x):=\sum_{\zeta\in\cZ_j} w_\zeta \newPhi_{N_\xi}(\xi\cdot \zeta)F_{\eps}(\zeta\cdot x),\quad x\in\SS.
\end{equation}
Observe that there is no connection between the nodal sets $\cX_j$ and $\cZ_j$ ($j\in\NN$).
In particular, the cubature $\sum_{\zeta\in\cZ_j}w_\zeta f(\zeta)$ from \eqref{simpl-cubature} has to be exact only for constants,
while the cubature $\sum_{\xi\in\cX_j}\ww_\xi f(\xi)$ from \eqref{cubature} is required to be exact
for all spherical harmonics of degree $\le 2^{j+1}$.

Given $0<\gamma_3\le 1$ (to be determined), we truncate the sum in \eqref{eq:g2} by including only the nodes
within distance $\rrr_\xi:=(\gamma_3 N_\xi)^{-1}$ from $\xi$ to obtain
\begin{equation}\label{eq:g3}
g_3(\xi;x):=\sum_{\substack{\zeta\in\cZ_j\\ \rho(\zeta,\xi)\le \rrr_\xi}} w_\zeta \newPhi_{N_\xi}(\xi\cdot \zeta)F_{\eps}(\zeta\cdot x),
\quad x\in\SS.
\end{equation}
The functions $g_1(\xi;x)$, $g_2(\xi;x)$, and $g_3(\xi;x)$ should be viewed as
consecutive approximations of $\newPhi_{N_\xi}(\xi\cdot x)$.

We obtain consecutive approximations to $\newPsi_{N_\xi}(\xi\cdot x)$ by
applying $\Delta_0^{K/2}$ to each of the functions $g_1$, $g_2$, $g_3$ in \eqref{eq:g1}, \eqref{eq:g2}, and \eqref{eq:g3}.
We set
\begin{multline}\label{eq:h1}
h_1(\xi;x):= \Delta_0^{K/2}g_1(\xi;x)
=\Delta_0^{K/2}\int_{\SS}\newPhi_{N_\xi}(x\cdot y)F_{\eps}(y\cdot \xi)\, d\sigma(y) \\
= \int_{\SS}\Delta_0^{K/2}\newPhi_{N_\xi}(x\cdot y)F_{\eps}(y\cdot \xi)\, d\sigma(y)
= \int_{\SS}\newPsi_{N_\xi}(x\cdot y)F_{\eps}(y\cdot \xi)\, d\sigma(y),
\end{multline}
\begin{equation}\label{eq:h2}
h_2(\xi;x)
:=\Delta_0^{K/2}g_2(\xi;x)=\sum_{\zeta\in\cZ_j}
w_\zeta \newPhi_{N_\xi}(\xi\cdot \zeta)\Delta_0^{K/2}F_{\eps}(\zeta\cdot x),
\end{equation}
\begin{equation}\label{eq:h3}
h_3(\xi;x)
:=\Delta_0^{K/2}g_3(\xi;x)=\sum_{\substack{\zeta\in\cZ_j\\ \rho(\zeta,\xi)\le\rrr_\xi}}
w_\zeta \newPhi_{N_\xi}(\xi\cdot \zeta)\Delta_0^{K/2}F_{\eps}(\zeta\cdot x).
\end{equation}
Above in \eqref{eq:h1} we first used the commutativity of the inner product of zonal functions \eqref{eq:comut}
in the definition of $g_1$ followed by \eqref{eq:LB2} in the last equality.

Observe that $h_3(\xi;x)$ is a linear combination of finitely many (independent of $\xi$) terms of the form
\begin{equation*}
\Delta_0^{K/2}F_{\eps}(\zeta\cdot x)
=\kappa_{\eps,m,d} \sum_{\ell=0}^m b_{\ell} \Delta_0^{K/2}(\zeta\cdot\nabla)^\ell |x-a\zeta|^{2-d},
\quad\hbox{if}\; d\ge 3,
\end{equation*}
see \eqref{eq:2.1}, \eqref{eq:potential3b}, and \eqref{eq:potential4}.
We have a similar representation of $h_3(\xi;x)$ in dimension $d=2$.
Replacing the differential operator
$\Delta_0^{K/2}\left(\zeta\cdot\nabla\right)^\ell$ in \eqref{eq:h3} by its discrete counterpart $\fL_t^{K/2}\fD^\ell_t(\zeta)$
with an appropriate small $t=t_j>0$ (to be specified)
we arrive at the following definition of $\theta^\diamond_\xi$, $\xi\in\cX_j$, $j\in\NN$, in dimension $d\ge 3$
\begin{equation}\label{eq:theta_diamond}
\theta^\diamond_\xi(x):=\kappa_{\eps,m,d}\sum_{\substack{\zeta\in\cZ_j\\ \rho(\zeta,\xi)\le\rrr_\xi}}
w_\zeta \newPhi_{N_\xi}(\xi\cdot \zeta)
\sum_{\ell=0}^m b_{\ell}\fL_{t_j}^{K/2}\fD^\ell_{t_j}(\zeta)  |x-a\zeta|^{2-d}.
\end{equation}
If $d=2$ we set
\begin{equation}\label{eq:theta_diamond2}
\theta^\diamond_\xi(x):=\kappa_{\eps,m,2}\sum_{\substack{\zeta\in\cZ_j\\ \rho(\zeta,\xi)\le\rrr_\xi}}
w_\zeta \newPhi_{N_\xi}(\xi\cdot \zeta)
\sum_{\ell=1}^m	b_{\ell} \fL_{t_j}^{K/2}\fD^\ell_{t_j}(\zeta)  \ln \frac{1}{|x-a\zeta|}.
\end{equation}
Several remarks are in order:
\begin{enumerate}
\item
The finite difference operator $\fD^\ell_t(\varsigma):=t^{-\ell}\sum_{k=0}^\ell (-1)^{\ell-k} \binom{\ell}{k} T(\varsigma,kt)$
is defined by the translation operator (in $\RR^d$) in direction $\varsigma\in\SS$ with step $t$ given by
$T(\varsigma,t)f(x)=f(x+t\varsigma)$ for $x\in\RR^d$.

\item Following \cite[p. 23 or (4.2.1) on p. 81]{DX} the rotation $Q_{1,2,t}\in SO(d)$ is given by
\begin{multline*}
Q_{1,2,t}\varsigma = Q_{1,2,t}(\varsigma_1,\varsigma_2,\dots,\varsigma_d) \\
:= (\varsigma_1 \cos t +\varsigma_2 \sin t, -\varsigma_1 \sin t +\varsigma_2 \cos t,\varsigma_3,\dots,\varsigma_d),
\quad \varsigma\in\SS,
\end{multline*}
and $Q_{i,\ell,t}\varsigma$ is defined similarly for any $1\le i<\ell\le d$.
The translation operator corresponding to the rotation $Q_{i,\ell,t}$, $1\le i<\ell\le d$, is given by
\begin{equation*}
T(Q_{i,\ell,t})f(\varsigma) := f(Q_{i,\ell,t}^{-1}\varsigma)=f(Q_{i,\ell,-t}\varsigma).
\end{equation*}
The operator
\begin{equation*}
\fL_tf(\varsigma):=t^{-2}\sum_{1\le i<\ell\le d}(T(Q_{i,\ell,t}) +T(Q_{i,\ell,-t})-2\II)f(\varsigma),
\end{equation*}
where $\II$ stands for the identity, approximates well $\Delta_0 f(\varsigma)$ for small $t$;
the powers of $\fL_t$ are defined as usual by $\fL_t^k:=\fL_t(\fL_t^{k-1})$.

\item The numbers $a$, $\delta$, $m$, and $b_\ell$, $\ell=0,1,\dots,m$, are determined in \thmref{thm:2}
as functions of $\eps$ from \eqref{eq:eps} and $M$.
\end{enumerate}

We now come to the first main assertion in this section.

\begin{thm}\label{thm:new_frame_diamond}
Let $d\ge 2$, $K\in 2\NN$, $M>K+d-1$, and $\gamma_0>0$.
Then there exist constants
$\gamma_1$, $\gamma_2$, $\gamma_3$, $\gamma_4>0$ depending only on $d, K, M, \gamma_0$,
and
for every $j\in\NN$
there exists $t_j>0$ depending only on $d, K, M, \gamma_0$, and $j$ such that for every $\xi\in\cX$
the element $\theta^\diamond_\xi$ from \eqref{eq:theta_diamond} or \eqref{eq:theta_diamond2} obeys
\begin{equation}\label{eq:small-0}
\big|\partial^\beta \big[\til{\psi}^\diamond_\xi(x)-\til{\theta}^\diamond_\xi(x)\big]\big|
\le \frac{\gamma_0 N_\xi^{|\beta|+d-1}}{(1+N_\xi\rho(\xi,x))^M},\quad \forall x\in\SS,~ 0\le|\beta|\le K,
\end{equation}
\begin{equation}\label{eq:small-1}
\Big|\int_{\SS} y^\beta \left[\psi^\diamond_\xi(y)-\theta^\diamond_\xi(y)\right]d\sigma(y)\Big|
\le \gamma_0 N_\xi^{-K},~~0\le|\beta|\le K-1.
\end{equation}
Here
$\til{\theta}^\diamond_\xi(x):=\theta^\diamond_\xi(x/|x|)$, $x\in\RR^d\setminus\{0\}$, and
$\til{\psi}^\diamond_\xi(x):=\til{\newPsi}_{N_\xi}(\xi; x)$, see \eqref{eq:constr_5}, \eqref{eq:needlet_1}.
\end{thm}

The proof of  \thmref{thm:new_frame_diamond} relies on four lemmas,
which establish estimates similar to estimates \eqref{eq:small-0} and \eqref{eq:small-1} for the differences
$\psi^\diamond_\xi-h_1(\xi;\cdot)$,
$h_1(\xi;\cdot)-h_2(\xi;\cdot)$,
$h_2(\xi;\cdot)-h_3(\xi;\cdot)$, and
$h_3(\xi;\cdot)-\theta^\diamond_\xi$.
The values of the parameters $\gamma_1$, $\gamma_2$, $\gamma_3$, $\gamma_4$, used in these four lemmas
will be selected in the proof of \thmref{thm:new_frame_diamond} (see \eqref{param0}).

In light of the last integral in \eqref{eq:h1} we define
\begin{equation*}
	\til{h}_1(\xi;x):=\int_{\SS}\til{\newPsi}_{N_\xi}(y;x)F_{\eps}(y\cdot \xi)\, d\sigma(y),\quad x\in\RR^d\setminus\{0\},
\end{equation*}
where $\til{\newPsi}_{N_\xi}(y; x)$ is defined in \eqref{eq:constr_5} and $\eps$ is from \eqref{eq:eps}.

\begin{lem}\label{lem:1_2}
Let $\xi\in\cX_j$, $j\in\NN$, $K\in 2\NN$, and $M>K+d-1$.
If $0<\gamma_1\le 1$, $\eps$ is from \eqref{eq:eps} and $F_{\eps}$ is from \eqref{eq:potential3},
then:

$(a)$
For any $\beta$, $0\le|\beta|\le K$,
\begin{equation}\label{eq:approx1_1}
\big|\partial^\beta\big[\til{\newPsi}_{N_\xi}(\xi; x)-\til{h}_1(\xi;x)\big]\big|
\le c_{10}\frac{\gamma_1 N_\xi^{|\beta|+d-1}}{(1+N_\xi\rho(\xi,x))^M},\quad \forall x\in\SS;
\end{equation}

$(b)$
For any $\beta$, $0\le|\beta|\le K-1$,
\begin{equation}\label{eq:approx1_3}
\Big|\int_{\SS} y^\beta \left[\newPsi_{N_\xi}(\xi\cdot y)-h_1(\xi;y)\right]d\sigma(y)\Big|
\le c_{10}\gamma_1 N_\xi^{-K},
\end{equation}
where $c_{10}$ depends only on $d, K, M$.
\end{lem}

\begin{proof}
\begin{sloppypar}
Let $0\le |\beta|\le K$.
We apply \propref{prop:5_2} with
$g(y)=\partial^\beta_x \til{\newPsi}_{N_\xi}(y;x)$, $x_1\!=\!x$, $N_1=N_\xi$,
$\kappa_1=c_5 N_\xi^{|\beta|}$ on account of \eqref{eq:local_Psi_c},
and with $f(y)=F_{\eps}(\xi\cdot y)$, $x_2=\xi$, $N_2=1/\eps$, $\eps=\gamma_1/N_\xi$,
$\kappa_2=c_8$ on account of \eqref{eq:potential6} with $\beta=0$.
Observe that $N_2=N_\xi/\gamma_1\ge N_1$.
Hence, because of \eqref{eq:potential5}, inequality \eqref{eq:4b} implies
\end{sloppypar}
\begin{equation}\label{eq:w1}
\big|\partial^\beta_x\til{h}_1(\xi;x)-\partial^\beta_x\til{\newPsi}_{N_\xi}(\xi; x)\big|
\le c_2c_5c_8 \frac{\gamma_1 N_\xi^{|\beta|+d-1}  }{(1+N_\xi\rho(\xi,x))^M},\quad x\in\SS.
\end{equation}

For the proof of \eqref{eq:approx1_3} we apply \propref{prop:5_2} with $x_1=\xi$, $N_1=N_\xi$,
$g(y)=\newPhi_{N_\xi}(\xi; y)$,
$\kappa_1=c_4 N_\xi^{-K}$ in view of \eqref{eq:local_Phi} with $|\beta|=1$,
$f(y)=F_{\eps}(x\cdot y)$ with any fixed $x\in\SS$, $N_2=1/\eps$, $\eps=\gamma_1/N_\xi$,
$\kappa_2=c_8$ in view of \eqref{eq:potential6} with $\beta=0$.
Consequently, because of \eqref{eq:potential5}, inequality \eqref{eq:4b} implies
\begin{equation}\label{eq:approx1_4}
\big|\newPhi_{N_\xi}(\xi\cdot x)-g_1(\xi;x)\big|
\le c_{11} \frac{\gamma_1 N_\xi^{-K+d-1}  }{(1+N_\xi\rho(\xi,x))^M},\quad x\in\SS,
\end{equation}
with $c_{11}:= c_2 c_4 c_8$.
Now, for $0\le|\beta|\le K-1$ we apply consecutively \eqref{eq:LB2}, \eqref{eq:h1},
the fact that the operator $\Delta_0$ is symmetric,
\eqref{eq:approx1_4} with $y$ in place of $x$, \eqref{eq:conv_1}, and \eqref{eq:small-1_powers}
to obtain
\begin{multline}\label{eq:w2}
\Big|\int_{\SS} y^\beta \left[\newPsi_{N_\xi}(\xi\cdot y)-h_1(\xi;y)\right]d\sigma(y)\Big|\\
=\Big|\int_{\SS} y^\beta \Delta_0^{K/2}\left[\newPhi_{N_\xi}(\xi\cdot y)-g_1(\xi;y)\right]d\sigma(y)\Big|\\
=\Big|\int_{\SS} \left[\newPhi_{N_\xi}(\xi\cdot y)-g_1(\xi;y)\right]\Delta_0^{K/2}y^\beta d\sigma(y)\Big|
\le c_{11}c_0c_6\gamma_1 N_\xi^{-K}.
\end{multline}
Finally, \eqref{eq:w1} and \eqref{eq:w2} imply \eqref{eq:approx1_1} and \eqref{eq:approx1_3}
with  $c_{10}:=\max\{c_2c_5c_8,c_{11}c_0c_6\}$.
The proof is complete.
\end{proof}

The first integral in \eqref{eq:h1}, \eqref{eq:comut}, and identity \eqref{eq:potential9}
give another representation of $\til{h}_1(\xi;x)$, namely,
\begin{equation}\label{eq:h1_b}
\til{h}_1(\xi;x)= \int_{\SS}\newPhi_{N_\xi}(\xi\cdot y)\Delta^{K/2}\til{F}_{\eps}(y; x)\, d\sigma(y),\quad x\in\RR^d\setminus\{0\},
\end{equation}
with $\til{F}_{\eps}(y; x)$ defined in \eqref{eq:potential8}.
Using \eqref{eq:h2} and \eqref{eq:potential9} we also set for $\xi\in\cX_j$
\begin{equation}\label{eq:h2_b}
\til{h}_2(\xi;x):=\sum_{\zeta\in\cZ_j}
w_\zeta \newPhi_{N_\xi}(\xi\cdot \zeta)\Delta^{K/2}\til{F}_{\eps}(\zeta; x),\quad x\in\RR^d\setminus\{0\}.
\end{equation}

\begin{lem}\label{lem:1_3}
Assume $\xi\in\cX_j$, $j\in\NN$, $K\in 2\NN$, $M>K+d-1$, and let $\gamma_1$, $\eps$, $F_{\eps}$ be as in \lemref{lem:1_2}.
If $0<\gamma_2\le \gamma_1$ and $\til{F}_{\eps}$ is from \eqref{eq:potential8},
then:

$(a)$
For any $\beta$, $0\le|\beta|\le K$,
\begin{equation}\label{eq:approx2_2}
\big|\partial^\beta\big[\til{h}_1(\xi;x)-\til{h}_2(\xi;x)\big]\big|
\le c_{20}\gamma_2 \gamma_1^{-2K-1} \frac{ N_\xi^{|\beta|+d-1}}{(1+N_\xi\rho(\xi,x))^M},\quad \forall x\in\SS;
\end{equation}

$(b)$
For any $\beta$, $0\le|\beta|\le K-1$,
\begin{equation}\label{eq:approx2_5}
\Big|\int_{\SS} y^\beta \left[h_1(\xi;y)-h_2(\xi;y)\right]d\sigma(y)\Big|
\le c_{20}\gamma_2 \gamma_1^{-2K-1} N_\xi^{-K},
\end{equation}
where $c_{20}$ depends only on $d, K, M$.
\end{lem}

\begin{proof}
Let $0\le |\beta|\le K$.
From \eqref{eq:h1_b} and \eqref{eq:h2_b} we get
\begin{multline}\label{eq:approx2_1}
\big|\partial^\beta\big[\til{h}_1(\xi;x)-\til{h}_2(\xi;x)\big]\big|\\
=\Big|\int_{\SS}\!\!\newPhi_{N_\xi}(\xi\cdot y)\partial^\beta\Delta^{K/2}\til{F}_{\eps}(y; x) d\sigma(y)
 -\!\sum_{\zeta\in\cZ_j} w_\zeta \newPhi_{N_\xi}(\xi\cdot \zeta)\partial^\beta\Delta^{K/2}\til{F}_{\eps}(\zeta; x)\Big|\\
=\Big|\int_{\SS}\!\!\big[\newPhi_{N_\xi}(\xi\cdot y)\partial^\beta\Delta^{K/2}\til{F}_{\eps}(y; x)
 -\newPhi_{N_\xi}(\xi\cdot \zeta(y))\partial^\beta\Delta^{K/2}\til{F}_{\eps}(\zeta(y); x)\big] d\sigma(y)\Big|\\
\le\Big|\int_{\SS}\newPhi_{N_\xi}(\xi\cdot y)\big[\partial^\beta\Delta^{K/2}\til{F}_{\eps}(y; x)
 -\partial^\beta\Delta^{K/2}\til{F}_{\eps}(\zeta(y); x)\big]d\sigma(y)\Big|~~~~~~~~~~~~~~~~\\
+\Big|\int_{\SS}\big[\newPhi_{N_\xi}(\xi\cdot y)
 -\newPhi_{N_\xi}(\xi\cdot \zeta(y))\big]\partial^\beta\Delta^{K/2}\til{F}_{\eps}(\zeta(y); x) d\sigma(y)\Big|,~~~~~~~~~~~~~~~~~~~~
\end{multline}
where $\zeta(y)$ is defined in \S\ref{s5_3}.
Let
$\eta=\eta(s)$, $s\in[0,\rho]$, $\rho=\rho(y,\zeta(y))$, be the geodesic line on $\SS$ such that $\eta(0)=y$ and $\eta(\rho)=\zeta(y)$.
Then
\begin{equation*}
\partial^\beta\Delta^{K/2}\til{F}_{\eps}(y; x)-\partial^\beta\Delta^{K/2}\til{F}_{\eps}(\zeta(y); x)
=\int_0^\rho \left.\nabla_\eta \partial_x^\beta\Delta_x^{K/2}F_{\eps}\Big(\frac{\eta\cdot x}{|\eta||x|}\Big)\right|_{\eta=\eta(s)}
\cdot \eta'(s)\,ds.
\end{equation*}
Using in the above representation \eqref{eq:potential6_c}, \eqref{comp_local},
$\rho(y, \zeta(y))\le\gamma_2N_\xi^{-1}$, $\gamma_2\le\gamma_1\le 1$ we get
\begin{multline*}
\big|\partial^\beta\Delta^{K/2}\til{F}_{\eps}(y; x)-\partial^\beta\Delta^{K/2}\til{F}_{\eps}(\zeta(y); x)\big|
\le c_{21} \frac{\gamma_2N_\xi^{-1}\eps^{-(K+|\beta|+d)}}{(1+\eps^{-1}\rho(y, x))^M}\\
= c_{21} \frac{\gamma_2\gamma_1^{-(K+|\beta|+1)} N_\xi^{K+|\beta|}\eps^{-d+1}}{(1+\eps^{-1}\rho(y, x))^M},\quad x,y\in\SS.
\end{multline*}
Applying \propref{prop:5_3} with $g(y)=\newPhi_{N_\xi}(\xi\cdot x)$, $N_1=N_\xi$,
$\kappa_1=c_4 N_\xi^{-K}$ (because of \eqref{eq:local_Phi} with $\beta=0$),
$f(y)=\partial^\beta\Delta^{K/2}\til{F}_{\eps}(y; x)-\partial^\beta\Delta^{K/2}\til{F}_{\eps}(\zeta(y); x)$, $N_2=1/\eps\ge N_1$,
and $\kappa_2=c_{21}\gamma_2\gamma_1^{-2K-1}N_\xi^{K+|\beta|}$ (using the above estimate)
we get
\begin{multline*}
\Big|\int_{\SS}\newPhi_{N_\xi}(\xi\cdot y)\big[\partial^\beta\Delta^{K/2}\til{F}_{\eps}(y; x)
-\partial^\beta\Delta^{K/2}\til{F}_{\eps}(\zeta(y); x)\big] d\sigma(y)\Big|\\
\le c_3 c_4 c_{21} \frac{\gamma_2 \gamma_1^{-2K-1}N_\xi^{|\beta|+d-1}}{(1+N_\xi\rho(\xi, x))^M},\quad x\in\SS.
\end{multline*}
Similarly, using \eqref{eq:local_Phi_c} and \eqref{eq:potential6} we obtain
\begin{multline*}
\Big|\int_{\SS}\big[\newPhi_{N_\xi}(\xi\cdot y)-\newPhi_{N_\xi}(\xi\cdot \zeta(y))\big]
\partial^\beta\Delta^{K/2}\til{F}_{\eps}(\zeta(y); x) d\sigma(y)\Big|\\
\le c_{22} \frac{\gamma_2 \gamma_1^{-2K}N_\xi^{|\beta|+d-1}}{(1+N_\xi\rho(\xi, x))^M},\quad x\in\SS.
\end{multline*}
Substituting the above two estimates in \eqref{eq:approx2_1} we get \eqref{eq:approx2_2}
with $c_{20}\ge c_{23}:= c_{22}+c_3 c_4 c_{21}$.

For the proof of \eqref{eq:approx2_5}
we repeat the arguments applied for the proof of \eqref{eq:approx2_2} and get
\begin{equation}\label{eq:approx2_4}
|g_1(\xi;x)-g_2(\xi;x)|
\le c_{24} \frac{\gamma_2 \gamma_1^{-1}N_\xi^{-K+d-1}}{(1+N_\xi\rho(\xi, x))^M}.
\end{equation}
Now, for $0\le|\beta|\le K-1$ we apply consecutively \eqref{eq:h1}, \eqref{eq:h2},
the fact that the operator $\Delta_0$ is symmetric,
\eqref{eq:approx2_4}, \eqref{eq:conv_1}, and \eqref{eq:small-1_powers} to obtain
\begin{align*}
&\Big|\int_{\SS} y^\beta \left[h_1(\xi;y)-h_2(\xi;y)\right]d\sigma(y)\Big|\\
&\qquad\qquad=\Big|\int_{\SS} y^\beta \Delta_0^{K/2}\left[g_1(\xi;y)-g_2(\xi;y)\right]d\sigma(y)\Big|\\
&\qquad\qquad=\Big|\int_{\SS} \left[g_1(\xi;y)-g_2(\xi;y)\right]\Delta_0^{K/2}y^\beta d\sigma(y)\Big|
\le c_0c_6c_{24} \gamma_2 \gamma_1^{-1} N_\xi^{-K},
\end{align*}
which yields \eqref{eq:approx2_5} with $c_{20}=\max\{c_{23}, c_0c_6c_{24}\}$ in view of $\gamma_1\le 1$.
\end{proof}

The estimates on $h_2(\xi;\cdot)-h_3(\xi;\cdot)$
are given in the following lemma,
where $\til{h}_2(\xi;x)$ is as in \lemref{lem:1_3} and for $\xi\in\cX_j$ we set
\begin{equation*}
\til{h}_3(\xi;x)=\sum_{\zeta\in\cZ_j:\rho(\zeta,\xi)\le\rrr_\xi}
w_\zeta \newPhi_{N_\xi}(\xi\cdot \zeta)\Delta^{K/2}\til{F}_{\eps}(\zeta; x).
\end{equation*}

\begin{lem}\label{lem:1_4}
Let $\xi\in\cX_j$, $j\in\NN$, $K\in 2\NN$, $M>K+d-1$, and
let $\gamma_1$, $\gamma_2$, $\eps$, $\til{F}_{\eps}$ be as in \lemref{lem:1_3}.
If $0<\gamma_3\le 1$ and $\rrr_\xi=(\gamma_3 N_\xi)^{-1}$, then:

$(a)$
For any $\beta$, $0\le|\beta|\le K$, we have
\begin{equation}\label{eq:approx3_2}
\big|\partial^\beta\big[\til{h}_2(\xi;x)-\til{h}_3(\xi;x)\big]\big|
\le c_{30}\gamma_3\gamma_1^{-2K} \frac{N_\xi^{|\beta|+d-1}}{(1+N_\xi\rho(\xi,x))^M},\quad \forall x\in\SS,
\end{equation}

$(b)$
For any $\beta$, $0\le|\beta|\le K-1$,  we have
\begin{equation}\label{eq:approx3_5}
\Big|\int_{\SS} y^\beta \left[h_2(\xi;y)-h_3(\xi;y)\right]d\sigma(y)\Big|
\le c_{30}\gamma_3\gamma_1^{-2K} N_\xi^{-K},
\end{equation}
where $c_{30}$ depends only on $d, K, M$.
\end{lem}

\begin{proof}
Let $\xi\in\cX_j$. Set $\newPhi_{N_\xi}^*(u):=\newPhi_{N_\xi}(u)\ONE_{(r_\xi,\pi]}(u)$.
Then for $0\le|\beta|\le K$ we get from \eqref{eq:h2} and \eqref{eq:h3}
\begin{align}\label{eq:approx3_1}
\partial^\beta\big[\til{h}_2(\xi;x)-\til{h}_3(\xi;x)\big]
&=\sum_{\zeta\in\cZ_j}
w_\zeta \newPhi_{N_\xi}^*(\xi\cdot \zeta)\partial^\beta\Delta^{K/2}\til{F}_{\eps}(\zeta;x)\notag\\
&= \int_{\SS}\newPhi_{N_\xi}^*(\xi\cdot \zeta(y))\partial^\beta\Delta^{K/2}\til{F}_{\eps}(\zeta(y);x)d\sigma(y),
\end{align}
where $\zeta(y)$ is defined in \S\ref{s5_3}.
For $\newPhi_{N_\xi}(\xi\cdot x)$ estimate \eqref{eq:local_Phi} with $\beta=0$
and $M$ replaced by $M+1$ yields
\begin{equation*}
\big|\newPhi_{N_\xi}(\xi\cdot x)\big|
\le c_4^* \frac{N_\xi^{-K+d-1}}{ (1+N_\xi\rho(\xi,x))^{M+1}},\quad x\in\SS.
\end{equation*}
This estimate and the inequality $(1+N_\xi\rho(\xi,x))^{-M-1}\le \gamma_3(1+N_\xi\rho(\xi,x))^{-M}$,
if $\rho(\xi,x)>\rrr_\xi$, yield
\begin{equation*}
\big|\newPhi_{N_\xi}^*(\xi\cdot x)\big|
\le c_4^* \frac{\gamma_3 N_\xi^{-K+d-1}}{ (1+N_\xi\rho(\xi,x))^M},\quad x\in\SS,
\end{equation*}
and hence
\begin{equation}\label{eq:local_Phi3}
\big|\newPhi_{N_\xi}^*(\xi\cdot \zeta(y))\big|
\le c_{31} \frac{\gamma_3 N_\xi^{-K+d-1}}{ (1+N_\xi\rho(\xi,y))^M},\quad y\in\SS,
\end{equation}
with $c_{31}=2^M c_4^*$ on account of $\rho(y,\zeta(y))\le \gamma_2 N_\xi^{-1}\le N_\xi^{-1}$ and \eqref{comp_local}.
Using \eqref{comp_local} again we get from \eqref{eq:potential6}
\begin{equation}\label{eq:potential7}
\big|\partial^\beta\Delta^{K/2}\til{F}_{\eps}(\zeta(y);x)\big|
\le c_{32} \frac{\eps^{-(K+|\beta|+d-1)}}{ (1+\eps^{-1}\rho(y,x))^M}
\end{equation}
with $c_{32}=2^M d^{K/2}c_8$.
We now apply \propref{prop:5_3} to the integral in \eqref{eq:approx3_1} with $x_1=\xi$,
$g(y)=\newPhi_{N_\xi}^*(\xi\cdot \zeta(y))$, $N_1=N_\xi$, $\kappa_1=c_{31}\gamma_3 N_\xi^{-K}$ from \eqref{eq:local_Phi3},
$x_2=x$, $f(y)=\partial^\beta\Delta^{K/2}\til{F}_{\eps}(\zeta(y);x)$,
$\eps=\gamma_1/N_\xi$, $N_2=1/\eps\ge N_1$, $\kappa_2=c_{32}(\eps^{-1})^{K+|\beta|}$ from \eqref{eq:potential7},
and get \eqref{eq:approx3_2} with $c_{30}\ge c_3 c_{31} c_{32}$.

For the proof of \eqref{eq:approx3_5} we repeat the argument from the proof of \eqref{eq:approx3_2} and get
\begin{multline}\label{eq:approx3_4}
|g_2(\xi;x)-g_3(\xi;x)|\\
= \Big|\int_{\SS}\newPhi_{N_\xi}^*(\xi\cdot \zeta(y))F_{\eps}(\zeta(y)\cdot x)d\sigma(y)\Big|
\le c_{33} \frac{\gamma_3 N_\xi^{-K+d-1}}{(1+N_\xi\rho(\xi, x))^M}.
\end{multline}
Now, for $0\le|\beta|\le K-1$ we apply consecutively \eqref{eq:h2}, \eqref{eq:h3},
the fact that the operator $\Delta_0$ is symmetric,
\eqref{eq:approx3_4}, \eqref{eq:conv_1}, and \eqref{eq:small-1_powers} to obtain
\begin{align*}
&\Big|\int_{\SS} y^\beta \left[h_2(\xi;y)-h_3(\xi;y)\right]d\sigma(y)\Big|\\
&\qquad\qquad=\Big|\int_{\SS} y^\beta \Delta_0^{K/2}\left[g_2(\xi;y)-g_3(\xi;y)\right]d\sigma(y)\Big|\\
&\qquad\qquad=\Big|\int_{\SS} \left[g_2(\xi;y)-g_3(\xi;y)\right]\Delta_0^{K/2}y^\beta d\sigma(y)\Big|
\le c_0 c_6c_{33} \gamma_3 N_\xi^{-K},
\end{align*}
which yields \eqref{eq:approx3_5} with $c_{30}\ge c_0 c_6c_{33}$ because of $\gamma_1\le 1$.
Finally, we set $c_{30}:=\max\{c_3 c_{31} c_{32}, c_0 c_6c_{33}\}$.
\end{proof}

\begin{lem}\label{lem:1_5}
Assume $\xi\in\cX_j$, $j\in\NN$, $K\in 2\NN$, $M>K+d-1$, and
let $\gamma_1$, $\gamma_2$, $\gamma_3$, $\eps$, $\rrr_\xi$, $\til{h}_3$ be as in \lemref{lem:1_4}.
Let $\til{\theta}^\diamond_\xi(x)$ be defined by \eqref{eq:theta_diamond} if $d\ge 3$
or by \eqref{eq:theta_diamond2} if $d=2$
with  $x/|x|$ in the place of $x$.
Then for any $\gamma_4>0$ there exists $t_j>0$ such that
\begin{equation}\label{eq:approx4_1}
\big|\partial^\beta\big[\til{h}_3(\xi;x)-\til{\theta}^\diamond_\xi(x)\big]\big|
\le \frac{\gamma_4 N_\xi^{|\beta|+d-1}}{(1+N_\xi\pi)^M},\quad \forall x\in\SS,~0\le|\beta|\le K,
\end{equation}
\begin{equation}\label{eq:approx4_2}
\Big|\int_{\SS} y^\beta \left[h_3(\xi;y)-\theta^\diamond_\xi(y)\right]d\sigma(y)\Big|
\le \gamma_4 N_\xi^{-M},~0\le|\beta|\le K-1.
\end{equation}
\end{lem}

\begin{proof}
Inequalities \eqref{eq:approx4_1} and \eqref{eq:approx4_2} follow from \eqref{eq:h3}--\eqref{eq:theta_diamond2}
by approximating
the operator $\Delta_0^{K/2}\left(\zeta\cdot\nabla\right)^m$ by $\fL_t^{K/2}\fD^m_t(\zeta)$ as $t\to 0$ and
the infinite smoothness of $|(1+\eps)\zeta-x|^{-d+2}$ and $\log 1/|(1+\eps)\zeta-x|$ on the compact $\SS$.
\end{proof}

\begin{proof}[Proof of \thmref{thm:new_frame_diamond}]
This proof follows at once by Lemmas~\ref{lem:1_2}, \ref{lem:1_3}, \ref{lem:1_4}, and \ref{lem:1_5}
with the following selection of parameters:
\begin{equation}\label{param0}
\begin{array}{ll}
\gamma_1:=\min\{\gamma_0/(4c_{10}),1\},&
\gamma_2:=\min\{\gamma_0\gamma_1^{2K+1}/(4c_{20}),\gamma_1\},\\
\gamma_3:=\min\{\gamma_0\gamma_1^{2K}/(4c_{30}),1\},&
\gamma_4:=\gamma_0/4.
\end{array}
\end{equation}
\end{proof}

\subsection{Completion of the construction of new frames on $\SS$}\label{s6_5}

We use the scheme from Section~\ref{s3} to complete the construction of a pair of dual frames
$\{\theta_\xi\}_{\xi\in\cX}$, $\{\tilde\theta_\xi\}_{\xi\in\cX}$ on $\SS$,
where each frame element $\theta_\xi$ is a linear combination of a fixed number of shifts of the Newtonian kernel.

Following the definition $\psi_\xi(x):=C^\diamond_\xi \psi^\diamond_\xi(x)$
of the elements of old frame $\LL$ given in \eqref{eq:needlet_1},
we similarly construct the elements
$$
\theta_\xi(x):=C^\diamond_\xi \theta^\diamond_\xi(x), \quad \xi\in\cX_j,~~j\ge 1,
$$
of the new frame $\Theta=\{\theta_\xi\}_{\xi\in\cX}$.
In light of \eqref{eq:theta_diamond} and \eqref{eq:theta_diamond2} we have for $j\ge 1$
\begin{equation}\label{eq:theta}
\theta_\xi(x):=C^\diamond_\xi\kappa_{\eps,m,d}\sum_{\substack{\zeta\in\cZ_j\\ \rho(\zeta,\xi)\le\rrr_\xi}}
w_\zeta \newPhi_{N_\xi}(\xi\cdot \zeta)
\sum_{\ell=0}^m b_{\ell}\fL_{t_j}^{K/2}\fD^\ell_{t_j}(\zeta)  |x-a\zeta|^{2-d},~~~ d\ge 3,
\end{equation}
\begin{equation}\label{eq:theta2}
\theta_\xi(x):=C^\diamond_\xi\kappa_{\eps,m,2}\sum_{\substack{\zeta\in\cZ_j\\ \rho(\zeta,\xi)\le\rrr_\xi}}
w_\zeta \newPhi_{N_\xi}(\xi\cdot \zeta)
\sum_{\ell=1}^m b_{\ell}\fL_{t_j}^{K/2}\fD^\ell_{t_j}(\zeta) \ln\frac{1}{|x-a\zeta|},~~ d=2.
\end{equation}
The only frame element excluded from this definition is the constant function corresponding to $\xi\in\cX_0$.
For $\xi\in\cX_0$ we set $\theta_\xi(x):=\psi_\xi(x)\equiv 1$, $x\in\SS$.

In \thmref{thm:new_frame} below we collect some important properties of the new frame $\Theta$.
Its proof is based on \thmref{thm:new_frame_diamond} and the following lemma.

\begin{lem}\label{lem:1_1}
Let $\eta\in\cX$, $K\in 2\NN$, and $M>K+d-1$.

$(a)$
For any $\beta$, $0\le|\beta|\le K$, we have
\begin{equation}\label{eq:approx1_0}
\big|\partial^\beta \til{\newPsi}_{N_\eta}(\eta; x)\big|
\le c_{40} \frac{N_\eta^{|\beta|+d-1}}{(1+N_\eta\rho(\eta,x))^M},\quad \forall x\in\SS,
\end{equation}
with $\til{\newPsi}_{N_\eta}(\eta; x)$ given by \eqref{eq:constr_5}.

$(b)$
For any $\beta$, $0\le|\beta|\le K-1$,  we have
\begin{equation}\label{eq:approx1_2}
\Big|\int_{\SS} y^\beta \newPsi_{N_\eta}(\eta\cdot y)d\sigma(y)\Big|
\le c_{40}N_\eta^{-K}.
\end{equation}
\end{lem}

\begin{proof}
For $\eta\in\cX_0$ we have $\til{\newPsi}_{N_\eta}(\eta; x)\equiv 1/\omega_d$, $N_\eta=2^{-1}$,
and inequalities \eqref{eq:approx1_0}, \eqref{eq:approx1_2} are trivial.

Let $\eta\in\cX\backslash\cX_0$.
Inequality \eqref{eq:local_Psi} with $N=N_\eta$ yields \eqref{eq:approx1_0} with $c_{40}\ge c_5$.
For any multi-index $\beta$, $0\le|\beta|\le K-1$, we get from \eqref{eq:LB2},
the fact that the operator $\Delta_0$ is symmetric, \eqref{eq:local_Phi} with $\beta=0$ and $N=N_\eta$,
\eqref{eq:conv_1}, and \eqref{eq:small-1_powers} that
\begin{multline*}
\Big|\int_{\SS} y^\beta \newPsi_{N_\eta}(\eta\cdot y)d\sigma(y)\Big|
=\Big|\int_{\SS} y^\beta \Delta_0^{K/2}\newPhi_{N_\eta}(\eta\cdot y)d\sigma(y)\Big|\\
=\Big|\int_{\SS} \newPhi_{N_\eta}(\eta\cdot y)\Delta_0^{K/2}y^\beta d\sigma(y)\Big|\le c_0c_6c_4N_\eta^{-K},
\end{multline*}
which proves the lemma with $c_{40}:=\max\{c_5,c_0c_6c_4\}$.
\end{proof}

\begin{thm}\label{thm:new_frame}
Let $d\ge 2$, $K\in 2\NN$, $M>K+d-1$, and $0<\gamma_0\le 1$. Then there exist constants
$\gamma_1$, $\gamma_2$, $\gamma_3$, $\gamma_4>0$ depending only on $d, K, M, \gamma_0$,
and for every $\xi\in\cX_j$, $j\in\NN$, there exists $t_j>0$ depending only on $d, K, M, \gamma_0$, and $j$
such that:

$(a)$ The new frame $\Theta=\{\theta_\xi\}_\xi\in\cX$ is real-valued and satisfies
\begin{equation}\label{eq:new_frame_local}
\big|\partial^\beta \til{\theta}_\xi(x)\big|
\le c_{41} \frac{N_\xi^{|\beta|+(d-1)/2}}{(1+N_\xi\rho(\xi,x))^M},
\quad \forall x\in\SS,~\forall\xi\in\cX,~ \forall\beta,~ 0\le|\beta|\le K,
\end{equation}
\begin{equation}\label{eq:new_frame_small}
|\left\langle \psi_\eta,\psi_\xi-\theta_\xi\right\rangle| \le
c_{42} \gamma_0 \omega_{\xi,\eta}^{(K,M)},
\quad \forall\xi,\eta\in\cX,
\end{equation}
with $c_{41}=\max\{2^{(d-1)/2}(1+\pi/2)^M, c_9(c_{40}+\gamma_0)\}$, $c_{42}=c_1 c_{40}c_9^2$.

$(b)$ Every frame element $\theta_\xi$, $\xi\in\cX\backslash\cX_0$,
is a linear combination of at most $\tilde{n}$ shifts of Newtonian kernels, where
$\tilde{n}\le c_{43}\gamma_0^{(4K+3)(1-d)}$ with $c_{43}$ depending only on $d, K, M$.

$(c)$ Moreover, if
\begin{equation}\label{cond-g0}
\gamma_0 \le \frac{\tc_5^2}{4 c_{42}},
\end{equation}
where $c_{42}$ is from \eqref{eq:new_frame_small} and $\tc_5$ is from \eqref{psi-xi-norm2}, then
\begin{equation}\label{local-norm}
\|\theta_\xi\|_{L^p(\SS)} \sim N_\xi^{(d-1)(1/2-1/p)},\quad \frac{d-1}{M}<p\le\infty, \quad \forall \xi\in\cX,
\end{equation}
with uniformly bounded constants of equivalence for $p\ge \frac{d-1+\delta}{M}$, $\delta>0$.
\end{thm}

\begin{proof}
For $\xi\in\cX_0$ we have $\theta_\xi=1$, $N_\xi=2^{-1}$, and \eqref{eq:new_frame_local}
is satisfied with $c_{41}\ge 2^{(d-1)/2}(1+\pi/2)^M$. Also $\psi_\xi-\theta_\xi=0$
and inequality \eqref{eq:new_frame_small} is obvious for all $\eta\in\cX$.

Let $\xi\in\cX\backslash\cX_0$. Then \eqref{eq:small-0}, \eqref{eq:approx1_0} with $\xi$ in the place of $\eta$,
and \eqref{eq:needlet_2} imply \eqref{eq:new_frame_local} with $c_{41}\ge c_9(c_{40}+\gamma_0)$.

For the proof of \eqref{eq:new_frame_small} first assume that $N_\xi\ge N_\eta$.
We apply \propref{prop:5_1} with $g=\psi_\eta=C^\diamond_\eta \psi^\diamond_\eta$, $x_1=\eta$,
$f=\psi_\xi-\theta_\xi=C^\diamond_\xi (\psi^\diamond_\xi-\theta^\diamond_\xi)$, $x_2=\xi$.
\lemref{lem:1_1} implies that \eqref{eq:1a} is satisfied with $N_1=N_\eta$, $\kappa_1=c_{40}C^\diamond_\eta$, and
\thmref{thm:new_frame_diamond} implies that \eqref{eq:2a}, \eqref{eq:3a}
are satisfied with $N_2=N_\xi$, $\kappa_2=\gamma_0 C^\diamond_\xi$.
Now, \eqref{eq:4a} and \eqref{eq:needlet_2}
give
\begin{equation*}
|\left\langle \psi_\eta,\psi_\xi-\theta_\xi\right\rangle|
\le c_1 \frac{c_{40}
C_9 N_\eta^{-\frac{d-1}{2}}\gamma_0 C_9 N_\xi^{-\frac{d-1}{2}}(N_\eta/N_\xi)^{K}N_\eta^{d-1}}{(1+N_\eta\rho(\xi,\eta))^M}
\le c_{42} \gamma_0 \omega_{\xi,\eta}^{(K,M)},
\end{equation*}
with $c_{42}:=c_1 c_{40}c_9^2$, which establishes \eqref{eq:new_frame_small} in this case.

Second, assume that $N_\xi\le N_\eta$.
Here, we apply \propref{prop:5_1} with
$g=\psi_\xi-\theta_\xi$,
$f=\psi_\eta$, and then
\eqref{eq:new_frame_small} follows similarly as above.
This completes the proof of $(a)$.

The number of $\zeta\in\cZ_j$ in \eqref{eq:theta} (or \eqref{eq:theta2}) can be estimated as follows. From
\begin{equation*}
	\bigcup_{\zeta\in\cZ_j:\rho(\zeta,\xi)\le\rrr_\xi}\nA_\zeta\subset B(\xi,\rrr_\xi+\gamma_2/N_\xi)
\end{equation*}
we find that the total volume covered by $\nA_\zeta$ does not exceed $c((\gamma_3^{-1}+\gamma_2)/N_\xi)^{d-1}$.
From this estimate and \eqref{eq:maximal_net} with $\gamma=\gamma_2$ we get that
the number of $\zeta\in\cZ_j$ in \eqref{eq:theta} is at most $c(\gamma_3\gamma_2)^{1-d}=c_{44}\gamma_0^{(4K+3)(1-d)}$
in light of \eqref{param0}.

Clearly, the number of translation terms in $\fL_{t_j}^{K/2}$ is at most $(d(d-1)/2)^{K/2}+1$.
The number of translation terms in $\fD^\ell_{t_j}(\zeta)$ is $\ell+1$, $\ell=0,1,\dots,m$,
and every such term is also a term for $\fD^m_{t_j}(\zeta)$.
This leads to the estimate $\tilde{n}\le c_{43}\gamma_0^{(4K+3)(1-d)}$ with $c_{43}:=c_{44}[(d(d-1)/2)^{K/2}+1](m+1)$
for the number of shifts of Newtonian kernels used in \eqref{eq:theta}
or in \eqref{eq:theta2}. Thus, the proof of $(b)$ is complete.

We now establish $(c)$. Inequality \eqref{eq:new_frame_local} with $|\beta|=0$ along with \eqref{eq:conv_1} imply
\begin{equation}\label{local-norm_1}
\|\theta_\xi\|_{L^p} \le c_0^{1/p} c_{41} N_\xi^{(d-1)(1/2-1/p)}
\end{equation}
with $c_0=c(d)/\delta$.

To prove the estimate in the other direction:
\begin{equation}\label{local-norm_2}
\|\theta_\xi\|_{L^p} \ge c_{45} N_\xi^{(d-1)(1/2-1/p)}
\end{equation}
we first consider the case $p=2$. From
\begin{equation*}
\langle \theta_\xi,\theta_\xi\rangle
=\langle \psi_\xi,\psi_\xi\rangle-2\langle \psi_\xi,\psi_\xi-\theta_\xi\rangle+\langle \psi_\xi-\theta_\xi,\psi_\xi-\theta_\xi\rangle,
\end{equation*}
\eqref{psi-xi-norm2}, \eqref{eq:new_frame_small}, and \eqref{cond-g0},  we get
\begin{equation*}
\langle \theta_\xi,\theta_\xi\rangle
\ge\langle \psi_\xi,\psi_\xi\rangle-2|\langle \psi_\xi,\psi_\xi-\theta_\xi\rangle|
\ge \tc_5^2-2\gamma_0 c_{42}\omega_{\xi,\xi}^{(K,M)}\ge \tc_5^2/2,
\end{equation*}
due to $\omega_{\xi,\xi}^{(K,M)}=1$ (see \eqref{eq:omega_xi_eta}).
This gives \eqref{local-norm_2} for $p=2$ with $c_{45}=\tc_5/\sqrt{2}$.

Now, consider the case $p<2$.
Using \eqref{local-norm_2} with $p=2$ and \eqref{local-norm_1} with $p=\infty$ we obtain
\begin{equation*}
\tc_5^2/2\le \|\theta_\xi\|_{L^2}^2\le \|\theta_\xi\|_{L^\infty}^{2-p}\|\theta_\xi\|_{L^p}^p
\le \left(c_{41} N_\xi^{(d-1)/2}\right)^{2-p}\|\theta_\xi\|_{L^p}^p,
\end{equation*}
which proves \eqref{local-norm_2} for $p<2$ with $c_{45}=(c_{41}^{p-2}\tc_5^2/2)^{1/p}$.

Finally, consider the case $2<p\le\infty$.
Using H\"{o}lder's inequality, \eqref{local-norm_2} with $p=2$ and \eqref{local-norm_1} with $1\le p'<2$ we obtain
\begin{equation*}
\tc_5^2/2\le \|\theta_\xi\|_{L^2}^2\le \|\theta_\xi\|_{L^{p'}}\|\theta_\xi\|_{L^p}
\le c_0^{1/p'}c_{41} N_\xi^{(d-1)(1/2-1/p')}\|\theta_\xi\|_{L^p},
\end{equation*}
which proves \eqref{local-norm_2} for $p>2$ with $c_{45}=c_0^{-1+1/p}c_{41}^{-1}\tc_5^2/2$.
This proves \eqref{local-norm_2} for all $p$ and completes the proof of the theorem.
\end{proof}

\begin{rem}
All poles of the Newtonian kernels in \eqref{eq:theta} and in \eqref{eq:theta2} are placed on $m+1$ concentric spheres of radii
	$1+\gamma_1 N_\xi^{-1}+kt_j$,~$k=0,1,\dots,m$.
	On every such sphere the poles are located in
the spherical cap of radius $(\gamma_3N_\xi)^{-1}+t_jK/2$ centred at $(1+\gamma_1 N_\xi^{-1}+kt_j)\xi$ .
\end{rem}

Our next step is to show that the above defined system $\Theta$
coupled with the dual system $\tilde\Theta=\{\tilde\theta_\xi\}_{\xi\in\cX}$ constructed by the scheme
from \S\ref{subsec:new-frame} form a pair of frames for all Besov and Triebel-Lizorkin space $\cB^{sq}_p$, $\cF^{sq}_p$
with parameters $(s, p, q)\in \QQ(A)$ for a fixed $A>1$ with $\QQ(A)$ defined in \eqref{indices-1}.

\begin{thm}\label{thm:prop-frame}
Assume $d\ge 2$, $A>1$, and
let $\Theta=\{\theta_\xi\}_{\xi\in\cX}$ be the real-valued system constructed
in \eqref{eq:theta} or \eqref{eq:theta2}, where
\begin{equation}\label{cond-K}
K \ge \left\lceil Ad\right\rceil, K\in 2\NN, \quad M = K+d.
\end{equation}
If the constant $\gamma_0$ in the construction of $\{\theta_\xi\}_{\xi\in\cX}$
is sufficiently small, namely,
\begin{equation}\label{cond-g}
\gamma_0 \le \frac{\epsilon}{c_{42} C_9},
\end{equation}
where $\epsilon$ is from \eqref{oper-eps}, $c_{42}$ is from \eqref{eq:new_frame_small}, and
$C_9$ is from \thmref{thm:almost-diag},
then:

$(a)$ The synthesis operator $T_\theta$ defined by $T_\theta h:= \sum_{\xi\in \cX}h_\xi\theta_\xi$
on sequences of complex numbers $h=\{h_\xi\}_{\xi\in\cX}$ is bounded as a map
$T_\theta: \bb_p^{sq} \mapsto \cB_p^{sq}$, uniformly with respect to to $(s, p, q)\in\QQ(A)$.

$(b)$ The operator
\begin{equation}\label{def:operator-T}
Tf:=\sum_{\xi\in\cX} \langle f, \psi_\xi\rangle \theta_\xi=T_{\theta}S_{\psi}f,
\end{equation}
is invertible on $\cB_p^{sq}$ and $T$, $T^{-1}$ are bounded on $\cB_p^{sq}$,
uniformly with respect to $(s, p, q)\in\QQ(A)$.

$(c)$ For $(s, p, q)\in \QQ(A)$ the dual system $\tilde\Theta=\{\tilde\theta_\xi\}_{\xi\in\cX}$ consists of bounded linear functionals
on $\cB_p^{sq}$ defined by
\begin{equation}\label{def-f-dual-f}
\tilde\theta_\xi(f)=
\langle f, \tilde\theta_\xi\rangle
:= \sum_{\eta\in\cX}
\langle T^{-1}\psi_\eta, \psi_\xi\rangle \langle f, \psi_\eta\rangle
\quad \hbox{for}\;\; f\in \cB_p^{sq},
\end{equation}
with the series converging absolutely.
Also, the analysis operator
$$
S_{\tilde\theta}: \cB_p^{sq} \mapsto \bb_p^{sq},\;
S_{\tilde\theta} = S_{\psi}T^{-1}T_{\psi}S_{\psi}=S_{\psi}T^{-1},
$$
is uniformly bounded with respect to $(s, p, q)\in\QQ(A)$.
Moreover, $\{\theta_\xi\}_{\xi\in\cX}$, $\{\tilde\theta_\xi\}_{\xi\in\cX}$
form a pair of dual frames for $\cB_p^{sq}$ in the following sense:
For any $f\in \cB_p^{sq}$
\begin{equation}\label{frame-B}
f=\sum_{\xi\in\cX} \langle f, \tilde\theta_\xi\rangle \theta_\xi
\quad\hbox{and}\quad
\|f\|_{\cB_p^{sq}} \sim \|\{\langle f, \tilde\theta_\xi\rangle\}\|_{\bb_p^{sq}},
\end{equation}
where the convergence is unconditional in $\cB_p^{sq}$.

Furthermore, $(a)$, $(b)$, and $(c)$ hold true when
$\cB_p^{sq}$, $\bb_p^{sq}$ are replaced by $\cF_p^{sq}$, $\ff_p^{sq}$, respectively.
\end{thm}

\begin{proof}
The parameters $K$ and $M$ from \eqref{cond-K} satisfy \eqref{cond-KM} with $\delta=1$
for all $(s, p, q)\in \QQ(A)$ and we can apply \thmref{thm:almost-diag}.
From estimate \eqref{eq:new_frame_small} in \thmref{thm:new_frame}, \lemref{lem:0.3}, and \thmref{thm:almost-diag}
we obtain that $\{\theta_\xi\}_{\xi\in\cX}$
satisfies \eqref{oper-eps} due to $\gamma_0 c_{42} C_9\le \epsilon$.
Also, all conditions on the old frame laid in Subsection~\ref{set-up} are satisfied as shown in Subsection~\ref{s3_4}.
Now, we apply \lemref{lem:T-bounded} and \lemref{lem:T-invert} to get $(a)$ and $(b)$.
Finally, \thmref{thm:new-frames} implies  $(c)$.
\end{proof}

\subsection{Frames on $B^d$ in terms of shifts of the Newtonian kernel}\label{subsec:frames-Bd}

Note that by the fact that each frame element $\theta_\xi$, $\xi\in\cX$, is represented as
a finite linear combination of shifts of the Newtonian kernel it readily follows
that $\theta_\xi(x)$ defined in \eqref{eq:theta} or \eqref{eq:theta2} as a function of $x\in B^d$ is harmonic on $B^d$.
This leads immediately to the conclusion that $\{\theta_\xi\}_{\xi\in\cX}$, $\{\tilde\theta_\xi\}_{\xi\in\cX}$
is a pair of dual frames for harmonic  Besov and Triebel-Lizorkin spaces in the sense
of the following

\begin{thm}\label{thm:harmonic-frame}
Under the hypothesis of Theorem~\ref{thm:prop-frame}
let $\{\theta_\xi\}_{\xi\in\cX}$, $\{\tilde\theta_\xi\}_{\xi\in\cX}$ be the frames from Theorem~\ref{thm:prop-frame}.
Then for any $(s, p, q)\in\QQ(A)$ and $U\in B_p^{sq}(\HHH)$
\begin{equation}\label{harmonic-frame}
U(x)=\sum_{\xi\in\cX} \langle f_U, \tilde\theta_\xi\rangle \theta_\xi(x),~x\in B^d,
\quad\hbox{and}\quad
\|U\|_{B_p^{sq}} \sim \|\{\langle f_U, \tilde\theta_\xi\rangle\}\|_{\bb_p^{sq}}.
\end{equation}
Here $f_U$ is the boundary value of $U$ $($see Proposition~\ref{prop:identify}$)$
and the series converges uniformly on every compact subset of $B^d$.
Furthermore, the above holds true when
$B_p^{sq}(\HHH)$, $\bb_p^{sq}$ are replaced by $F_p^{sq}(\HHH)$, $\ff_p^{sq}$, respectively.
\end{thm}

This theorem follows immediately by Theorems~\ref{thm:prop-frame}, \ref{thm:equiv-norms-F}, and \ref{thm:equiv-norms-B}.

\section{Nonlinear approximation from shifts of the Newtonian kernel}\label{s7}

The primary goal of this article is to establish a Jackson type estimate for nonlinear
$n$-term approximation of harmonic functions on $B^d$ from shifts of the Newtonian kernel in the harmonic Hardy space $\HHH^p(B^d)$.
For any $n\ge 1$ write
\begin{equation}\label{def:N-n-d}
\NNN_n:=\Big\{G: G(x) =a_0+\sum_{\nu=1}^{n} \frac{a_\nu}{|x-y_\nu|^{d-2}}, \; |y_\nu|>1, a_\nu\in \CC\Big\},
\;\;\hbox{if $d>2$},
\end{equation}
and
\begin{equation}\label{def:N-n-2}
\NNN_n:=\Big\{G: G(x) =a_0+\sum_{\nu=1}^{n} a_\nu\ln \frac{1}{|x-y_\nu|}, \; |y_\nu|>1, a_\nu\in \CC\Big\},
\;\;\hbox{if $d=2$}.
\end{equation}
Observe that the points $\{y_\nu\}$ above may vary with $G$ and hence $\NNN_n$ is nonlinear.

Let $\fB$ be one of the spaces $\HHH^p(B^d)$, $B_p^{0q}(\HHH)$, or $F_p^{0q}(\HHH)$, $0<p,q<\infty$.
Given $U\in \fB$ we define
\begin{equation}\label{def:best-app}
E_n(U)_{\fB}:=\inf_{G\in\NNN_n}\|U-G\|_{\fB}.
\end{equation}
We call $E_n(U)_{\fB}$ the best nonlinear $n$-term approximation of $U$ from shifts of the Newtonian kernel
in the harmonic space $\fB$.

We now come to the main result in this article.

\begin{thm}\label{thm:Jackson}
Let $s>0$, $0<p<\infty$, and $1/\tau=s/(d-1)+1/p$.
If $U\in B_\tau^{s\tau}(\HHH)$, then $U\in \HHH^p(B^d)$ and
\begin{equation}\label{jackson-H}
E_n(U)_{\HHH^p(B^d)} \le cn^{-s/(d-1)}\|U\|_{B_\tau^{s\tau}(\HHH)},\quad n\ge 1,
\end{equation}
where the constant $c>0$ depends only on $p, s, d$.
\end{thm}

\thmref{thm:Jackson} is an immediate consequence of
\thmref{thm:Jackson-B-F} bellow with $q=2$ and \thmref{thm:identify-Hp}.

Our approach to approximating a harmonic function $U$ on $B^d$ amounts
to first establishing a Jackson estimate for nonlinear $n$-term approximation of its boundary value function $f_U$
on $\SS$ from the frame elements $\{\theta_\xi\}_{\xi\in\cX}$ constructed in Section~\ref{s6}
and then considering the harmonic extension to $B^d$ of the approximant.
The gist of our approximation method is that each frame element $\theta_\xi$
is a linear combination of a fixed number of shifts of the Newtonian kernel.

\subsection{Nonlinear $n$-term frame approximation on $\SS$}

Let $\{\theta_\xi\}_{\xi\in\cX}$ be the frame constructed in Section~\ref{s6}
with parameters to be specified.
Denote by $\Sigma_n$ the set of all functions $g$ on $\SS$ of the form
\begin{equation*}
g=\sum_{\xi\in Y_n} a_\xi\theta_\xi, \quad a_\xi\in \CC,
\end{equation*}
where $Y_n\subset \cX$ is an index set such that $\# Y_n\le n$.
Define
\begin{equation}\label{best-app}
\sigma_n(f)_\fB:=\inf_{g\in\Sigma_n}\|f-g\|_{\fB},
\end{equation}
where $\fB$ is one of the spaces $L^p(\SS)$, $\cB_p^{0q}(\SS)$, or $\cF_p^{0q}(\SS)$, $0<p,q<\infty$.

As one can expect the smoothness spaces on $\SS$ governing this kind of approximation
should be the Besov spaces $\cB_\tau^{s\tau}(\SS)$ with $s$ and $\tau$ as in Theorem~\ref{thm:Jackson}.
For~this to be true, however, $\{\theta_\xi\}_{\xi\in\cX}$, $\{\tilde\theta_\xi\}_{\xi\in\cX}$
have to provide frame decomposition of all spaces involved just as in Theorem~\ref{thm:prop-frame}.

\smallskip

\noindent
{\bf Assumptions:}
The construction of the frames $\{\theta_\xi\}_{\xi\in\cX}$, $\{\tilde\theta_\xi\}_{\xi\in\cX}$
in Section~\ref{s6} depends on the parameters $A$, $K$, $M$, and $\gamma_0$.
The main parameter is $A$. In~light of Theorems~\ref{thm:new_frame} and \ref{thm:prop-frame}
we require that
\begin{equation*}
0<s\le A, \;\; A^{-1}\le p\le A, \;\; A^{-1}\le q <\infty, \;\; A^{-1}\le s/(d-1)+1/p\le  A,\;\; A>1,
\end{equation*}
which reduces to the following principle conditions:
\begin{equation}\label{main-cond}
0<s\le A, \;\; 0<p \le A, \;\; A^{-1}\le q <\infty, \;\; s/(d-1)+1/p\le  A,\;\; A>1.
\end{equation}
The parameters $K$, $M$ are secondary and are defined as
\begin{equation}\label{param-KM}
K:= 2\left\lceil \frac{Ad}{2}\right\rceil,\quad M:=K+d.
\end{equation}
Further, $\gamma_0>0$ is a ``small'' parameter
and using the notation from Theorems~\ref{thm:new_frame} and \ref{thm:prop-frame}
we fix it as
\begin{equation}\label{param-gamma0}
\gamma_0:=\min\left\{ \frac{\epsilon}{c_{42} C_9},\frac{\tc_5^2}{4 c_{42}}\right\}.
\end{equation}
It is easy to see that $\gamma_0$ depends only on $A$, $d$, and the ``old" frame $\{\psi_\xi\}_{\xi\in\cX}$,
described in \S\ref{subsec:frame-SS}.

Conditions \eqref{main-cond} can be viewed in two ways:

(a) Given $A>1$ consider \eqref{main-cond} as conditions on $s, p, q$;

(b) Given $s, p, q$ consider \eqref{main-cond} as conditions on $A$.
For \thmref{thm:Jackson} we take $A$ to be the smallest number satisfying \eqref{main-cond}.

Either way conditions \eqref{main-cond} coupled with \eqref{param-KM}--\eqref{param-gamma0}
imply that the hypotheses of Theorems~\ref{thm:new_frame} and \ref{thm:prop-frame} are obeyed
and hence the conclusions of these theorems are valid for
the frames $\{\theta_\xi\}_{\xi\in\cX}$, $\{\tilde\theta_\xi\}_{\xi\in\cX}$
and the spaces $\cB_p^{0q}(\SS)$, $\cF_p^{0q}(\SS)$, and $\cB_\tau^{s\tau}(\SS)$.

Although we are mainly interested in approximation of functions
in the harmonic Hardy space $\HHH^p(B^d)$ or their boundary values in $\cF_p^{02}(\SS)$,
to put it in perspective we shall examine the approximation process at hand in the slightly
more general  spaces $\cF_p^{0q}(\SS)$ and $\cB_p^{0q}(\SS)$.

\subsubsection{Nonlinear $n$-term frame approximation in Triebel-Lizorkin spaces on $\SS$}\label{ss:TL}

\begin{thm}\label{thm:Jackson-S-F}
Assume $A>1$ and let the parameters $K$, $M$, and $\gamma_0$ be defined as in \eqref{param-KM}--\eqref{param-gamma0}.
Let $s>0$, $0<p,q<\infty$, and $1/\tau=s/(d-1)+1/p$,
and assume that $s, p, q$ satisfy conditions \eqref{main-cond}.
If $f\in \cB_\tau^{s\tau}(\SS)=\cF_\tau^{s\tau}(\SS)$, then $f\in \cF_p^{0q}(\SS)$
and with the notation from \eqref{best-app}
\begin{equation}\label{jackson-S-F}
\sigma_n(f)_{\cF_p^{0q}(\SS)} \le cn^{-s/(d-1)}\|f\|_{\cB_\tau^{s\tau}(\SS)},\quad n\ge 1,
\end{equation}
where the constant $c>0$ depends only on $A, d$.
\end{thm}

The proof of this theorem depends on the following simple

\begin{lem}\label{lem:est-F}
For any $0<p<\infty$ and any \emph{finite} subset $Y$ of $\cX$ we have
\begin{equation}\label{est-F_q}
\Big\|\sum_{\xi\in Y}|\nB_\xi|^{-1/p}\ONE_{\nB_\xi}\Big\|_{L^{p}}\le c(\# Y)^{1/p}
\end{equation}
with $\nB_\xi=B(\xi,\gamma N_\xi^{-1})$ from Definition~\ref{def:TL}.
\end{lem}

\begin{proof}
First, we observe that for any $j\in\NN$ and any $x\in\SS$
the number of points $\eta\in\cX_j$ such that $x\in B_\eta=B(\eta,\delta_j)$ does not exceed a constant $c(d)$.
Indeed, the spherical caps $\{B(\eta,\delta_j/2)\}_{\eta\in\cX_j}$ are mutually disjoint
and $x\in B_\eta$ implies $B(\eta,\delta_j/2)\subset B(x,3\delta_j/2)$, which together with \eqref{sph_cap2} justifies the observation.

Given $Y\subset \cX$ we set
$\Omega:= \cup_{\xi\in Y}B_\xi$
and
$h(x):= \min \big\{|B_\xi|: x\in B_\xi, \xi\in Y\big\}$ if $x\in \Omega$.
Clearly, if $x\in B_\xi$ for some $\xi\in Y$, then the above observation and \eqref{sph_cap2} imply

\begin{equation*}
\sum_{\substack{\eta\in Y: x\in B_\eta\\ |B_\eta|\ge |B_\xi|}}
(|B_\xi|/|B_\eta|)^{1/p} \le c\sum_{j=0}^{\infty} 2^{-j(d-1)/p}=c_\star<\infty,
\end{equation*}
yielding
\begin{equation*}
\sum_{\xi\in Y} |B_\xi|^{-1/p}\ONE_{B_\xi}(x)
\le c_\star h(x)^{-1/p}
\le c_\star\Big(\sum_{\xi\in Y} |B_\xi|^{-1}\ONE_{B_\xi}(x)\Big)^{1/p},
\quad x\in \Omega.
\end{equation*}
In turn, this readily implies \eqref{est-F_q}.
\end{proof}

\begin{proof}[Proof of Theorem~\ref{thm:Jackson-S-F}]
Under the hypotheses of Theorem~\ref{thm:Jackson-S-F} assume $f\!\in \!\cB_\tau^{s\tau}(\SS)$.
Embedding \eqref{eq:12} implies $f\in \cF_p^{0q}(\SS)$.

As we already alluded to above the frames $\{\theta_\xi\}_{\xi\in\cX}$, $\{\tilde\theta_\xi\}_{\xi\in\cX}$ are well defined
and the conclusions of Theorems~\ref{thm:new_frame} and \ref{thm:prop-frame} are valid
for the spaces $\cF_p^{0q}(\SS)$ and $\cB_\tau^{s\tau}(\SS)$.
In particular, condition \eqref{param-gamma0} implies \eqref{cond-g0}
and conditions \eqref{main-cond}--\eqref{param-KM} implies $(d-1+\delta)/M\le p$ with $\delta=1$.
Hence the assumptions of \thmref{thm:new_frame}~(c) are fulfilled and from \eqref{local-norm} and \eqref{sph_cap} we get
\begin{equation}\label{local-norm2}
\|\theta_\xi\|_{L^p}:=\|\theta_\xi\|_{L^p(\SS)} \sim N_\xi^{(d-1)(1/2-1/p)}\sim |\nB_\xi|^{1/p-1/2},\quad \xi\in\cX,
\end{equation}
with constants of equivalence depending only on $d$ and $A$.
Set $a_\xi:=\langle f, \tilde\theta_\xi\rangle$, $\xi\in \cX$.
From \eqref{local-norm2} and \eqref{def-b-space}--\eqref{def-f-space} we obtain
\begin{equation}\label{eq:Nf_0}
\|a\|_{\bb_\tau^{s\tau}(\SS)}\sim \|a\|_{\ff_\tau^{s\tau}(\SS)}\sim
\Big(\sum_{\xi\in \cX} (\|a_\xi\theta_\xi\|_{L^p})^\tau\Big)^{1/\tau}=:N(f).
\end{equation}
We may assume $N(f)>0$.
Denote
\begin{equation}\label{def-Y-nu}
\cY_\nu:=\Big\{\xi\in\cX: 2^{-\nu}N(f)<\|a_\xi\theta_\xi\|_{L^p}\le 2^{-\nu+1}N(f)\Big\},\quad \nu\in\NN.
\end{equation}
Then
\begin{equation*}
\bigcup_{\nu\le\mu}\cY_\nu=\Big\{\xi: 2^{-\mu}N(f)<\|a_\xi\theta_\xi\|_{L^p}\Big\},\quad \mu\in\NN,
\end{equation*}
and hence
\begin{equation}\label{eq:Nf_1}
\#\cY_\mu\le\sum_{\nu\le\mu}\#\cY_\nu=\#\Big(\bigcup_{\nu\le\mu}\cY_\nu \Big)\le 2^{\mu\tau}.
\end{equation}

Set $F_\mu:=\sum_{\xi\in \cY_\mu}|a_\xi|^q|\nB_\xi|^{-q/2}\ONE_{\nB_\xi}$.
We next show that for $m\ge 0$
\begin{equation}\label{eq:Nf_3}
\Big\|\Big(\sum_{\mu\ge m+1}F_\mu\Big)^{1/q}\Big\|_{L^p(\SS)}
\le c2^{-m\tau s/(d-1)}N(f).
\end{equation}
To this end we first estimate $\|F_\mu\|_{L^{p/q}}$.
Using \eqref{def-Y-nu}, \eqref{local-norm2},
\lemref{lem:est-F} with $p$ replaced by $p/q$, and \eqref{eq:Nf_1} we obtain
\begin{align}
\|F_\mu\|_{L^{p/q}}\label{est-F-norm}
&= \Big\|\sum_{\xi\in\cY_\mu}\big(|a_\xi||B_\xi|^{1/p-1/2}\big)^q|\nB_\xi|^{-q/p}\ONE_{\nB_\xi}\Big\|_{L^{p/q}} \notag
\\
&\le c 2^{-q(\mu-1)} N(f)^q
 \Big\|\sum_{\xi\in \cY_\mu}|\nB_\xi|^{-q/p}\ONE_{\nB_\xi}\Big\|_{L^{p/q}}
\le c 2^{-q\mu} N(f)^q(\# \cY_\mu)^{q/p}
\\
&\le c2^{-q\mu(1-\tau/p)} N(f)^q
= c 2^{-q\mu\tau s/(d-1)} N(f)^q. \notag
\end{align}
To prove \eqref{eq:Nf_3} we consider two cases.
If $q\le p$, then using \eqref{est-F-norm}
\begin{align*}
\Big\|\Big(\sum_{\mu\ge m+1}F_\mu\Big)^{1/q}\Big\|_{L^p}^q
= \Big\|\sum_{\mu\ge m+1}F_\mu\Big\|_{L^{p/q}}
&\le \sum_{\mu\ge m+1}\|F_\mu\|_{L^{p/q}}
\\
\le c\sum_{\mu\ge m+1} 2^{-q\mu\tau s/(d-1)} N(f)^q
&\le c 2^{-qm\tau s/(d-1)} N(f)^q,
\end{align*}
which implies \eqref{eq:Nf_3}.
In the case $q>p$ using that $p/q<1$ and \eqref{est-F-norm} we have
\begin{align*}
\Big\|\Big(\sum_{\mu\ge m+1}F_\mu\Big)^{1/q}\Big\|_{L^p}^p
= \Big\|\sum_{\mu\ge m+1}F_\mu\Big\|_{L^{p/q}}^{p/q}
&\le \sum_{\mu\ge m+1}\|F_\mu\|_{L^{p/q}}^{p/q}
\\
\le c\sum_{\mu\ge m+1} 2^{-p\mu\tau s/(d-1)} N(f)^p
&\le c 2^{-pm\tau s/(d-1)} N(f)^p,
\end{align*}
yielding again \eqref{eq:Nf_3}.

Choose $m\ge 0$ so that
$2^{m\tau}\le n<2^{(m+1)\tau}$
and denote $\cZ_m:=\bigcup_{\nu\le m}\cY_\nu$.
Also, set $a^\star_\xi:= a_\xi$ if $\xi\in \cX\setminus\cZ_m$ and
$a^\star_\xi:= 0$ if $\xi\in \cZ_m$.
By \eqref{eq:Nf_1} it follows that $\# \cZ_m \le 2^{m\tau}\le n$.
This,
the frame representation \eqref{frame-B} for $f\in\cF_p^{0q}(\SS)$,
and the boundedness of the synthesis operator $T_\theta$ from \thmref{thm:prop-frame}~(b) yield
\begin{align*}
\sigma_n(f)_{\cF_p^{0q}}
&\le \Big\|f-\sum_{\xi\in\cZ_m}a_\xi\theta_\xi\Big\|_{\cF_p^{0q}}
= \Big\|\sum_{\xi\in\cX\setminus\cZ_m}a_\xi\theta_\xi\Big\|_{\cF_p^{0q}}
\le c\|a^\star\|_{\ff_p^{0q}}
\\
&= c\Big\|\Big(\sum_{\xi\in\cX\setminus\cZ_m}|a_\xi|^q|\nB_\xi|^{-q/2}\ONE_{\nB_\xi}\Big)^{1/q}\Big\|_{L^p}
= c\Big\|\Big(\sum_{\mu\ge m+1}F_\mu\Big)^{1/q}\Big\|_{L^p}.
\end{align*}
Finally, we use \eqref{eq:Nf_3}, \eqref{eq:Nf_0}, and Theorem~\ref{thm:prop-frame} (c)
to obtain
\begin{align*}
\sigma_n(f)_{\cF_p^{0q}}
\le c2^{-m\tau s/(d-1)}N(f)
\le cn^{-s/(d-1)}\|a\|_{\bb_\tau^{s\tau}}
\le cn^{-s/(d-1)}\|f\|_{\cB_\tau^{s\tau}(\SS)},
\end{align*}
which confirms \eqref{jackson-S-F}.
\end{proof}

From \thmref{thm:Jackson-S-F} and the equivalence $\cF_p^{02}(\SS)\sim L^p(\SS)$ for $1<p<\infty$ we immediately get

\begin{cor}\label{cor:Jackson-S}
Assume $A>1$ and let the parameters $K$, $M$, and $\gamma_0$ be defined as in \eqref{param-KM}--\eqref{param-gamma0}.
Let $s>0$, $1<p<\infty$, and $1/\tau=s/(d-1)+1/p$,
and assume that $s, p$ satisfy conditions \eqref{main-cond} with $q=2$.
If $f\in \cB_\tau^{s\tau}(\SS)$, then $f\in L^p(\SS)$ and
\begin{equation}\label{jackson-S}
\sigma_n(f)_{L^p(\SS)} \le cn^{-s/(d-1)}\|f\|_{\cB_\tau^{s\tau}(\SS)},~~n\ge 1,
\end{equation}
where the constant $c>0$ depends only on $A, d$.
\end{cor}

\subsubsection{Nonlinear $n$-term frame approximation in Besov spaces on $\SS$}\label{ss:Besov}

\begin{thm}\label{thm:Jackson-S-B}
Assume $A>1$ and let the parameters $K$, $M$, and $\gamma_0$ be as in \eqref{param-KM}--\eqref{param-gamma0}.
Let $s>0$, $0<p, q<\infty$, $1/\tau=s/(d-1)+1/p$, and $q\ge \tau$,
and assume that $s, p, q$ satisfy conditions \eqref{main-cond}.
If $f\in \cB_\tau^{s\tau}(\SS)$, then $f\in \cB_p^{0q}(\SS)$ and for every $n\ge 1$ we have
\begin{equation}\label{jackson-S-B}
\sigma_n(f)_{\cB_p^{0q}(\SS)} \le cn^{-s/(d-1)}\|f\|_{\cB_\tau^{s\tau}(\SS)},\quad p\le q,
\end{equation}
\begin{equation}\label{jackson-S-B2}
\sigma_n(f)_{\cB_p^{0q}(\SS)} = o(n^{1/q-1/\tau})\|f\|_{\cB_\tau^{s\tau}(\SS)},\quad \tau\le q<p,
\end{equation}
where the constant $c>0$ depends only on $A, d$.
\end{thm}

For the proof of this theorem we shall utilize inequality (6.7) from \cite{NPW2}
given in the following

\begin{lem}\label{lem:est-B}
Let $0<\tau< p<\infty$ and $x_1\ge x_2\ge \dots \ge 0$. Then for every $n\in\NN$ we have
\begin{equation*}
\Big(\sum_{k=n+1}^\infty x_k^p\Big)^{1/p}\le n^{1/p-1/\tau}\Big(\sum_{k=1}^\infty x_k^\tau\Big)^{1/\tau}.
\end{equation*}
\end{lem}

\begin{proof}[Proof of Theorem~\ref{thm:Jackson-S-B}]
Assume that the hypotheses of Theorem~\ref{thm:Jackson-S-B} are obeyed and let $f\in \cB_\tau^{s\tau}(\SS)$.
Embeddings \eqref{eq:11} and \eqref{eq:2} imply $f\in \cB_p^{0q}(\SS)$.

As above the frames $\{\theta_\xi\}_{\xi\in\cX}$, $\{\tilde\theta_\xi\}_{\xi\in\cX}$ are well defined
and the conclusions of Theorems~\ref{thm:new_frame} and \ref{thm:prop-frame} are valid
for the spaces $\cB_p^{0q}(\SS)$ and $\cB_\tau^{s\tau}(\SS)$.
Denote $a_\xi:=\langle f, \tilde\theta_\xi\rangle$, $\xi\in \cX$.
Recall the equivalence \eqref{local-norm2} which holds in this case.

\begin{sloppypar}
Let $\big\{\|a_{\eta_k}\theta_{\eta_k}\|_{L^p}\big\}_{k=1}^\infty$ be the non-increasing rearrangement
of the sequence $\big\{\|a_{\xi}\theta_\xi\|_{L^p}\big\}_{\xi\in \cX}$,
i.e.
\begin{equation*}
\|a_{\eta_1}\theta_{\eta_1}\|_{L^p}\ge \|a_{\eta_2}\theta_{\eta_2}\|_{L^p}\ge\cdots.
\end{equation*}
\end{sloppypar}

Consider first the case $p\le q$.
Fix $n\ge 1$ and set $a^\star_{\eta_k}:= a_{\eta_k}$ if $k>n$ and $a^\star_{\eta_k}:=0$ if $k\le n$.
Note that from \eqref{def-b-space} and \eqref{local-norm2} it follows that
\begin{equation*}
\|a^\star\|_{\bb_p^{0p}}\sim \Big(\sum_{\xi\in\cX} \|a^\star_{\xi}\theta_{\xi}\|_{L^p}^p\Big)^{1/p}
=\Big(\sum_{k=n+1}^\infty \|a_{\eta_k}\theta_{\eta_k}\|_{L^p}^p\Big)^{1/p}.
\end{equation*}
Using this, embedding \eqref{eq:2}, and the boundedness of the synthesis operator $T_\theta$ from \thmref{thm:prop-frame}~(b)
we get
\begin{align*}
\sigma_n(f)_{\cB_p^{0q}(\SS)}
&\le \Big\|\sum_{\eta\in\cX}a_\eta\theta_\eta-\sum_{k=1}^n a_{\eta_k}\theta_{\eta_k}\Big\|_{\cB_p^{0q}}
\le c\Big\|\sum_{k=n+1}^\infty a_{\eta_k}\theta_{\eta_k}\Big\|_{\cB_p^{0p}}
\\
& \le c\|a^\star\|_{\bb_p^{0p}}
\le c\Big(\sum_{k=n+1}^\infty \|a_{\eta_k}\theta_{\eta_k}\|_{L^p}^p\Big)^{1/p}.
\end{align*}
Further, we apply the inequality of \lemref{lem:est-B} with $x_k=\|a_{\eta_k}\theta_{\eta_k}\|_{L^p}$,
\eqref{local-norm2} and Theorem~\ref{thm:prop-frame} (c) to obtain
\begin{align*}
\sigma_n(f)_{\cB_p^{0q}(\SS)}
&\le c n^{1/p-1/\tau}\Big(\sum_{k=1}^\infty \|a_{\eta_k}\theta_{\eta_k}\|_{L^p}^\tau\Big)^{1/\tau}
\\
&\le cn^{-s/(d-1)}\|a\|_{\bb_\tau^{s\tau}}\le cn^{-s/(d-1)}\|f\|_{\cB_\tau^{s\tau}(\SS)},
\end{align*}
which confirms \eqref{jackson-S-B}.

In the case $q=\tau$, we use the embedding $\cB_\tau^{s\tau}(\SS)\subset\cB_p^{0\tau}(\SS)$ (see \eqref{eq:11}) to obtain
$\|f-g\|_{\cB_p^{0\tau}} = o(1)\|f\|_{\cB_\tau^{s\tau}}$ with $g=\sum_{k=1}^n a_{\eta_k}\theta_{\eta_k}$.
Combining this estimate with estimate \eqref{jackson-S-B} with $q=p$ and applying H\"older's inequality we obtain
in the case $\tau\le q<p$
\begin{multline*}
\big\|f-g\big\|_{\cB_p^{0q}}
\le\big\|f-g\big\|_{\cB_p^{0\tau}}^{(p-q)\tau/((p-\tau)q)}
\big\|f-g\big\|_{\cB_p^{0p}}^{(q-\tau)p/((p-\tau)q)}\\
=o(1)\big\|f\big\|_{\cB_\tau^{s\tau}}^{(p-q)\tau/((p-\tau)q)}
n^{1/q-1/\tau}\big\|f\big\|_{\cB_\tau^{s\tau}}^{(q-\tau)p/((p-\tau)q)}
= o(n^{1/q-1/\tau})\|f\|_{\cB_\tau^{s\tau}}.
\end{multline*}
This proves \eqref{jackson-S-B2} and completes the proof of the theorem.
\end{proof}

\begin{rem}
In comparing \thmref{thm:Jackson-S-F} and \thmref{thm:Jackson-S-B}
we see that the optimal rate $n^{s/(d-1)}$ holds for approximation in $\cF_p^{0q}(\SS)$ for every $q>0$
but only for $q\ge p$
for approximation in $\cB_p^{0q}(\SS)$.
\thmref{thm:Jackson-S-B} cannot be extended for $q<\tau$
because $\cB_\tau^{s\tau}(\SS)\setminus\cB_p^{0q}(\SS)\ne\emptyset$ if $q<\tau$.
\end{rem}

\begin{sloppypar}
\begin{rem}
Note that in both \thmref{thm:Jackson-S-F} and \thmref{thm:Jackson-S-B}
we form the near best approximant as the sum of the $n$ terms $\langle f, \tilde\theta_\xi\rangle \theta_\xi$
with the biggest norms $\|\langle f, \tilde\theta_\xi\rangle \theta_\xi\|_{L^p(\SS)}$.
\end{rem}
\end{sloppypar}

\subsection{Nonlinear $n$-term approximation of harmonic functions on $B^d$}\label{ss:TL-B}

We next use Theorems~\ref{thm:Jackson-S-F} and \ref{thm:Jackson-S-B} to establish
respective Jackson estimates for nonlinear $n$-term approximation of harmonic functions
on $B^d$ from shifts of Newtonian kernel.

\begin{thm}\label{thm:Jackson-B-F}
Let $s>0$, $0<p,q<\infty$, and $1/\tau=s/(d-1)+1/p$.
If the harmonic function $U\in B_\tau^{s\tau}(\HHH)=F_\tau^{s\tau}(\HHH)$, then $U\in F_p^{0q}(\HHH)$ and
\begin{equation}\label{jackson-B-F}
E_n(U)_{F_p^{0q}(\HHH)} \le cn^{-s/(d-1)}\|U\|_{B_\tau^{s\tau}(\HHH)},\quad n\ge 1,
\end{equation}
where the constant $c>0$ depends only on $s, p, q, d$.
\end{thm}

\begin{proof}
\begin{sloppypar}
By \thmref{thm:equiv-norms-B} it follows that
the boundary value function $f_U$ of $U$ given in Proposition~\ref{prop:identify}
belongs to $\cB_\tau^{s\tau}(\SS)=\cF_\tau^{s\tau}(\SS)$
and
$\|f_U\|_{\cB_\tau^{s\tau}(\SS)} \sim \|U\|_{B_\tau^{s\tau}(\HHH)}$.
Then embedding \eqref{eq:12} of \propref{embed} implies that $f_U\in\cF_p^{0q}(\SS)$
and in turn Theorem~\ref{thm:equiv-norms-F} yields $U\in F_p^{0q}(\HHH)$.
\end{sloppypar}

With $s, p, q, \tau$ already fixed, we choose
$A:=\max\big\{s, p, q^{-1}, \tau^{-1}, 2\big\}$.
Then conditions \eqref{main-cond} are satisfied.
Pick the parameters $K$, $M$, and $\gamma_0$ as in \eqref{param-KM}--\eqref{param-gamma0}.
Then the frames $\{\theta_\xi\}_{\xi\in\cX}$, $\{\tilde\theta_\xi\}_{\xi\in\cX}$ are well defined.
Appealing to \thmref{thm:Jackson-S-F} we conclude that for any $n\ge 1$ there exist
$\xi_1, \dots, \xi_n\in\cX$ and coefficients $a_1, \dots, a_n\in\CC$ such that
$$
\Big\|f_U - \sum_{j=1}^n a_j \theta_{\xi_j}\Big\|_{\cF_p^{02}(\SS)}
\le cn^{-s/(d-1)}\|f_U\|_{\cB_\tau^{s\tau}(\SS)}
\le cn^{-s/(d-1)}\|U\|_{B_\tau^{s\tau}(\HHH)}.
$$
Write
$G_n(x):= \sum_{j=1}^n a_j \theta_{\xi_j}(x)$, $x\in B^d$.
From above by harmonic extension using \thmref{thm:equiv-norms-F} we obtain
\begin{equation}\label{est-U-G}
\|U - G_n\|_{F_p^{0q}(\HHH)}
\le cn^{-s/(d-1)}\|U\|_{B_\tau^{s\tau}(\HHH)}.
\end{equation}
However, by Theorem~\ref{thm:new_frame} (b) we know that
for every $\xi\in\cX\setminus \cX_0$ the frame element $\theta_\xi$ is
a linear combination of $\le \tilde{n}$ shifts of the Newtonian kernel,
where $\tilde n$ is a~constant.
Therefore, $G_n\in\NNN_{\tilde{n}n}$ and then estimate \eqref{jackson-B-F} follows readily by \eqref{est-U-G}.
\end{proof}

\begin{thm}\label{thm:Jackson-B-B}
Let $s>0$, $0<p, q<\infty$, $1/\tau=s/(d-1)+1/p$, and $q\ge \tau$.
If the harmonic function $U\in B_\tau^{s\tau}(\HHH)=F_\tau^{s\tau}(\HHH)$, then $U\in B_p^{0q}(\HHH)$
and for every $n\ge 1$ we have
\begin{equation}\label{jackson-B-B}
E_n(U)_{B_p^{0q}(\HHH)} \le cn^{-s/(d-1)}\|U\|_{B_\tau^{s\tau}(\HHH)},\quad p\le q,
\end{equation}
\begin{equation}\label{jackson-B-B2}
E_n(U)_{B_p^{0q}(\HHH)} = o(n^{1/q-1/\tau})\|U\|_{B_\tau^{s\tau}(\HHH)},\quad \tau\le q<p,
\end{equation}
where the constant $c>0$ depends only on $A, d$.
\end{thm}

The proof of \thmref{thm:Jackson-B-B} goes along the lines of the proof of \thmref{thm:Jackson-B-F}
with \thmref{thm:Jackson-S-F} replaced by \thmref{thm:Jackson-S-B}.
We omit it.

\section{Approximation of harmonic functions on $\RR^d\setminus \overline{B^d}$ and $\RR_+^d$}\label{s8}

The results in Section~\ref{s7} have their analogues for approximation of harmonic functions
on $\RR^d\setminus \overline{B^d}$ or $\RR_+^d$.
In the following we established the analogue of the main result (Theorem~\ref{thm:Jackson}) on $\RR^d\setminus \overline{B^d}$
and explain briefly its analogue on $\RR_+^d$.

In analogy to the set $\NNN_n$ from \eqref{def:N-n-d}--\eqref{def:N-n-2}
we denote by $\overline{\NNN}_n$ the set of all linear combinations of shifts of the Newtonian kernel
as in \eqref{def:N-n-d}--\eqref{def:N-n-2}
with the requirement that the poles $y_\nu\in B^d$.

The approximation will take place in the harmonic Hardy space
$\HHH^p(\RR^d\setminus \overline{B^d})$.
Let $\HHH(\RR^d\setminus\overline{B^d})$ denote the set of all harmonic functions $U$ on $\RR^d\setminus\overline{B^d}$
such that
$\lim_{|x|\to\infty} U(x) =0$ if $d>2$
or
$\lim_{|x|\to\infty} U(x) = {\rm const.}$ if $d=2$.
The harmonic Hardy space
$\HHH^p(\RR^d\setminus \overline{B^d})$, $0<p<\infty$,
is defined as the set of all harmonic functions $U\in \HHH(\RR^d\setminus\overline{B^d})$ such that
\begin{equation}\label{def-Hp-c}
\|U\|_{\HHH^p(\RR^d\setminus \overline{B^d})} := \|\sup_{r>1}r^{d-2}|U(r\cdot)|\|_{L^p(\SS)} <\infty.
\end{equation}
Given $U\in \HHH^p(\RR^d\setminus \overline{B^d})$ we define
\begin{equation}\label{def:best-app-2}
E_n(U)_{\HHH^p(\RR^d\setminus \overline{B^d})}:=\inf_{G\in \overline{\NNN}_n}\|U-G\|_{\HHH^p(\RR^d\setminus \overline{B^d})}.
\end{equation}

Denote by $\bar B_p^{sq}(\HHH)$ the harmonic Besov spaces
on $\RR^d\setminus \overline{B^d}$ (see \cite[Section 8]{IP}).

As one can expect the following Jackson type theorem is valid:

\begin{thm}
\label{thm:Jackson-H2}
Let $s>0$, $0<p<\infty$, and $1/\tau=s/(d-1)+1/p$.
If $U\in \bar B_\tau^{s\tau}(\HHH)$, then $U\in \HHH^p(\RR^d\setminus \overline{B^d})$ and
\begin{equation}\label{jackson-H-2}
E_n(U)_{\HHH^p(\RR^d\setminus \overline{B^d})} \le cn^{-s/(d-1)}\|U\|_{\bar B_\tau^{s\tau}(\HHH)},\quad n\ge 1,
\end{equation}
where the constant $c>0$ depends only on $p, s, d$.
\end{thm}

\begin{proof}
As is well known the Kelvin transform
$
KU(x):= |x|^{2-d}U(x/|x|^2)
$
maps one-to-one $\HHH(B^d)$ onto $\HHH(\RR^d\setminus \overline{B^d})$
and $K^{-1}=K$.
It is easy to see that the Kelvin transform is an isometric isomorphism of
$\HHH^p(\RR^d\setminus \overline{B^d})$ onto $\HHH^p(B^d)$.
Also, as shown in \cite[Section 8]{IP} the Kelvin transform is an isometric isomorphism
between the harmonic Besov spaces $\bar B^{sq}_p(\HHH)$ on $\RR^d\setminus \overline{B^d}$ and
the harmonic Besov spaces $B^{sq}_p(\HHH)$ on $B^d$.
Furthermore, it is readily seen by the symmetry lemma that for a fixed $y\in \RR^d$, $y\ne 0$,
$$
K\Big(\frac{1}{|x-y|^{d-2}}\Big)(x) = \frac{|y|^{2-d}}{|x-y/|y|^2|^{d-2}}, \quad d>2,
$$
and
$$
K\Big(\ln \frac{1}{|x-y|}\Big)(x) = \ln \frac{1}{|y|}+ \ln |x| + \ln\frac{1}{|x-y/|y|^2|}, \quad d=2.
$$
Assuming that $U\in \bar B_\tau^{s\tau}(\HHH)$ we apply estimate \eqref{jackson-H} to $KU$
and use all of the above to conclude that estimate \eqref{jackson-H-2} holds true.
\end{proof}

\smallskip

\noindent
{\bf Approximation of harmonic functions on $\RR_+^d$.}
Closely related to the approximation problem considered above
is the problem for nonlinear $n$-term approximation of functions in the harmonic Hardy spaces
$\HHH^p(\RR^d_{+})$, $0<p<\infty$, from linear combinations of shifts of the Newtonian kernel
with poles in $\RR^d_{-}$.
This problem should be regarded as a limiting case of the same problem on $B(0, R)\subset \RR^d$
as $R\to \infty$.
For lack of space we do not elaborate on this sort of approximation.
We would like to observe only that all definitions and statements in this article have analogues
in the more common setting on $\RR_+^d$ from Harmonic analysis point of view,
in particular, our main Jackson estimate \eqref{jackson-H} is valid.

\section{Proofs}\label{appendix}

\subsection{Proofs of Propositions~\ref{prop:5_1}, \ref{prop:5_2}, and \ref{prop:5_3}}

For the proofs of \propref{prop:5_1}
we need the following simple

\begin{lem}\label{lem:5_1}
Let $K\in\NN$, $x_0\in\SS$, $g\in W_\infty^K(\SS)$ and
$\til{g}(x):=g(x/|x|)$ for $x\in\RR^d\backslash\{0\}$.
Then for every $x\in\SS$ with $\rho(x,x_0)\le 1$ we have
\begin{equation*}
\Big|\til{g}(x)-\sum_{|\beta|\le K-1}\frac{\partial^\beta\til{g}(x_0)}{\beta!}(x-x_0)^\beta\Big|
\le c \rho(x,x_0)^K\sup_{\substack{z\in\SS\\ \rho(z,x_0)\le\rho(x,x_0)}}\max_{|\beta|=K}
\left|\partial^\beta\til{g}(z)\right|
\end{equation*}
with $c$ depending only on $d$ and $K$.
\end{lem}

\begin{proof}
Assuming $x\ne x_0$, we set $\eta:=(x-x_0)/|x-x_0|\in\SS$.
Then from Taylor's theorem
there exists  $\lambda\in(0,1)$ such that
\begin{multline*}
\Big|\til{g}(x)-\sum_{|\beta|\le K-1}\frac{\partial^\beta\til{g}(x_0)}{\beta!}(x-x_0)^\beta\Big|\\
=\frac{|x-x_0|^K}{K!}\left|(\eta\cdot\nabla)^K\til{g}(x_\lambda)\right|
=\frac{|x-x_0|^K}{|x_\lambda|^K K!}\Big|(\eta\cdot\nabla)^K\til{g}\Big(\frac{x_\lambda}{|x_\lambda|}\Big)\Big|,
\end{multline*}
where $x_\lambda:=x_0+\lambda(x-x_0)$ and the definition of $\til{g}$ is used for the last equality.
Now, we use that $|x-x_0|\le\rho(x,x_0)$, $|x_\lambda|\ge\cos 1/2$ for $\lambda\in(0,1)$, and
\begin{equation*}
\left|(\eta\cdot\nabla)^K\til{g}(y)\right|
\le c \max_{|\beta|=K}\left|\partial^\beta\til{g}(y)\right|,\quad y\in\SS,
\end{equation*}
to complete the proof.
\end{proof}

\begin{proof}[Proof of \propref{prop:5_1}]
We represent $\left\langle g,f\right\rangle$ in the form
\begin{equation}\label{eq:5a}
\begin{split}
\left\langle g,f\right\rangle&=S_1+S_2,\\
S_1&:=\int_{\SS}\Big(\til{g}(y)
-\sum_{|\beta|\le K-1}\frac{\partial^\beta\til{g}(x_2)}{\beta!}(y-x_2)^\beta\Big)\overline{f(y)}\, d\sigma(y),\\
S_2&:=\sum_{|\beta|\le K-1}\frac{\partial^\beta\til{g}(x_2)}{\beta!}\int_{\SS}(y-x_2)^\beta \overline{f(y)}\, d\sigma(y).
\end{split}
\end{equation}
From \eqref{eq:3a} we get
\begin{equation*}
\Big|\int_{\SS}(y-x_2)^\beta \overline{f(y)}\, d\sigma(y)\Big|\le c\kappa_2 \scal_2^{-K},\quad ~0\le |\beta|\le K-1
\end{equation*}
and using \eqref{eq:1a}
\begin{equation*}
|S_2|
\le \sum_{|\beta|\le K-1} \frac{\kappa_1 \scal_1^{|\beta|+d-1}}{\beta!(1+\scal_1\rho(x_1,x_2))^M} c\kappa_2 \scal_2^{-K}
\le c\frac{\kappa_1 \kappa_2 (\scal_1/\scal_2)^{K}\scal_1^{d-1}}{(1+\scal_1\rho(x_1,x_2))^M}.
\end{equation*}
We bound $S_1$ by
\begin{equation*}
|S_1|\le\int_{\SS}\Big|\til{g}(y)
-\sum_{|\beta|\le K-1}\!\!\frac{\partial^\beta\til{g}(x_2)}{\beta!}(y-x_2)^\beta\Big||f(y)|\, d\sigma(y)
=:\int_{\nA_1}+\int_{\nA_2}+\int_{\nA_3},
\end{equation*}
where
\begin{align*}
\nA_1&=\{y\in\SS : \rho(x_2, y)\le N_1^{-1}\},\\
\nA_2&=\{y\in\SS : \rho(x_2, y)>N_1^{-1}, \rho(x_1,y)\le\rho(x_1,x_2)/2\},\\
\nA_3&=\{y\in\SS : \rho(x_2, y)>N_1^{-1}, \rho(x_1,y)>\rho(x_1,x_2)/2\}.
\end{align*}

For $y\in\nA_1$, \lemref{lem:5_1} and \eqref{eq:1a} imply
\begin{multline*}
\Big|\til{g}(y)-\sum_{|\beta|\le K-1}\frac{\partial^\beta\til{g}(x_2)}{\beta!}(y-x_2)^\beta\Big|
\le c \rho(y,x_2)^K\sup_{z\in\nA_1}\max_{|\beta|= K}\left|\partial^\beta\til{g}(z)\right|\\
\le c \sup_{z\in\nA_1}\frac{\kappa_1 \scal_1^{K+d-1}}{(1+\scal_1\rho(x_1,z))^M}\rho(y,x_2)^K
\le c \frac{\kappa_1 \scal_1^{K+d-1}}{(1+\scal_1\rho(x_1,x_2))^M}\rho(y,x_2)^K
\end{multline*}
due to \eqref{comp_local}.
Using the above estimate, \eqref{eq:2a}, and \eqref{eq:conv_1} we see that
\begin{align*}
\int_{\nA_1}
&\le c \frac{\kappa_1 \scal_1^{K+d-1}}{(1+\scal_1\rho(x_1,x_2))^M}
\int_{\nA_1}\rho(y,x_2)^K \frac{\kappa_2 \scal_2^{d-1}}{(1+\scal_2\rho(y,x_2))^M}\, d\sigma(y)\\
&\le c \frac{\kappa_1 \kappa_2 \scal_1^{K+d-1}}{(1+\scal_1\rho(x_1,x_2))^M} \scal_2^{-K}
\int_\SS \frac{\scal_2^{d-1}}{(1+\scal_2\rho(y,x_2))^{M-K}}\, d\sigma(y)\\
&\le c\frac{\kappa_1 \kappa_2 (\scal_1/\scal_2)^{K}\scal_1^{d-1}}{(1+\scal_1\rho(x_1,x_2))^M}.
\end{align*}

\begin{sloppypar}
For $y\in\nA_2$ we have $\rho(x_1,x_2)/2\le\rho(x_2, y)\le 3\rho(x_1,x_2)/2$,
and hence
$3\scal_2\rho(y,x_2)\ge (\scal_2/\scal_1)(1+\scal_1\rho(x_1,x_2))$, using that $ \rho(y,x_2)> \scal_1^{-1}$.
Therefore,
\begin{equation*}
(1+\scal_2\rho(y,x_2))^{-M}\le (\scal_2\rho(y,x_2))^{-M}\le 3^M(\scal_1/\scal_2)^M(1+\scal_1\rho(x_1,x_2))^{-M}.
\end{equation*}
This combined with \eqref{eq:1a} and \eqref{eq:2a} implies
\begin{align*}
\int_{\nA_2}
&\le \int_{\nA_2}\frac{\kappa_1 \scal_1^{d-1}}{(1+\scal_1\rho(x_1,y))^M}
\frac{\kappa_2 \scal_2^{d-1}}{(1+\scal_2\rho(y,x_2))^M}\, d\sigma(y)\\
&\qquad\qquad\quad+ \int_{\nA_2}\sum_{|\beta|\le K-1}\frac{\kappa_1 \scal_1^{|\beta|+d-1}\rho(y,x_2)^{|\beta|}}
 {\beta!(1+\scal_1\rho(x_1,x_2))^M}\frac{\kappa_2 \scal_2^{d-1}}{(1+\scal_2\rho(y,x_2))^M}\, d\sigma(y)\\
&\le \int_{\nA_2}\frac{\kappa_1 \scal_1^{d-1}}{(1+\scal_1\rho(x_1,y))^M}\, d\sigma(y)
\frac{3^M\kappa_2 (\scal_1/\scal_2)^M\scal_2^{d-1}}{(1+\scal_1\rho(x_1,x_2))^M}\\
&+\sum_{|\beta|\le K-1}\int_{\nA_2}
\frac{\kappa_1 \scal_1^{|\beta|+d-1}(3/2)^{|\beta|}\rho(x_1,x_2)^{|\beta|}}
 {\beta!(1+\scal_1\rho(x_1,x_2))^M}\, d\sigma(y)
\frac{3^M\kappa_2 (\scal_1/\scal_2)^M\scal_2^{d-1}}{(1+\scal_1\rho(x_1,x_2))^M}\\
&\le c\frac{\kappa_1 \kappa_2 (\scal_1/\scal_2)^{K}\scal_1^{d-1}}{(1+\scal_1\rho(x_1,x_2))^M}.
\end{align*}
Here we used that $M>K+d-1$ and \eqref{eq:conv_1}
as well as the fact that $\scal_1/\scal_2\le 1$ and $\sigma(\nA_2)\le c\rho(x_1,x_2)^{d-1}$.
\end{sloppypar}

For $y\in\nA_3$, we have $(1+\scal_1\rho(x_1,x_2))/2\le 1+\scal_1\rho(x_1,y)$ and $\rho(x_2, y)>\scal_1^{-1}$.
Therefore,
\begin{align*}
\int_{\nA_3}
&\le \int_{\nA_3}\frac{\kappa_1 \scal_1^{d-1}}{(1+\scal_1\rho(x_1,y))^M}
\frac{\kappa_2 \scal_2^{d-1}}{(1+\scal_2\rho(y,x_2))^M}\, d\sigma(y)\\
&\qquad\qquad\quad + \int_{\nA_3}\sum_{|\beta|\le K-1}\frac{\kappa_1 \scal_1^{|\beta|+d-1}\rho(y,x_2)^{|\beta|}}
 {\beta!(1+\scal_1\rho(x_1,x_2))^M}\frac{\kappa_2 \scal_2^{d-1}}{(1+\scal_2\rho(y,x_2))^M}\, d\sigma(y)\\
&\le c \frac{\kappa_1 \kappa_2 \scal_1^{d-1}}{(1+\scal_1\rho(x_1,x_2))^M}
\int_{\nA_3}\frac{(\scal_1\rho(y,x_2))^{K}\scal_2^{d-1}}{(1+\scal_2\rho(y,x_2))^M}\, d\sigma(y)\\
&\le c\frac{\kappa_1 \kappa_2 (\scal_1/\scal_2)^{K}\scal_1^{d-1}}{(1+\scal_1\rho(x_1,x_2))^M},
\end{align*}
using that $M>K+d-1$ and
\begin{equation*}
\int_{\SS}\frac{(\scal_2\rho(y,x_2))^{K}\scal_2^{d-1}}{(1+\scal_2\rho(y,x_2))^M}\, d\sigma(y)
\le\int_{\SS}\frac{\scal_2^{d-1}}{(1+\scal_2\rho(y,x_2))^{M-K}}\, d\sigma(y)\le c
\end{equation*}
because of \eqref{eq:conv_1}. This completes the proof.
\end{proof}

\begin{proof}[Proof of \propref{prop:5_2}]
This proof is the same as the proof of \propref{prop:5_1} for $K=1$.
Instead of estimating $S_2$ in \eqref{eq:5a} we move it to the left-hand side.
Only the localization of the first derivatives of $\til{g}$, but not of $\til{g}$ itself, is needed here.
\end{proof}

\begin{proof}[Proof of \propref{prop:5_3}]
This proof follows along the lines of the proof of \propref{prop:5_1} with $K=0$.
Of course, in this case the Taylor series is missing from the definitions of both $S_1$ and $S_2$ in \eqref{eq:5a},
i.e. $S_2\equiv 0$. \lemref{lem:5_1} is also not used in the proof.
\end{proof}

\subsection{Proof of Theorem~\ref{thm:almost-diag}}

This proof depends on the next three lemmas.

\begin{lem}\label{lem:Besov_1}
Let $j,m\ge 0$, $0<\beta\le 1$, $x\in\SS$ and $\xi\in\cX_j$. Then
\begin{equation}\label{eq:besov_0}
\sum_{\eta\in\cX_{j+m}} \frac{1}{(1+N_\xi\rho(x,\eta))^{d-1+\beta}}\le c^\star_1 2^{m(d-1)}
\end{equation}
with $c^\star_1=c(d)\beta^{-1}$.
\end{lem}
\begin{proof}
Using that $\cX_{j+m}$ is a maximal $\gamma 2^{-j-m+1}$ net with a fixed $\gamma=c(d)\in(0,1)$ (as stated in \S\ref{s5_3}),
\eqref{eq:maximal_net}, the inequality $(1+\gamma)(1+N_\xi\rho(x,\eta))\ge 1+N_\xi\rho(x,y)$ for any $y\in A_\eta$, and (2.6) we obtain
\begin{multline*}
\sum_{\eta\in\cX_{j+m}} \frac{1}{(1+N_\xi\rho(x,\eta))^{d-1+\beta}}
\le c(d)N_\eta^{d-1}\sum_{\eta\in\cX_{j+m}} \frac{|A_\eta|}{(1+N_\xi\rho(x,\eta))^{d-1+\beta}}\\
\le c(d)N_\eta^{d-1}\int_{\SS} \frac{d\sigma(y)}{(1+N_\xi\rho(x,y))^{d-1+\beta}}
\le c(d)N_\eta^{d-1}c(d)\beta^{-1} N_\xi^{-d+1},
\end{multline*}
which proves \eqref{eq:besov_0}.
\end{proof}

\begin{lem}\label{lem:mrax}
Let  $0<t\le 1$ and $M\ge(d-1)/t+\delta$, $0<\delta\le 1$. Then for any sequence of complex numbers
$\{h_{\eta}\}_{\eta\in \cX_{m} }$,   $m\ge 0$,
and for any $x\in B_\xi=B(\xi,\gamma 2^{-j})$, $\xi\in\cX_j$, $j\ge 0$, we have
\begin{multline}\label{est-M}
\sum_{\eta\in\cX_m} \frac{|h_{\eta}|}{\big(1+\min\{N_\xi, N_\eta\}\rho(\xi,\eta)\big)^{M}}\\
\le c^\star_2 \max\big\{1,2^{(m-j)(d-1)/t}\big\} \MM_t\Big(\sum_{\eta\in \cX_m}
|h_{\eta}|\ONE_{\nB_\eta}\Big)(x),
\end{multline}
where $c^\star_2:=(2/\ln 2) 4^{(d-1)/t}(2/\gamma)^M \delta^{-1}$
with $\gamma\in(0,1)$ being the constant from the construction of the old frame $\Psi$ in \S\ref{subsec:frame-SS}.
\end{lem}

\begin{proof}
Two cases present themselves here.

{\bf Case 1:} $m\ge j$.
Set
$\cQ_{0}:=\{\eta\in \cX_m: 2^{j-1}\rho(\xi,\eta)< \gamma \}$
and
$$
\cQ_{\nu}:=\{\eta\in \cX_m: \gamma 2^{\nu-1}\le 2^{j-1}\rho(\xi,\eta)< \gamma 2^{\nu}\},\quad \nu\ge 1.
$$
Since $0<t\le 1$ we have for  $\nu\ge 1$
\begin{align*}
\sum_{\eta\in\cQ_\nu} \frac{|h_{\eta}|}{\bigl(1+2^{j-1}\rho(\xi,\eta)\bigr)^{M}}
\le \Big(\frac{2}{\gamma}\Big)^M 2^{-\nu M} \sum_{\eta\in\cQ_\nu} |h_{\eta}|
\le \Big(\frac{2}{\gamma}\Big)^M 2^{-\nu M} \Big(\sum_{\eta\in\cQ_\nu} |h_{\eta}|^t\Big)^{1/t}.
\end{align*}
The same estimate holds trivially for $\nu=0$.
Put
$$
\cR_\nu :=B(\xi,2\gamma (2^{-m}+2^{-j+\nu})),
\quad \nu\ge 0.
$$
Clearly $\cup_{\eta\in \cQ_\nu} B_\eta \subset \cR_\nu$.
Using this, the fact that the sets $\{B_\eta : \eta\in\cX_m\}$ are disjoint, and \eqref{sph_cap2}
 we obtain for every $x\in B_\xi \subset \cR_\nu$
\begin{align*}
\sum_{\eta\in \cQ_\nu}|h_{\eta}|^t
&=\int_\SS \biggl(\sum_{\eta\in \cQ_\nu}|h_\eta| |\nB_\eta|^{-1/t} \ONE_{\nB_\eta}(y)\biggr)^t \, d\sigma(y)\\
& =\frac{|\cR_\nu|}{|B(\xi, \gamma 2^{-m})|}\frac{1}{|\cR_\nu|}
\int_{\cR_\nu} \Big(\sum_{\eta\in \cQ_\nu}|h_\eta|\ONE_{\nB_\eta}(y)\Big)^t d\sigma(y)
\\
&\le 4^{d-1} 2^{(m-j+\nu)(d-1)} \Big[\MM_t\bigl(\sum_{\eta\in \cX_m}|h_\eta| \ONE_{\nB_\eta}\bigr)(x)\Big]^t.
\end{align*}
Therefore, since $M\ge (d-1)/t+\delta$ we get for any $x\in B_\xi$
\begin{align*}
\sum_{\eta\in\cX_m} &\frac{|h_{\eta}|}{\big(1+2^{j-1}\rho(\xi,\eta)\big)^{M}}\\
&\le \Big(\frac{2}{\gamma}\Big)^M 4^{(d-1)/t}\MM_t\bigl(\sum_{\eta\in \cX_m}|h_\eta| \ONE_{\nB_\eta}\bigr)(x)
\sum_{\nu\ge 0} 2^{-\nu M}2^{(\nu-j+m)(d-1)/t}
\\
&\le \Big(\frac{2}{\gamma}\Big)^M 4^{(d-1)/t}2^{(m-j)(d-1)/t}
\MM_t\bigl(\sum_{\eta\in \cX_m}|h_\eta| \ONE_{\nB_\eta}\bigr)(x)
\sum_{\nu\ge 0} 2^{-\nu\delta}
\\
&\le \frac{2}{\delta\ln 2}\Big(\frac{2}{\gamma}\Big)^M 4^{(d-1)/t}2^{(m-j)(d-1)/t}
\MM_t\bigl(\sum_{\eta\in \cX_m}|h_\eta| \ONE_{\nB_\eta}\bigr)(x),
\end{align*}
which confirms (\ref{est-M}).

\smallskip

{\bf Case 2:} $m<j$.
Set
$\tQ_{0}:=\{\eta\in \cX_m: 2^{m-1} \rho(\xi,\eta)< \gamma\}$
and
$$
\tQ_{\nu}:=\{\eta\in \cX_m: \gamma 2^{\nu-1}\le 2^{m-1}\rho(\xi,\eta)< \gamma 2^{\nu}\},\quad \nu\ge 1.
$$
Write
$$
\tilde \cR_\nu :=B(\xi,\gamma 2^{-m+1}(1+2^{\nu})),
\quad \nu\ge 0.
$$
We use that $0<t\le 1$ to obtain
\begin{align*}
\sum_{\eta\in\tQ_\nu} \frac{|h_{\eta}|}{\bigl(1+2^{m-1}\rho(\xi,\eta)\bigr)^{M}}
\le \Big(\frac{2}{\gamma}\Big)^M 2^{-\nu M} \sum_{\eta\in\tQ_\nu} |h_{\eta}|
\le \Big(\frac{2}{\gamma}\Big)^M 2^{-\nu M} \Big(\sum_{\eta\in\tQ_\nu} |h_{\eta}|^t\Big)^{1/t}.
\end{align*}
Just as in Case 1 we obtain for $x\in B_\xi \subset \tilde \cR_\nu$
\begin{align*}
\sum_{\eta\in \tQ_\nu}|h_{\eta}|^t
&=\int_\SS \biggl(\sum_{\eta\in \tQ_\nu}|h_\eta| |\nB_\eta|^{-1/t} \ONE_{\nB_\eta}(y)\biggr)^t \, d\sigma(y)\\
& =\frac{|\tilde \cR_\nu|}{|B(\xi, \gamma 2^{-m})|}\frac{1}{|\tilde \cR_\nu|}
\int_{\tilde \cR_\nu} \Big(\sum_{\eta\in \tQ_\nu}|h_\eta|\ONE_{\nB_\eta}(y)\Big)^t d\sigma(y)
\\
&\le 4^{d-1} 2^{\nu(d-1)} \Big[\MM_t\bigl(\sum_{\eta\in \cX_m}|h_\eta| \ONE_{\nB_\eta}\bigr)(x)\Big]^t.
\end{align*}
As before, since $M\ge (d-1)/t+\delta$ we get for any $x\in B_\xi$
\begin{align*}
\sum_{\eta\in\cX_m} &\frac{|h_{\eta}|}{\big(1+2^{m-1}\rho(\xi,\eta)\big)^{M}}\\
&\le \Big(\frac{2}{\gamma}\Big)^M 4^{(d-1)/t}\MM_t\bigl(\sum_{\eta\in \cX_m}|h_\eta| \ONE_{\nB_\eta}\bigr)(x)
\sum_{\nu\ge 0} 2^{-\nu(M-(d-1)/t)}
\\
&\le \frac{2}{\delta\ln 2}\Big(\frac{2}{\gamma}\Big)^M 4^{(d-1)/t}
\MM_t\bigl(\sum_{\eta\in \cX_m}|h_\eta| \ONE_{\nB_\eta}\bigr)(x),
\end{align*}
which verifies (\ref{est-M}).
The proof of the lemma is complete.
\end{proof}

In the next lemma we specify the constants in certain well known discrete Hardy inequalities that will be needed.

\begin{lem}\label{lem:hardy}
Let $\gam >0$, $0<q<\infty$, and $a_m \ge 0$ for $m\ge 0$. Then
\begin{equation}\label{hardy-1}
\Big(\sum_{j\ge 0}\Big(\sum_{m\ge j} 2^{-(m-j)\gam}a_m\Big)^q\Big)^{1/q}
\le c^\star_3\Big(\sum_{m\ge 0} a_m^q\Big)^{1/q}
\end{equation}
and
\begin{equation}\label{hardy-2}
\Big(\sum_{j\ge 0}\Big(\sum_{m=0}^j 2^{-(j-m)\gam}a_m\Big)^q\Big)^{1/q}
\le c^\star_3\Big(\sum_{m\ge 0} a_m^q\Big)^{1/q}
\end{equation}
with
\begin{equation*}
c^\star_3=2^\gam\max\left\{\frac{1}{\gam \ln2},\frac{1}{(\gam q \ln 2)^{1/q}}\right\}.
\end{equation*}
\end{lem}

\begin{proof}
In the case $0<q\le 1$ inequalities (\ref{hardy-1})--(\ref{hardy-2}) follow readily by
applying the $q$-inequality and switching the order of summation.
More precisely, we have
\begin{align*}
&\Big(\sum_{j\ge 0}\Big(\sum_{m\ge j} 2^{-(m-j)\gam}a_m\Big)^q\Big)^{1/q}
\le \Big(\sum_{j\ge 0}\sum_{m\ge j} 2^{-(m-j)\gam q}a_m^q\Big)^{1/q}\\
&= \Big(\sum_{m\ge 0}\sum_{j= 0}^m 2^{-(m-j)\gam q}a_m^q\Big)^{1/q}
\le \Big(\sum_{\nu\ge 0} 2^{-\nu\gam q}\Big)^{1/q}\Big(\sum_{m\ge 0}a_m^q\Big)^{1/q}.
\end{align*}
Clearly
\begin{equation*}
\Big(\sum_{\nu\ge 0} 2^{-\nu\gam q}\Big)^{1/q} =\Big(\frac{1}{1-2^{-\gam q}}\Big)^{1/q}
\le \frac{2^\gam}{(\gam q \ln 2)^{1/q}}
\le c^\star_3,
\end{equation*}
which implies (\ref{hardy-1}).
The proof of (\ref{hardy-2}) in the case $0<q\le 1$ is similar and gives the same constant $c^\star_3$.

In the case $q>1$ using the H\"{o}lder's inequality ($1/q'+1/q=1$)
and switching the order of summation just as above we obtain
\begin{align*}
&\Big(\sum_{j\ge 0}\Big(\sum_{m\ge j} 2^{-(m-j)\gam}a_m\Big)^q\Big)^{1/q}
= \Big(\sum_{j\ge 0}\Big(\sum_{m\ge j} 2^{-(m-j)\gam/q'}2^{-(m-j)\gam/q}a_m\Big)^q\Big)^{1/q}\\
&\le \Big(\sum_{j\ge 0}
\Big(\sum_{m\ge j} 2^{-(m-j)\gam }\Big)^{q/q'}
\Big(\sum_{m\ge j} 2^{-(m-j)\gam }a_m^q\Big)
\Big)^{1/q}\\
&\le
\Big(\sum_{\nu\ge 0} 2^{-\nu\gam }\Big)^{1/q'}
\Big(\sum_{\nu\ge 0} 2^{-\nu\gam }\Big)^{1/q}
\Big(\sum_{m\ge 0} a_m^q \Big)^{1/q}
=\sum_{\nu\ge 0} 2^{-\nu\gam }\Big(\sum_{m\ge 0} a_m^q \Big)^{1/q},
\end{align*}
which gives \eqref{hardy-1} with $c^\star_3\ge 2^\beta/(\beta \ln2)$.
The proof of \eqref{hardy-2} in the case $q>1$ is similar and with the same constant.
\end{proof}

\begin{proof}[Proof of Theorem~\ref{thm:almost-diag}]

We shall use the abbreviated notation  $\omega_{\xi,\eta}:= \omega_{\xi,\eta}^{(K,M)}$
for $\xi,\eta\in\cX$ (see \eqref{eq:omega_xi_eta}).

We first establish the result for the {\em sequence Besov spaces} $\bb^{sq}_p$, that is,
\begin{equation}\label{B-AlmDiag}
\|\Omega h\|_{\bb^{sq}_p}\le C_9\|h\|_{\bb^{sq}_p}.
\end{equation}
Set $p_*:=\max\{p,1\}$. We start with the proof of the estimate
\begin{equation}\label{eq:besov_1}
\begin{split}
\Big(\sum_{\xi\in\cX_j}\Big|\sum_{\eta\in\cX}\omega_{\xi,\eta}&h_\eta\Big|^p\Big)^{1/p}\\
&\le C_{11}\sum_{m=0}^\infty 2^{-m(K+(d-1)(1/2-1/p'_*)-\delta/2)}\Big(\sum_{\eta\in\cX_{j+m}} |h_\eta|^p\Big)^{1/p}\\
&+C_{11}\sum_{m=1}^j 2^{-m(K+(d-1)(1/2-1/p)-\delta/2)}\Big(\sum_{\eta\in\cX_{j-m}} |h_\eta|^p\Big)^{1/p}
\end{split}
\end{equation}
for any $j\ge 0$. For $1<p<\infty$ using \eqref{eq:omega_xi_eta} and the convexity of $u^p$ we obtain
\begin{multline}\label{eq:besov_2}
\sum_{\xi\in\cX_j}\Big|\sum_{\eta\in\cX}\omega_{\xi,\eta}h_\eta\Big|^p
\le 2^{p-1}\sum_{\xi\in\cX_j}\Big[
\Big(\sum_{m=0}^\infty \sum_{\eta\in\cX_{j+m}}\frac{2^{-m(K+(d-1)/2)}|h_\eta|}{(1+N_\xi\rho(\xi,\eta))^M}\Big)^p\\
+\Big(\sum_{m=1}^{j} \sum_{\eta\in\cX_{j-m}}\frac{2^{-m(K+(d-1)/2)}|h_\eta|}{(1+N_\eta\rho(\xi,\eta))^M}\Big)^p\Big].
\end{multline}
Applying in the first double sum in the right-hand side of \eqref{eq:besov_2}
twice H\"older's inequality, first in the summation on $m$ and then on $\eta$, and \lemref{lem:Besov_1} with $\beta=\delta$ we get
with $M_1=(d-1+\delta)/p'$ and $M_2=M-M_1$
\begin{multline}\label{eq:besov_3}
\Big(\sum_{m=0}^\infty \sum_{\eta\in\cX_{j+m}}\frac{2^{-m(K+(d-1)/2)}|h_\eta|}{(1+N_\xi\rho(\xi,\eta))^M}\Big)^p\\
\le \Big(\sum_{m=0}^\infty 2^{-mp'\delta/2}\Big)^{p/p'}
\sum_{m=0}^\infty \Big(\sum_{\eta\in\cX_{j+m}}\frac{2^{-m(K+(d-1)/2-\delta/2)}|h_\eta|}{(1+N_\xi\rho(\xi,\eta))^M}\Big)^p\\
\le C_{12}^p\sum_{m=0}^\infty \Big(\sum_{\eta\in\cX_{j+m}} \frac{1}{(1+N_\xi\rho(\xi,\eta))^{M_1p'}}\Big)^{p/p'}
\sum_{\eta\in\cX_{j+m}}\frac{2^{-m(K+(d-1)/2-\delta/2)p}|h_\eta|^p}{(1+N_\xi\rho(\xi,\eta))^{M_2p}}\\
\le {c^\star_1}^{p-1} C_{12}^p\sum_{m=0}^\infty
\sum_{\eta\in\cX_{j+m}}\frac{2^{-m(K+(d-1)(1/2-1/p')-\delta/2)p}|h_\eta|^p}{(1+N_\xi\rho(\xi,\eta))^{M_2p}},
\end{multline}
where $C_{12}:=\big(1- 2^{-p'\delta/2}\big)^{-1/p'}\le 5\delta^{-1/p'}$.
Applying the same arguments to the second double sum in the right-hand side of \eqref{eq:besov_2} we obtain
\begin{multline}\label{eq:besov_4}
\Big(\sum_{m=1}^j \sum_{\eta\in\cX_{j-m}}\frac{2^{-m(K+(d-1)/2)}|h_\eta|}{(1+N_\eta\rho(\xi,\eta))^M}\Big)^p\\
\le {c^\star_1}^{p-1} C_{12}^p\sum_{m=1}^j
\sum_{\eta\in\cX_{j-m}}\frac{2^{-m(K+(d-1)/2-\delta/2)p}|h_\eta|^p}{(1+N_\eta\rho(\xi,\eta))^{M_2p}},
\end{multline}
Note that \eqref{cond-KM} implies $M_2p\ge d-1+\delta$.
Substituting \eqref{eq:besov_3} and \eqref{eq:besov_4} in \eqref{eq:besov_2}
and using \lemref{lem:Besov_1} with $\beta=\delta$ we get
\begin{multline}\label{eq:besov_5}
\sum_{\xi\in\cX_j}\Big|\sum_{\eta\in\cX}\omega_{\xi,\eta}h_\eta\Big|^p\\
\le 2^{p-1} {c^\star_1}^{p-1} C_{12}^p\Big(
\sum_{m=0}^\infty \sum_{\eta\in\cX_{j+m}}\sum_{\xi\in\cX_j}
\frac{2^{-m(K+(d-1)(1/2-1/p')-\delta/2)p}|h_\eta|^p}{(1+N_\xi\rho(\xi,\eta))^{d-1+\delta}}\\
+\sum_{m=1}^j \sum_{\eta\in\cX_{j-m}}\sum_{\xi\in\cX_j}
\frac{2^{-m(K+(d-1)/2-\delta/2)p}|h_\eta|^p}{(1+N_\eta\rho(\xi,\eta))^{d-1+\delta}}\Big)\\
\le 2^{p-1} {c^\star_1}^{p-1} C_{12}^p\Big(
\sum_{m=0}^\infty {c^\star_1} 2^{-m(K+(d-1)(1/2-1/p'_*)-\delta/2)p}\sum_{\eta\in\cX_{j+m}}|h_\eta|^p\\
+\sum_{m=1}^j {c^\star_1} 2^{-m(K+(d-1)(1/2-1/p)-\delta/2)p}\sum_{\eta\in\cX_{j-m}}|h_\eta|^p\Big).
\end{multline}
Now, we raise both sides of \eqref{eq:besov_5} to the power $1/p<1$
and apply the $1/p$-inequality to its right-hand side
to obtain \eqref{eq:besov_1} for $1<p<\infty$ with $C_{11}\ge C_{13}:=2^{1-1/p} c^\star_1 C_{12}$.

Let $0<p\le 1$. Using the $p$-inequality,
observing that \eqref{cond-KM} implies in this case $Mp\ge d-1+\delta p$,
and using Lemma~\ref{lem:Besov_1} with $\beta=\delta p\le 1$ we obtain
\begin{multline}\label{eq:besov_7}
\sum_{\xi\in\cX_j}\Big|\sum_{\eta\in\cX}\omega_{\xi,\eta}h_\eta\Big|^p
\le \sum_{m=0}^\infty {c^\star_1} 2^{-mp\delta/2}2^{-m(K+(d-1)/2-\delta/2)p}\sum_{\eta\in\cX_{j+m}}|h_\eta|^p\\
+\sum_{m=1}^j {c^\star_1} 2^{-mp\delta/2}2^{-m(K+(d-1)(1/2-1/p)-\delta/2)p}\sum_{\eta\in\cX_{j-m}}|h_\eta|^p.
\end{multline}
We now raise both sides of \eqref{eq:besov_7} to the power $1/p\ge 1$,
use the convexity of $u^{1/p}$ to break the right-hand side to two terms
and apply H\"older's inequality with exponents $r=1/(1-p)$ and $r'=1/p$ in the summations on $m$
in order to get the $1/p$ power inside the sum and to prove \eqref{eq:besov_1}
for $0<p\le 1$ with $C_{11}:=\max\{C_{13}, C_{15}\}$,
where $C_{15}=2^{1/p-1}{c^\star_1}^{1/p} C_{14}$ with $c^\star_1$ is for $\beta=\delta p$
and $C_{14}:=\big(1- 2^{-\delta p/(2(1-p))}\big)^{-(1-p)/p}$.

Finally, using \eqref{def-b-space}, \eqref{eq:besov_1}, \eqref{cond-KM}, and \lemref{lem:hardy} with $\beta=\delta/2$ and
$a_m=2^{m[s+(d-1)(1/2-1/p)]}\Big(\sum_{\eta\in\cX_{m}} |h_\eta|^p\Big)^{1/p}$ we obtain
\begin{multline}\label{eq:besov_6}
\|\Omega h\|_{\bb_p^{sq}} = \Big(\sum_{j=0}^\infty \Big[2^{j[s+(d-1)(1/2-1/p)]}
\Big(\sum_{\xi \in \cX_j} \Big|\sum_{\eta\in\cX}\omega_{\xi,\eta}h_\eta\Big|^p\Big)^{1/p}\Big]^q\Big)^{1/q}\\
\le C_{11}\Big(\sum_{j=0}^\infty \Big[
\sum_{m=0}^\infty 2^{j[s+(d-1)(1/2-1/p)]-m(K+(d-1)(1/2-1/p'_*)-\delta/2)}\Big(\sum_{\eta\in\cX_{j+m}} |h_\eta|^p\Big)^{1/p}\\
+\sum_{m=1}^j 2^{j[s+(d-1)(1/2-1/p)]-m(K+(d-1)(1/2-1/p)-\delta/2)}\Big(\sum_{\eta\in\cX_{j-m}} |h_\eta|^p\Big)^{1/p}
\Big]^q\Big)^{1/q}\\
= \!C_{11}\Big(\sum_{j=0}^\infty \!\Big[\!
\sum_{m=0}^\infty \!2^{-m(K+s+(d-1)(1/p_*-1/p)-\delta/2)}a_{j+m}
+\sum_{m=1}^j 2^{-m(K-s-\delta/2)}a_{j-m}
\Big]^q\Big)^{1/q}\\
\le C_{11}\Big(\sum_{j=0}^\infty \Big[
\sum_{m\ge j} 2^{-(m-j)\delta/2}a_{m}
+\sum_{m=0}^{j-1} 2^{-(j-m)\delta/2}a_{m}
\Big]^q\Big)^{1/q}\\
\le 2^{1/q+1}C_{11}c^\star_3\Big(\sum_{m=0}^\infty a_m^q\Big)^{1/q}=C_9\|h\|_{\bb_p^{sq}}.
\end{multline}
Thus, \eqref{B-AlmDiag} is established with a constant $C_9=2^{1/q+1}C_{11}c^\star_3$ of the claimed form.

We next prove the result for the {\em sequence Triebel-Lizorkin spaces} $\ff^{sq}_p$, that is,
\begin{equation}\label{F-AlmDiag}
\|\Omega h\|_{\ff^{sq}_p}\le C_9\|h\|_{\ff^{sq}_p}.
\end{equation}
Taking into account Remark~\ref{rem:TL} we chose the quasi-norm of $\ff^{sq}_p$ in Definition~\ref{def:TL} to be defined with
$\nB_\xi=B(\xi,\gamma 2^{-j})$ for $\xi\in\cX_j$, $j\ge 1$. Thus $\nB_\xi\cap \nB_\eta=\emptyset$ for $\xi\ne\eta\in\cX_j$.

Let $h\in \ff^{sq}_p$. Then
$
(\Omega h)_{\xi}=\sum_{\eta\in\cX} \omega_{\xi,\eta}h_{\eta},
$
where the series converges absolutely (see proof below).
Then by \eqref{sph_cap}
\begin{multline}\label{F_AD1}
\|\Omega h\|_{\ff^{sq}_p}
:= \Big\|\Big(\sum_{\xi\in\cX}
\big[|\nB_\xi|^{-s/(d-1)-1/2}|(\Omega h)_\xi|\ONE_{\nB_\xi}\big]^q\Big)^{1/q}\Big\|_{L^p}
\\
\le C_{21}\Big\|\Big(\sum_{\xi\in\cX}\big[N_\xi^{s+(d-1)/2}
\sum_{\eta\in\cX} \omega_{\xi,\eta}|h_{\eta}|\ONE_{\nB_\xi}\big]^q\Big)^{1/q}\Big\|_{L^p}
\le C_{21} 2^{1/p+1/q} (\Sigma_1+\Sigma_2),
\end{multline}
where
$C_{21}:={\tc_2}^{~|s/(d-1)+1/2|}(2/\gamma)^{s+(d-1)/2}$,
\begin{align*}
&\Sigma_1:=\Big\|\Big(\sum_{\xi\in\cX}\Big[N_\xi^{s+(d-1)/2}
\sum_{\eta\in\cX: N_\eta \ge N_\xi}
\omega_{\xi,\eta}|h_{\eta}| \ONE_{\nB_\xi}\Big]^q\Big)^{1/q}\Big\|_{L^p},
\quad
\mbox{and}\\
\quad
&\Sigma_2:=\Big\|\Big(\sum_{\xi\in\cX}\Big[N_\xi^{s+(d-1)/2}
\sum_{\eta\in\cX: N_\eta < N_\xi}
\omega_{\xi,\eta}|h_{\eta}| \ONE_{\nB_\xi}\Big]^q\Big)^{1/q}\Big\|_{L^p}.
\end{align*}
Write $\lambda_\xi:= N_\xi^{s+(d-1)/2}\ONE_{\nB_\xi}$, $\xi\in\cX$,
and choose $t$ so that $(d-1)/t=\cJ +\delta/2$.
Then $0<t<\min\{1, p, q\}$.
If $N_\eta \ge N_\xi$, then
$$
\omega_{\xi,\eta} = \Big(\frac{N_\xi}{N_\eta}\Big)^{K+(d-1)/2}
\big( 1+ N_\xi\rho(\xi, \eta)\big)^{-M}.
$$
Then we have
\begin{align*}
&\Sigma_1 \le \Big\|\Big(\sum_{\xi\in\cX}
\Big[\sum_{\eta\in\cX: N_\eta \ge N_\xi}
\Big(\frac{N_\xi}{N_\eta}\Big)^{\cJ-s-(d-1)/2+\delta} (1+N_\xi\rho(\xi, \eta))^{-\cJ-\delta}
|h_{\eta}| \lambda_{\xi}\Big]^q\Big)^{\frac{1}{q}}\Big\|_{L^p}\\
&=\Big\|\Big(\sum_{j\ge 0}\sum_{\xi\in\cX_j}
\Big[\sum_{m\ge j}  2^{-(m-j)(\cJ-s-\frac{d-1}{2}+ \delta)}\sum_{\eta\in\cX_m}
\big(1+N_\xi \rho(\xi,\eta)\big)^{-\cJ-\delta}|h_{\eta}| \lambda_{\xi}\Big]^{q}\Big)^{\frac 1q}\Big\|_{L^p}.
\end{align*}
We now apply Lemma~\ref{lem:mrax} (with $(d-1)/t$ and $\delta/2$ in the place of $M$ and $\delta$) and
the fact that the sets $\{B_\xi : \xi\in\cX_j\}$ are mutually disjoint to obtain
\begin{align*}
\Sigma_1 &\le c^\star_2 \Big\|\Big(\sum_{j\ge 0}\sum_{\xi\in\cX_j}
\Big[\sum_{m\ge j}  2^{-(m-j)(\cJ-s-\frac{d-1}{2}+ \delta - \frac{d-1}{t})}
\MM_t\Big(\sum_{\eta\in\cX_m}|h_{\eta}|\ONE_{\nB_\eta}\Big) \lambda_{\xi}\Big]^{q}\Big)^{\frac 1q}\Big\|_{L^p}
 \\
&\le c^\star_2 \Big\|\Big(\sum_{j\ge 0}\Big[\sum_{m\ge j}  2^{-(m-j)\delta/2}
\MM_t\Big(\sum_{\eta\in\cX_m}|h_{\eta}|\lambda_{\eta}\Big)\Big]^{q}\Big)^{\frac 1q}\Big\|_{L^p}.
\end{align*}
The application of inequality (\ref{hardy-1}), the maximal inequality (\ref{max-ineq}),
the fact that the sets $\{B_\eta : \eta\in\cX_j\}$ are mutually disjoint,
 and \eqref{sph_cap} leads to
\begin{align}
\Sigma_1
&\le c^\star_2 c^\star_3\Big\|\Big(\sum_{j\ge 0}
\Big[\MM_t\Big(\sum_{\eta\in\cX_j}|h_{\eta}|\lambda_{\eta}\Big)\Big]^{q}\Big)^{\frac 1q}\Big\|_{L^p} \label{F_AD2}
\\
&\le c^\star_2 c^\star_3 \tc_1\Big\|\Big(\sum_{j\ge 0}
\Big[\sum_{\eta\in\cX_j}N_\eta^{s+(d-1)/2}|h_{\eta}|\ONE_{\nB_\eta}\Big]^{q}\Big)^{\frac 1q}\Big\|_{L^p} \nonumber
\\
&\le c^\star_2 c^\star_3 \tc_1 C_{22}\Big\|\Big(\sum_{j\ge 0}
\sum_{\eta\in\cX_j}\Big[|\nB_\eta|^{-s/(d-1)-1/2}|h_{\eta}|\ONE_{\nB_\eta}\Big]^{q}\Big)^{\frac 1q}\Big\|_{L^p}
= c\|f\|_{\ff^{sq}_p} \nonumber
\end{align}
with  $C_{22}:=\tc_2^{~|s/(d-1)+1/2|}(\gamma/2)^{s+(d-1)/2}$.

If $N_\eta < N_\xi$, then
$$
\omega_{\xi,\eta} = \Big(\frac{N_\eta}{N_\xi}\Big)^{K+(d-1)/2}
\big(1+ N_\eta\rho(\xi, \eta)\big)^{-M}
$$
and hence
\begin{align*}
&\Sigma_2 \le \Big\|\Big(\sum_{\xi\in\cX}
\Big[\sum_{\eta\in\cX: N_\eta < N_\xi}
\Big(\frac{N_\eta}{N_\xi}\Big)^{s+(d-1)/2+\delta} (1+N_\eta\rho(\xi, \eta))^{-\cJ-\delta}
|h_{\eta}| \lambda_{\xi}\Big]^q\Big)^{\frac{1}{q}}\Big\|_{L^p}\\
&= \Big\|\Big(\sum_{j\ge 0}\sum_{\xi\in\cX_j}
\Big[\sum_{m< j}  2^{-(j-m)(s+\frac{d-1}{2}+ \delta)}\sum_{\eta\in\cX_m}
\big(1+N_\eta \rho(\xi,\eta)\big)^{-\cJ-\delta}|h_{\eta}| \lambda_{\xi}\Big]^{q}\Big)^{\frac 1q}\Big\|_{L^p}.
\end{align*}
As above employing Lemma~\ref{lem:mrax}, using the fact that the sets $\{B_\xi : \xi\in\cX_j\}$ are mutually disjoint,
applying \eqref{hardy-2}, the maximal inequality \eqref{max-ineq}, and \eqref{sph_cap}
we obtain
\begin{align}
\Sigma_2
&\le c^\star_2\Big\|\Big(\sum_{j\ge 0}\Big[\sum_{m< j}  2^{-(j-m)\delta}
\MM_t\Big(\sum_{\eta\in\cX_m}|h_{\eta}|\lambda_{\eta}\Big)\Big]^{q}\Big)^{\frac 1q}\Big\|_{L^p} \label{F_AD3}
\\
&\le c^\star_2 c^\star_3 \Big\|\Big(\sum_{j\ge 0}
\Big[\MM_t\Big(\sum_{\eta\in\cX_j}|h_{\eta}|\lambda_{\eta}\Big)\Big]^{q}\Big)^{\frac 1q}\Big\|_{L^p} \nonumber
\\
&\le c^\star_2 c^\star_3 \tc_1\Big\|\Big(\sum_{j\ge 0}
\Big[\sum_{\eta\in\cX_j}N_\eta^{s+(d-1)/2}|h_{\eta}|\ONE_{\nB_\eta}\Big]^{q}\Big)^{\frac 1q}\Big\|_{L^p} \nonumber
\\
&\le c^\star_2 c^\star_3 \tc_1 C_{22}\Big\|\Big(\sum_{j\ge 0}
\sum_{\eta\in\cX_j}\Big[|\nB_\eta|^{-s/(d-1)-1/2}|h_{\eta}|\ONE_{\nB_\eta}\Big]^{q}\Big)^{\frac 1q}\Big\|_{L^p}
=c\|f\|_{\ff^{sq}_p}, \nonumber
\end{align}
where the constants $c^\star_2, c^\star_3, \tc_1, C_{22}$ are as above.

Finally, using estimates \eqref{F_AD2} and \eqref{F_AD3} in \eqref{F_AD1}
we obtain \eqref{F-AlmDiag}
with $C_9=C_{21} 2^{1/p+1/q}c^\star_2 c^\star_3 \tc_1 C_{22}$,
which is of the claimed form.
This completes the proof of Theorem~\ref{thm:almost-diag}.

\end{proof}

\end{document}